\documentclass[11pt]{amsart}
\usepackage{graphics,graphicx} 

\usepackage[T1]{fontenc}

\usepackage[cp1250]{inputenc}

\usepackage{amsmath,amsthm,amssymb,dsfont}
\usepackage{amscd}
\usepackage[arrow, matrix, curve]{xy}

\usepackage{textcomp}

\newenvironment{proof of}[1]{\emph{Proof of #1.}}{\hfill $\qquad \square$\par}

\newcommand{\CP}{\textrm{CP}}
\newcommand{\Pos}{\textrm{Pos}}

\DeclareMathOperator{\RM}{\mathcal{RM}}
\DeclareMathOperator{\LM}{\mathcal{LM}}

\DeclareMathOperator{\M}{\mathcal{M}}
\DeclareMathOperator{\id}{id}

\DeclareMathOperator{\Aut}{Aut}

\DeclareMathOperator{\clsp}{\overline{span}}

\DeclareMathOperator{\DR}{\mathcal{DR}}
\DeclareMathOperator{\FR}{\mathcal{FR}}
\DeclareMathOperator{\End}{End}
\DeclareMathOperator{\CE}{\bold{E}}

\newcommand{\HH}{\mathcal H}

\newcommand{\K}{\mathcal K}

\newcommand{\VV}{\mathcal V}
\newcommand{\F}{\mathbb F}

\newcommand{\LL}{\mathcal{L}}
\newcommand{\LLL}{\mathbb{L}}
\newcommand{\KK}{\mathcal K}
\newcommand{\al}{\alpha}

\newcommand{\FF}{\mathcal F}

\newcommand{\OO}{\mathcal O}
\newcommand{\A}{\mathcal{NT}(X)}
\renewcommand{\P}{\mathcal{P}}

\newcommand{\B}{\mathcal B}

\newcommand{\C}{\mathbb C}

\newcommand{\N}{\mathbb N}

\newcommand{\TT}{\mathcal T}
\renewcommand{\LL}{\mathcal L}
\newcommand{\NT}{\mathcal{NT}}
\newcommand{\NN}{\mathcal{N}}

\numberwithin{equation}{section}

\newtheorem{thm}{Theorem}[section]\newtheorem{lem}[thm]{Lemma}
\newtheorem{prop}[thm]{Proposition}
\newtheorem{cor}[thm]{Corollary}
\theoremstyle{definition}
\newtheorem{defn}[thm]{Definition}
\newtheorem{ex}[thm]{Example}
\newtheorem{assump}[thm]{Standing assumptions}
\newtheorem{rem}[thm]{Remark}

\title[Nica-Toeplitz algebras over right LCM semigroups]{Nica-Toeplitz algebras associated with\\
 product systems over right LCM semigroups}

\author{Bartosz K.  Kwa\'sniewski}
 \email{bartoszk@math.uwb.edu.pl}
 \address{Institute of Mathematics,  University  of Bia\l ystok,
ul. K. Cio\l kowskiego 1M, 15-245 Bia\-\l ystok, Poland}

\author{Nadia S.  Larsen}
\email{nadiasl@math.uio.no}
\address{Department of Mathematics, University of Oslo, PO Box 1053 Blindern, 0316 Oslo, Norway }

\date{15 June 2017. Revised 17 November 2017 and 18 September 2018.}
\begin{document}

\begin{abstract}
  We prove uniqueness of representations of Nica-Toeplitz algebras associated to product systems of $C^*$-correspondences over right LCM semigroups by applying our previous abstract uniqueness results developed for $C^*$-precategories. Our results provide an interpretation of  conditions identified in work of Fowler and Fowler-Raeburn, and apply also to their crossed product twisted by a product system, in the new context of right LCM semigroups, as well as to a new, Doplicher-Roberts type $C^*$-algebra associated to the Nica-Toeplitz algebra. As a derived construction we develop Nica-Toeplitz crossed products by actions with completely positive maps. This provides a unified framework for Nica-Toeplitz semigroup crossed products by endomorphisms and by transfer operators. We illustrate these two classes of examples with semigroup $C^*$-algebras of right and left semidirect products.
\end{abstract}

\maketitle

 \setcounter{tocdepth}{1}

\section*{Introduction}

Product systems of $C^*$-correspondences were introduced by Fowler following ideas of Arveson. As spelled out in \cite{F99},  Fowler's construction  served as motivation for his investigation with Raeburn into uniqueness theorems for  $C^*$-algebras arising as certain twisted crossed products over positive cones in quasi-lattice ordered groups, \cite{FR}. Three uniqueness theorems in this context have dominated the attention: \cite[Theorem 2.1]{Fow-Rae}, \cite[Theorem 7.2]{F99} and \cite[Theorem 5.1]{FR}. All three give a necessary condition for faithfulness of a representation $\pi$ of a Toeplitz-type $C^*$-algebra $\TT_X$ or $\NT(X)$, where $X$ is  generic notation for a single $C^*$-correspondence or a product system of such over a semigroup, and $\pi$ arises from a representation of $X$.

Two aspects are striking here: the first is that the necessary condition, to which we choose to refer as
 \emph{condition (C)} - for compression or for Coburn, who proved the archetypical result of this form - is only sufficient when  the left action in each correspondence is by generalized compacts. The second is that, in the aforementioned results on product systems, an auxiliary $C^*$-algebra is involved. It has the structural appearance of a crossed product twisted by a product system and  a somewhat  unaccountable involvement in the uniqueness of representations of  $\NT(X)$.

 The first main point we make in the present paper is that there is another $C^*$-algebra for which uniqueness of representations coming from $X$ is precisely encoded, as a \emph{necessary and sufficient condition}, by condition (C). This $C^*$-algebra, which we generically denote $\mathcal{DR}(\NT(X))$,  bears the flavor of a Doplicher-Roberts algebra for $\NT(X)$. The second point we make is that uniqueness of a representation $\pi$ of $\NT(X)$ can, in good situations, be precisely encoded by a weaker condition than (C) which we call \emph{Toeplitz covariance}. The third point we make is that the strategy for proving these results relies on our previous work on $C^*$-precategories developed in \cite{kwa-larI}, and as a very satisfactory bonus provides a clear picture of how $\NT(X)$, the Fowler-Raeburn crossed product twisted by a product system and $\mathcal{DR}(\NT(X))$ are included in each other, respectively, and how uniqueness of representations on $\mathcal{DR}(\NT(X))$ sieves down to corresponding results on the smaller subalgebras.  Together, these three uniqueness results offer a different picture of endeavors by many hands over several decades.
In addition we extend these results beyond the scope of
quasi-lattice ordered pairs, which is non-trivial as LCM semigroups allow invertible elements and might not be cancellative, cf. \cite{kwa-larI}.

 As an application of our uniqueness  results we define a Nica-Toeplitz crossed product for a dynamical system involving a semigroup action of completely positive maps on a $C^*$-algebra. For a single completely positive map, a similar construction was proposed by the first named author in \cite{kwa-exel}. In this new setup of  semigroup actions by completely positive maps our construction models Toeplitz-type crossed products in two important contexts: actions by endomorphisms, see e.g. \cite{F99} (where the assumptions on the acting semigroup and the conventions on covariance are different), and actions by transfer operators, see e.g.  \cite{Larsen} where the acting semigroup is abelian. We formulate uniqueness theorems for our crossed products, and illustrate the two classes of actions with semigroup $C^*$-algebras as in Li \cite{Li}, through the perspective of algebraic dynamical systems developed by the second named author in collaboration with Brownlowe and  Stammeier \cite{bls2}. The left-semidirect product  semigroup $C^*$-algebras coming from \cite{bls2} will serve to motivate crossed products by transfer operators, and, somehow  unexpectedly though in hindsight not that surprisingly, right semidirect product semigroup $C^*$-algebras will motivate crossed products by endomorphisms.

 The paper is organized as follows. In a preliminaries section we review briefly the basics of $C^*$-correspondences and product systems of these, after which we collect the main ingredients needed about $C^*$-precategories and their $C^*$-algebras from \cite{kwa-larI}. In section \ref{technical subsection} we associate a $C^*$-precategory to a product system $X$ of $C^*$-correspondences over a right LCM semigroup $P$ and in  section \ref{product systems over LCMs} we use it to construct a Doplicher-Roberts version $\DR(\NT(X))$ and a reduced version of $\NT(X)$. We introduce conditions under which representations of $X$ give rise to faithful representations of the core subalgebras of $\NT(X)$ and  $\DR(\NT(X))$. In subsection \ref{subsect:uniqueness theorems} we prove uniqueness results for $\NT(X)$ and  $\DR(\NT(X))$ and discuss some implications. In subsection \ref{Fowler-Raeburn section} we extend this discussion   by introducing   $C^*$-algebras $\FR(X)$   generalizing semigroup $C^*$-algebras twisted by product systems studied by Fowler and Raeburn, see \cite{FR}, \cite{F99}. In section~\ref{section:NT-cp-ccp-maps}  we introduce Nica-Toeplitz crossed products of a $C^*$-algebra by the action of a right LCM semigroup of completely positive maps, and prove uniqueness results for the two major types of examples, crossed products by endomorphisms and by transfer operators. Finally, in section~\ref{section:semigroupCstar alg} we show that the  Nica-Toeplitz crossed products by endomorphisms  and by transfer operators  can be perfectly embodied by semigroup $C^*$-algebras associated to right and left semidirect products of semigroups, respectively. By specializing the uniqueness results to these contexts we generalize and complement  earlier results from \cite{LR} and \cite{bls2}.

While we were laying the last hand on this article, another paper appeared  \cite{Fle} in which Fletcher takes on clarifying the uniqueness result \cite[Theorem 7.2]{F99} for $\NT(X)$ in the context of quasi-lattice ordered pairs.

\subsection{Acknowledgements}

The research leading to these
results has received funding from the European Union's
Seventh Framework Programme (FP7/2007-2013) under grant agreement number 621724. B.K. was  partially supported by the NCN (National Centre of Science) grant number 2014/14/E/ST1/00525. Part of the work was carried out when B.K.  participated in the Simons Semester at IMPAN - Fundation grant 346300 and the Polish Government MNiSW 2015-2019 matching fund,  the participation of both  authors in the program "Classification of operator algebras: complexity, rigidity, and dynamics" at the Mittag-Leffler Institute (Sweden) in January 2016.

\section{Preliminaries}

\subsection{LCM semigroups} We refer to \cite{BRRW} and \cite{bls} and the references therein for basic facts about right LCM semigroups. All semigroups considered in this paper will have an identity $e$.  We let $P^*$ be the group of \emph{units}, or invertible elements, in $P$. A \emph{principal right ideal} in $P$ is a right ideal in $P$ of the form $pP=\{ps: s\in P\}$ for some $p\in P$.
The relation of inclusion on the principal right ideals induces a left invariant \emph{preorder} on $P$ given by $p \leq q$ when $qP\subseteq  pP$. Clearly $\leq$ is a partial order if and only if $P^*=\{e\}$.

 A semigroup $P$ is a \emph{right LCM semigroup} if  the family $\{pP\}_{p\in P}$ of principal right ideals extended by the empty set  is closed under intersections, that is if for every pair of elements
$p, q\in P$ we have $pP\cap qP=\emptyset$ or $pP\cap qP=rP$ for some $r\in P$.
In the case that $pP\cap qP=rP$, the element $r$ is a \emph{right least common multiple (LCM)}  of $p$ and $q$.
 If $P$ is a right LCM semigroup then  we  refer to
$
J(P):=\{pP\}_{p\in P}\cup\{\emptyset\}
$ as the \emph{semilattice of principal right ideals} of $P$. Right LCMs in a right LCM semigroup are determined  up to  multiplication from the right by an invertible element. Namely, if  $pP\cap qP=rP$, then  $pP\cap qP=tP$ if and only if there is $h\in P^*$ such that $t=rh$.

\begin{ex} \textnormal{(a)} One of the most known and studied examples of right LCM semigroups are positive cones in quasi-lattice ordered groups, introduced by Nica \cite{N}. In fact, $P$ is  a  positive cone   in a weakly quasi-lattice ordered group $(G,P)$ if and only if $P$ is an  LCM subsemigroup of a group $G$ such that $P^*=\{e\}$.

\textnormal{(b)} In semigroup theory, notions similar to right LCM have been known for some time, see e.g. \cite{Law2}. New large classes of right LCM semigroups with relevance to $C^*$-algebraic context were identified in \cite{BRRW}.  Semidirect product semigroups which are right LCM semigroups were studied in \cite{bls, bls2}. More on this in section~\ref{section:semigroupCstar alg}.
\end{ex}

We recall from \cite[Definition 2.4]{kwa-larI} that a \emph{controlled map of right LCM semigroups} is an identity preserving homomorphism $\theta:P\to \P$ between right LCM semigroups $P,\P$ such that $\theta(P^*)=\P^*$ and for all $s,t\in P$ with $sP\cap tP =rP$ we have
$
\theta(s)\P\cap \theta(t)\P = \theta(r)\P
$
and
$
\theta(s)=\theta(t)$ only if   $s=t$.

\begin{ex}
Let $P_i$, $i\in I$, be a family of right LCM semigroups. Put $P:=\prod^\ast_{i\in I} P_i$  and $\P:=\bigoplus_{i\in I} P_i$. The homomorphism $\theta:P\to \P$  which is the identity on each $P_i$, $i\in I$, is a controlled map of right LCM semigroups, by \cite[Proposition 2.3]{kwa-larI}.
\end{ex}
\subsection{$C^*$-correspondences and product systems}

The notion of a $C^*$-correspondence $X$ over a $C^*$-algebra $A$ and its associated Toeplitz algebra $\TT(X)$ are standard, and we refer to  \cite{Pim, KPW, Fow-Rae} for details. We recall from 	\cite{F99}\label{Fowler's definition} that a \emph{product system} over a semigroup $P$ with coefficients in a $C^*$-algebra $A$ is a semigroup $X= \bigsqcup_{p\in P}X_{p}$, with each $X_p$ a $C^*$-correspondence over $A$,
equipped with a semigroup homomorphism $d\colon X \to P$ such that $X_p = d^{-1}(p)$ is a $C^*$-correspondence over $A$ for each $p\in P$,  $X_e$ is the standard bimodule $_AA_A$, and the multiplication on $X$ extends to isomorphisms $X_p \otimes_A X_q \cong X_{pq}$
for $p,q \in P \setminus \{e\}$ and coincides with the right and left actions of $X_e = A$ on each $X_p$.
For each $p\in P$ we write $\langle\cdot,\cdot\rangle$ for the $A$-valued
inner product on $X_p$  and we denote $\phi_p$ the homomorphism from $A$ into $\LL(X_p)$
 which implements the left action of $A$ on $X_p$.

A \emph{Hilbert $A$-bimodule} is a $C^*$-correspondence which is also a left Hilbert module such that
  \({}_A\langle x,y\rangle\cdot z = x\cdot\langle y, z\rangle_A\)
  for all \(x,y,z\in X\). An \emph{equivalence $A$-bimodule} is a Hilbert bimodule which is full over left and right.
We say that two Hilbert $A$-bimodules $X$, $Y$ are Morita equivalent if there is an equivalence $A$-bimodule $E$	such that $X\otimes_A E\cong Y\otimes_A E$.

\begin{rem}\label{rem:on essentiality and Fell bundles} A product system $X$ is (left) \emph{essential} if each $C^*$-correspondence $X_p$, $p\in P$, is essential. We claim that $X$ is automatically  essential whenever the group $P^*$ of units in $P$  is non-trivial. Indeed, for any $h\in P^*\setminus\{e\}$ and $p\in P$ we have natural isomorphisms
$$
X_p= X_{hh^{-1}p}\cong X_h \otimes_A X_{h^-1}\otimes_A X_p \cong X_e \otimes_A  X_p \cong \phi_p(A)X_p
$$
that give $X_p=\phi_p(A)X_p$. Moreover,  isomorphisms $X_h \otimes_A X_{h^{-1}}\cong A_A$ and $X_{h^{-1}} \otimes_A X_{h}\cong A_A$ imply that $X_h$ and $X_{h^{-1}}$ are mutually adjoint Hilbert bimodules, i.e. there is  an  antilinear isometric bijection $\flat_h:X_{h}\to X_{h^{-1}}$ such that $\flat_h(ab)=\flat_h(b)a$ and $\flat_h(ba)=a\flat_h(b)$ for all $a\in A$ and $b\in X_h$, cf. \cite[Remark 6.2]{bls2}. In particular, the family  of Banach spaces $\{X_{h}\}_{h\in P^*}$ together with multiplication inherited from $X$ and involution defined by
$b^*:=   \flat_h(b)$, for $b\in X_{h}$, $h\in P^*$, is a saturated Fell bundle over the (discrete) group $P^*$, cf. \cite{exel-book}.

\end{rem}
 Given a product system $X$ and $p, q \in P$ with $p \not= e$, there is a homomorphism
$\iota^{pq}_p \colon \LL(X_p) \to \LL(X_{pq})$ characterized by
\begin{equation}\label{iotapq def}
\iota^{pq}_p(S)(xy) = (Sx)y\text{ for all $x \in X_p$, $y \in
X_{q}$ and $S \in \LL(X_p)$.}
\end{equation}
For each $p\in P$, $\K(A,X_p)$ is a $C^*$-correspondence with $A$-valued inner product $\langle T,S\rangle_A=T^*S$
and pointwise  actions. In fact, see \cite[Lemma 2.32]{RaeWill}, there is a $C^*$-correspondence isomorphism  $X_p\cong \K(A,X_p)$ implemented by the  map
\begin{equation}\label{C-correspondence isomorphism}
X_p\ni x\mapsto t_x \in \K(A,X_p) \qquad \textrm{ where } t_x(a)=x\cdot a.
\end{equation}
One defines $\iota^p_e \colon \K(X_e)\to
\LL(X_{p})$ by letting $\iota^p_e(t_a)=\phi_p(a)$ for  $p\in P$, $a\in A$, see \cite[Section 2.2]{SY}.

A \emph{representation of the product system} $X$ in a $C^*$-algebra $B$ is a semigroup homomorphism $\psi:X\to B$, where $B$ is viewed as a semigroup with multiplication, such that
$(\psi_e,\psi_p)$  is a  representation of the $C^*$-correspondence   $X_p$,  for all  $p\in P$,
where we put $\psi_p:=\psi|_{X_p}$ for all $p\in P$. The Toeplitz algebra $\TT(X)$ is the $C^*$-algebra generated by a  universal  representation of $X$.

In the case of a quasi-lattice ordered pair $(G, P)$, Fowler introduced in \cite{F99} the notions of compactly aligned product system over $P$ and Nica covariant representation of it. In \cite{bls2}, these concepts were extended to the case when $P$ is a right LCM semigroup. Given a right LCM  semigroup $P$, a product system
$X$ over $P$ is called \emph{compactly aligned} if for all $p,q\in P$ such that there is a right LCM $r$ for $p,q$, then  $\iota^{r}_p(S) \iota^{r}_q(T) \in \KK(X_{r})$
whenever $S \in \KK(X_p)$ and $T \in \KK(X_q)$. Assume $X$ is a compactly aligned product system over $P$ and let $\psi$ be a  representation of $X$ in a $C^*$-algebra. For each $p\in P$, denote  $\psi^{(p)}$ the Pimsner $*$-homomorphism defined on $\KK(X_p)$ by $\psi^{(p)}(\Theta_{x,y})=
\psi_p(x)\psi_p(y)^*$ for $x,y\in X_p$. Then $\psi$ is \emph{Nica covariant} if
$$
\displaystyle \psi^{(p)}(S)\psi^{(q)}(T) =
\begin{cases}
\psi^{(r)}\big(\iota^{r}_p(S)\iota^{r}_q(T)\big)
& \text{if $pP\cap qP=rP$} \\
0 &\text{otherwise}
\end{cases}
$$
for all $S \in \KK(X_p)$ and $T \in \KK(X_q)$ (see also
\cite[Definition 5.7]{F99}). The \emph{Nica Toeplitz algebra}  $\NT(X)$ is the  $C^*$-algebra generated by a Nica covariant representation $i_X$ which is  universal in the following sense: if
$\psi$ is a Nica covariant Toeplitz representation of $X$ in
$B$ there is a $*$-homomorphism $\psi_* : \NT(X) \to B$ such
that $\psi_*\circ i_X=\psi.$

\subsection{$C^*$-precategories}

$C^*$-precategories should be regarded as  non-unital versions of $C^*$-categories, cf. \cite{glr}, \cite{dr}. We give here a very brief account, for more details and background material on $C^*$-precategories, see \cite{kwa-doplicher}, \cite{kwa-larI}.

Recall that a \emph{$C^*$-precategory} $\LL$  with object set  $P$ is identified with
a collection of Banach spaces  $\{\LL(p,q)\}_{p,q\in P}$, viewed as morphisms, equipped  with bilinear maps, viewed as composition of morphisms,
$
\LL(p,q)\times \LL(q,r)\ni(a, b)\mapsto ab\in \LL(p,r)$, $p,q,r\in P,
$
satisfying $\|ab\|\leq \|a\|\cdot \|b\|$, and  an antilinear involutive contravariant mapping  $^*:\LL\to\LL$ such that if $a \in   \LL(p,q)$, then $a^* \in   \LL(q,p)$ and the $C^*$-equality $\|a^* a\|=\|a\|^2$ holds. In particular, $\LL(p,p)$ is naturally  a $C^*$-algebra, and  we require that for every $a \in   \LL(q,p)$ the element $a^*a$ is positive in the $C^*$-algebra $\LL(p,p)$.

 An \emph{ideal in a $C^*$-precategory} $\LL$ is a collection $\KK=\{\KK(p,q)\}_{p,q\in P}$  of closed linear subspaces $\KK(p,q)$ of $\LL(p,q)$,  $ p,q \in P$,  such that
$$
  \LL(p,q)\KK(q,r) \subseteq \KK(p,r)\quad \textrm{ and } \quad \KK(p,q)\LL(q,r)\, \subseteq \KK(p,r),
  $$
  for all $p,q,r\in P$. Then $\KK$  is automatically selfadjoint and hence  a $C^*$-precategory. An ideal $\KK$  in $\LL$ is uniquely determined by the $C^*$-algebras  $\{\KK(p,p)\}_{p\in P}$, which are in fact ideals in the corresponding $C^*$-algebras $\{\LL(p,p)\}_{p\in P}$. We say that   $\KK$ is an \emph{essential ideal}  in $\LL$  if $\KK(p,p)$ is an essential ideal in  $\LL(p,p)$, for every $p\in P$.

A \emph{representation $\Psi:\LL \to B$ of a $C^*$-precategory} $\LL$ in
a $C^*$-algebra $B$ is a family  $\Psi=\{ \Psi_{p,q}\}_{p,q\in P}$
of linear operators  $\Psi_{p,q}:\LL(p,q)\to B$ such that
$$
\Psi_{p,q}(a)^*=\Psi_{q,p}(a^*), \quad \textrm{ and } \quad  \Psi_{p,r}(ab)=\Psi_{p,q}(a)\Psi_{q,r}(b),
$$
for  all $a\in \LL(p,q)$, $b\in \LL(q,r)$. Then automatically all the maps $\Psi_{p,q}$, $p,q\in P$, are contractions, and they all are isometries if and only if all the maps $\Psi_{p,p}$, $p\in P$, are injective. In the latter case we say that  $\Psi$ is \emph{injective}.
We denote by  $C^*(\Psi(\LL))$ the $C^*$-algebra  generated by the spaces $\Psi(\LL(p,q))$, $p,q\in P$.
A \emph{representation} $\Psi$ of $\LL$ \emph{on a Hilbert space} $H$  is a representation of  $\LL$ in the $C^*$-algebra $\B(H)$ of all bounded operators on $H$. If in addition  $C^*(\Psi(\LL))H=H$ we say that the representation $\Psi$ is \emph{nondegenerate}.

If $\KK$ is an ideal in a $C^*$-precategory $\LL$ and  $\Psi=\{ \Psi_{p,q}\}_{p,q\in P}$ is  a representation of $\KK$ on a Hilbert space $H$, then there is  a unique extension $\overline{\Psi}=\{\overline{\Psi}_{p,q}\}_{p,q\in P}$ of  $\Psi$  to a representation  of  $\LL$   such that the essential subspace of $\overline{\Psi}_{p,q}$ is  contained in the essential subspace of $\Psi_{p,q}$, for every $p,q\in P$.  Namely, we have
\begin{equation}\label{formula defining extensions of right tensor representations}
\overline{\Psi}_{p,q}(a)(\KK(q,q) H)^\bot =0,\quad \text{ and }\quad  \overline{\Psi}_{p,q}(a) \Psi_{q,q}(b)h = \Psi_{p,q}(ab)h
\end{equation}
for all $ a\in \LL(p,q)$, $b \in \KK(q,q)$, $h \in H$.
Moreover, $\overline{\Psi}$ is injective if and only if $\Psi$ is injective and $\KK$ is an essential ideal in $\LL$.

\subsection{Right tensor $C^*$-precategories and their $C^*$-algebras}

We recall the basic definitions and facts from \cite[Section 3]{kwa-larI}.  A \emph{right-tensor $C^*$-precategory} is a  $C^*$-precategory $\LL=\{\LL(p,q)\}_{p,q\in P}$ whose objects form a  semigroup $P$ with identity $e$ and which is equipped with a semigroup $\{\otimes 1_r\}_{r\in P}$ of endomorphisms of $\LL$   sending  $p$ to $pr$, for all $p,r\in P$, and $\otimes 1_e=id$.  More precisely, we have linear operators $\LL(p,q)\in a \mapsto a\otimes 1_r \in \LL(pr,qr)$ such that for each $ a\in \LL(p,q)$, $ b\in \LL(q,s)$,  and $p,q,r, s\in P$ we have
$$
 ((a\otimes 1_r)\otimes 1_s) = a\otimes 1_{rs},
 \qquad
(a\otimes 1_r)^*=a^*\otimes 1_r,\qquad   (a \otimes 1_r)  (b\otimes 1_r)= (ab)\otimes 1_r.
$$
  We refer to $\{\otimes 1_r\}_{r\in P}$  as to a \emph{right tensoring} on $\LL=\{\LL(p,q)\}_{p,q\in P}$.
	
If $\KK$  is an ideal in a right-tensor $C^*$-precategory $\LL$, we say that $\KK$ is $\otimes 1$-\emph{invariant}, and write $\KK\otimes 1\subseteq \KK$,  if  $\KK(p,p)\otimes 1_r \subseteq \KK(pr,pr)$ for all $p,r\in P$. One can show that
 $\KK\otimes 1\subseteq \KK$ if and only if $\KK(p,q)\otimes 1_r \subseteq \KK(pr,qr)$ for all $p,q,r\in P$.	Right tensor representations and the corresponding Toeplitz
	algebras are defined for all ideals, $\otimes 1$-invariant or not, in some $\LL$.

Let $\KK$ be an ideal in a right-tensor $C^*$-precategory $\LL$. We say that a representation $\Psi:\KK\to B$ of  $\KK$  in a $C^*$-algebra $B$ is a \emph{right-tensor  representation} if for all
 $a\in \KK(p,q)$, $b\in   \KK(s,t) $ such that $sP\subseteq qP$ we have
\begin{equation}\label{right tensor representation condition}
\Psi(a)\Psi(b)
=
\Psi \left((a \otimes 1_{q^{-1}s}) b\right).
\end{equation}
Note that, since $\KK$ is an ideal,  the right hand side  of \eqref{right tensor representation condition}  makes sense. One can show there is an injective right-tensor representation $t_{\KK}: \KK \to \TT_{\LL}(\KK) $ with the universal property that
 for every right-tensor  representation $\Psi$ of $\KK$ there is a homomorphism  $\Psi\times P$
of $\TT_{\LL}(\KK)$ such that $\Psi\times P\circ t_{\KK} =\Psi$, and
 $\TT_{\LL}(\KK)=C^*(t_{\KK}(\KK))$. We call  $\TT_{\LL}(\KK)$  the  \emph{Toeplitz algebra} of $\KK$. We write $\TT(\LL)$ for the Toeplitz algebra $\TT_{\LL}(\LL)
$ associated to $\LL$, viewed as an  ideal in itself.

If the underlying semigroup is right LCM, then for  well-aligned ideals we can make sense  of a condition of Nica type, which is stronger than \eqref{right tensor representation condition}.

Let $(\LL, \{\otimes 1_r\}_{r\in P})$ be a right-tensor  $C^*$-precategory  over a right LCM semigroup $P$. An ideal $\KK$ in $\LL$ is \emph{well-aligned} in $\LL$ if   for all $a\in \KK(p,p)$, $b\in \KK(q,q) $ we have
\begin{equation}\label{compact alignment relation}
(a\otimes 1_{p^{-1}r}) (b\otimes 1_{q^{-1}r}) \in \KK(r,r)\qquad \textrm{whenever}\quad pP\cap qP=rP.
 \end{equation}
 By \cite[Lemma 3.7]{kwa-larI}, for any  ideal
$\KK$ condition \eqref{compact alignment relation} implies the formally stronger condition that for every
 $a\in \KK(p,q)$, $b\in   \KK(s,t) $ we have
\begin{equation}\label{compact alignment relation2}
(a\otimes 1_{q^{-1}r}) (b\otimes 1_{s^{-1}r}) \in \KK(pq^{-1}r,ts^{-1}r)\qquad \textrm{whenever}\quad qP\cap sP=rP.
 \end{equation}
\begin{assump}
 For every well-aligned ideal $\KK$ in $\LL$, in this paper   we will also assume the following condition
\begin{itemize}
\item $\KK$  is $\otimes 1$-\emph{nondegenerate}, \cite[Definition 9.6]{kwa-larI} that is
\begin{equation}\label{non-degeneracy condition}
(\KK(p,p)\otimes 1_{r})\KK(pr,pr)=\KK(pr,pr)\textrm{ for every } p\in P\setminus{P^*} \textrm{ and } r\in P.
\end{equation}
\item $\KK$ satisfies  condition (7.6)  in \cite[Proposition 7.6]{kwa-larI} for $t=e$, that is
\begin{equation}\label{condition for reducing Fock-reps}
\overline{\KK(p,e)\KK(e,p)}\text{ is an  essential  ideal in the $C^*$-algebra } \LL(p,p), \textrm{ for every } p\in P.
\end{equation}
\end{itemize}
These conditions will be satisfied by right-tensor $C^*$-precategories arising from product systems.
\end{assump}

\subsection{Nica-Toeplitz algebras associated with right-tensor $C^*$-precategories}
Let us fix a right-tensor $C^*$-precategory     $(\LL, \{\otimes 1_r\}_{r\in P})$   over an LCM semigroup $P$, and   a well-aligned ideal $\KK$ in $\LL$.
A representation $\Psi:\KK\to B$ of  $\KK$  in a $C^*$-algebra $B$ is \emph{Nica covariant} if for all
 $a\in \KK(p,q)$, $b\in   \KK(s,t) $ we have
\begin{equation}\label{Nica covariance}
\Psi(a)\Psi(b)
=\begin{cases}
\Psi \left((a \otimes 1_{q^{-1}r}) (b\otimes 1_{s^{-1}r})\right)  & \textrm{ if } qP\cap sP=rP \textrm{ for some } r\in P,
\\
0 & \textrm{ otherwise}.
\end{cases}
 \end{equation}
Note that by \eqref{compact alignment relation2}  the right hand side of \eqref{Nica covariance} makes sense. By \cite{kwa-larI} there is an injective   Nica covariant representation $i_{\KK}: \KK \to \NT_{\LL}(\KK) $ with the universal property:
 for every Nica covariant  representation $\Psi$ of $\KK$ there is a homomorphism  $\Psi\rtimes P$
of $\NT_{\LL}(\KK)$ such that $\Psi\rtimes P\circ i_{\KK} =\Psi$, and
 $\NT_{\LL}(\KK)=C^*(i_{\KK}(\KK))$. We call $\NT_{\LL}(\KK)$  the \emph{Nica-Toeplitz algebra} of $\KK$. We write $\NT(\LL)$ for the Nica-Toeplitz algebra $\NT_{\LL}(\LL)
$ associated to $\LL$, viewed as a well-aligned ideal in itself (in particular, in this paper we assume that $\LL$ satisfies the analogue of \eqref{condition for reducing Fock-reps}). By \eqref{non-degeneracy condition} and \cite[Lemma 11.1]{kwa-larI} we have a natural embedding
$$\NT_{\LL}(\KK)\hookrightarrow \NT(\LL).
$$

 The Fock representation  of $\KK$ constructed in  \cite{kwa-larI} is a direct sum  of Nica covariant representations. By \cite[Proposition 7.6]{kwa-larI} and  \eqref{condition for reducing Fock-reps},  here we may  use the $e$-th summand of it. We recall the relevant construction.
 For  $s\in P$, the space $X_{s}:=\KK(s,e)$ is naturally equipped with the  structure of a right Hilbert module over $A:=\KK(e,e)$ inherited from $C^*$-precategory structure of $\KK$: we put
$
 x \cdot a:=xa$, $\langle x, y\rangle:=x^*y$, for    $x,y \in X_{s}$,  $a\in A$.
Thus we may consider the  direct sum  Hilbert $A$-module:
$
\FF_{\KK}:=\bigoplus_{s\in P} X_{s}.
$
By \cite[Remark 4.3 and Proposition 5.2]{kwa-larI} we have an injective Nica covariant representation  $\LLL:\KK\to \LL(\FF_{\KK})$, there denoted $T^e$,  determined by
\begin{equation}\label{Toeplitz representation definition}
\LLL_{p,q}(a)x =\begin{cases}
(a \otimes 1_{q^{-1}s})  x & \textrm{ if } s\in   qP ,
\\
0 & \textrm{ otherwise},
\end{cases}
\end{equation}
for $a \in \KK(p,q)$, $x\in X_{s}$ and $p,q,s\in P$.
We call   $\LLL$ given by \eqref{Toeplitz representation definition} the \emph{Fock representation} of $\KK$. The \emph{reduced Nica-Toeplitz algebra} of $\KK$  is the $C^*$-algebra $
\NT^{r}_{\LL}(\KK):=C^*(\LLL(\KK))$.  When $\KK=\LL$, we also write
$
\NT^{r}(\LL):=\NT^{r}_{\LL}(\LL).
$
By  \eqref{condition for reducing Fock-reps} and \cite[Proposition 7.6]{kwa-larI}, the $C^*$-algebra $\NT^{r}_{\LL}(\KK)$ defined above is naturally isomorphic to the one introduced in \cite[Definition 5.3]{kwa-larI}. Hence  the two definitions are consistent.
We refer to
 $\LLL \rtimes P:\NT_\LL(\KK)\to \NT^{r}_{\LL}(\KK)$ as  \emph{the regular representation} of $\NT_{\LL}(\KK)$. We say that $\KK$ is \emph{amenable} when $\LLL \rtimes P$ is an isomorphism. A number of amenability criteria are given in \cite[Section 8]{kwa-larI}.
\section{$C^*$-algebras associated to product systems}\label{section:Cstar-algebras-product-systems}

In this section we construct and analyze a canonical right-tensor $C^*$-precategory associated to an arbitrary product system $X$. We employ it  to prove the announced uniqueness results.

\subsection{Right tensor $C^*$-precategories associated to product systems}\label{technical subsection}
Let $X$ be  a product system  over a semigroup $P$ with coefficients in a $C^*$-algebra $A$. We will associate to $X$ a right-tensor $C^*$-precategory. In the case $P=\N$, it was constructed in \cite[Example 3.2]{kwa-doplicher}, and in the case $P$ is arbitrary, but the product system $X$ is regular, it was introduced in \cite[Definition 3.1]{kwa-szym}.   For $p,q\in P$ we put
$$
\LL_X(p,q):=\begin{cases} \LL(X_{q},X_{p}), & \text{ if }p,q\in P\setminus\{e\},
\\
\KK(X_{q},X_{p}), & \text{ otherwise}.
\end{cases}
$$
 With  operations inherited from the corresponding spaces, $\LL_X$ forms   a $C^*$-precategory.  The reason for considering smaller spaces than $\LL(X_{q},X_{p})$ when $p$ or $q$ is the unit $e$ is that in general it is not clear how to define right tensoring on such spaces, cf. Remark \ref{extending the right tensoring} below.
On the other hand, using the isomorphism \eqref{C-correspondence isomorphism}, for all $p,q\in P$, we have the following isomorphisms of $C^*$-correspondences over $A$:
$$
\LL_X(p,e)\cong X_p,\qquad \LL_X(e,q)\cong \widetilde{X}_q
$$
where $\widetilde{X}_q$ is a (left) $C^*$-correspondence dual to $X_q$. In particular, $\LL_X(e,e)=A$.

We will  describe a right tensoring structure on $\LL_X$  by introducing a family of mappings $\iota_{{p},{q}}^{{pr},{qr}}: \LL(X_{q},X_{p}) \to \LL(X_{qr},X_{pr})$, $p,q,r\in P$, which extends the standard family of diagonal homomorphisms $\iota_{q}^{qp}$, see \eqref{iotapq def}. If $q\neq e$ we put
$$
\iota^{pr,qr}_{p,q}(T)(xy):=(Tx)y,\,\, \,\,\,\,\,\textrm{ where } x\in X_q,\, y\in X_r  \textrm{ and } T\in \LL(X_{q},X_{p}).
$$
Note that under the canonical isomorphisms $ X_{qr}\cong X_q \otimes_A X_r$ and $X_{pr}\cong X_p \otimes_A X_r$, the operator
$\iota^{pr,qr}_{p,q}(T)$ corresponds to $T\otimes 1_{r}$, where $1_r$ is the identity in $\LL(X_r)$, and in particular $\iota^{pr,qr}_{p,q}(T)\in \LL(X_{qr},X_{pr})$. In the case $q=e$,
using \eqref{C-correspondence isomorphism},  the formula
$$
\iota^{pr,r}_{p,e}(t_x)(y):=xy,\,\, \,\,\,\,\,\textrm{ where }\, y\in X_r  \textrm{ and } t_x\in \KK(X_{e},X_{p}),  x\in X_p,
$$
yields a well defined map. As above, under natural identifications, the operator $\iota^{pr,r}_{p,e}(t_x)$ corresponds to $t_x\otimes 1_{r}\in \LL(X_e \otimes_A X_r, X_p \otimes_A X_r)$ and therefore $\iota^{pr,r}_{p,e}(t_x)\in \LL(X_{r},X_{pr})$.   Note that  $\iota_{p,p}^{pr,pr}=\iota_{p}^{pr}$.
\begin{prop}
 The linear maps $\iota_{{p},{q}}^{{pr},{qr}}: \LL_X(p,q) \to \LL_X(pr,qr)$, $p,q,r\in P$,
yield a right tensoring on the $C^*$-precategory $\LL_X$. We write
 $$
T\otimes 1_r:=\iota^{pr,qr}_{p,q}(T),\qquad \quad T\in\LL_X(p,q),\,\, p,q \in P.
 $$
\end{prop}
\begin{proof} It suffices to check that
$
 \iota_{{p},{q}}^{{pr},{qr}}(T)^*=\iota_{{q},{p}}^{{qr},{pr}}(T^*)$,
 $\iota_{{p},{q}}^{{pr},{qr}}(T) \iota_{{q},{s}}^{{qr},{sr}}(S)= \iota_{{p},{s}}^{{pr},{sr}}(TS)$,
 and $
  \iota_{{pr},{qr}}^{{prs},{qrs}}(\iota_{{p},{q}}^{{pr},{qr}}(T)) = \iota_{{p},{q}}^{{prs},{qrs}}(T)$,
for all $T \in \LL_X(X_{q},X_{p})$, $S \in \LL_X(X_{s},X_{q})$, $p,q,r,s\in P$. Viewing operators $\iota^{pr,qr}_{p,q}(T)$ as $T\otimes 1_{r}$, see discussion above, this is straightforward.
\end{proof}
\begin{defn}\label{right tensor C-precategory}
We call the pair $(\LL_X,\{\otimes 1_r\}_{r\in P})$ constructed above,
\emph{the right-tensor $C^*$-precategory associated to the  product system} $X$. We also put
$$
\KK_X(p,q):=\KK(X_{q},X_{p}) \qquad p,q\in P.
$$
Clearly, $\KK_X:=\{\KK_X(p,q)\}_{p,q\in P}$ is an essential  ideal in the $C^*$-precategory $\LL_X$.
\end{defn}
\begin{rem}\label{extending the right tensoring}  If each $C^*$-correspondence $X_p$, $p\in P$, is left essential (which is automatic when $P^*\neq \{e\}$) then the formula
$$
(T\otimes 1_r)(ax):=T(a)x, \qquad T\in\LL(A,X_{p}), \,\, a\in A, \,\,x \in X_r, \,\, p,r \in P,
$$
 allows one  to  extend the right tensoring from $\LL_X$ on the whole $C^*$-category $\{\LL(X_q,X_p)\}_{p,q\in P}$.
Note that $\LL_X$ is a $C^*$-category if and only if $A$ is  unital, if and only if $\LL_X=\{\LL(X_q,X_p)\}_{p,q\in P}$.
\end{rem}

 \begin{lem}\label{non-degeneracy of K_X}
The ideal $\KK_X$ in the right-tensor $C^*$-precategory $(\LL_X,\{\otimes 1_r\}_{r\in P})$  associated to the  product system $X$ is  $\otimes 1$-nondegenerate. Moreover, $\KK_X \otimes 1\subseteq \KK_X$  if and only if $\phi_p(A)\subseteq \KK(X_p)$ for every $p\in P$.
\end{lem}
\begin{proof}
Let $x,y,z \in X_p$,  $u\in X_r$ and $v\in X_{pr}$ for some $p,r\in P$, $p\neq e$. Then
$
(\Theta_{x,y}\otimes 1_r)  \Theta_{zu,v}=\Theta_{x\langle y,z\rangle_{X_p} u, v}.
$
Since elements of the form $x\langle y,z\rangle_{X_p}$ span $X_p$ and since  $X_pX_r=X_{pr}$, we conclude that  elements  $(\Theta_{x,y}\otimes 1_r)  \Theta_{zu,v}$ span  $\KK(X_{pr})$. Hence $\KK_X$  is  $\otimes 1$-nondegenerate. It is clear that
  $\KK_X \otimes 1\subseteq \KK_X$ implies $\phi_p(A)=A\otimes 1_p \subseteq \KK(X_p)$ for $p\in P$. The converse implication follows from the standard fact \cite[Proposition 4.7]{lance}, cf. \cite[Lemma 3.2]{kwa-szym}.
\end{proof}
\begin{rem}\label{rem:ideals in C*-precategories of Hilbert modules}
We claim that  $\KK_X$ and $A$ may be regarded as being Morita equivalent as $C^*$-precategories, see also Lemma~\ref{representations of C*-precategories of Hilbert modules}. To make this more precise, we introduce some notation first. If $Z$ and $Y$ are two right Hilbert $A$-modules and $I$ is an ideal in $A$, we let
$
\LL_I(Z,Y):=\{a\in \LL(Z,Y): a(Z)\subseteq YI\}.
$
Note that  $a(Z)\subseteq YI$ is equivalent to $a^*(Y)\subseteq ZI$. We also put $\KK_I(Z,Y):=\LL_I(Z,Y)\cap \KK(Z,Y)$, cf. \cite[Lemmas 1.1 and 1.2]{kwa-doplicher}.
 Now, given a product system $X$ over $P$,
 for any ideal $J$  in $A$ the formulas
$$
 \KK_X(J):=\{\KK_{J}(X_q,X_p)\}_{p,q \in P},\qquad  \LL_X(J):=\{\LL_{J}(X_q,X_p)\}_{p,q \in P}
$$
define  ideals in $\LL_X=\{\LL(X_q,X_p)\}_{p,q \in P}$, as the reader may readily verify. We claim that every ideal in $\KK_X=\{\KK(X_q,X_p)\}_{p,q\in P}$ is of the form $\KK_X(J)$ for some  ideal $J$  in $A$. Indeed,   this  is proved in  \cite[Proposition 2.17]{kwa-doplicher} in the case $P=\N$ but the proof works for general $P$.
\end{rem}

\begin{lem}\label{representations of C*-precategories of Hilbert modules}
  We have a one-to-one correspondence, established by  the formula
\begin{equation}\label{precategory representation}
\Psi_{p,q}(\Theta_{x,y})=
\psi_p(x)\psi_q(y)^*\,\,\,  \; \text{for}\; x\in X_p,\,\, y \in X_q, \,\, p,q\in P,
\end{equation}
between representations $\Psi=\{\Psi_{p,q}\}_{p,q\in P}$ of $\KK_X$  and  families $\psi=\{(\psi_e,\psi_{p})\}_{p\in P}$
 where, for each $p\in P$, $(\psi_e,\psi_{p})$ is a representation of the Hilbert $A$-module $X_p$.

Moreover, if $\Psi$ is a representation of $\KK_X$ on a Hilbert space $H$ and $\overline{\Psi}$
is its extension to $\LL_X$ determined by
\eqref{formula defining extensions of right tensor representations}, then with $\psi_e=\Psi_{e,e}$ we have
$$
\ker\Psi= \KK_X(\ker\psi_e)\qquad\text{ and }\qquad \ker \overline{\Psi}= \LL_X(\ker\psi_e).
$$
\end{lem}
\begin{proof}
 Let $\psi=\{(\psi_e,\psi_{p})\}_{p\in P}$ be a family of representations of right-Hilbert modules  $X_p$ for $p\in P$ in a $C^*$-algebra $B$. Let  $p,q,r \in P$. It is   well known that  \eqref{precategory representation} determines uniquely a linear contraction $\Psi_{p,q} : \K(X_q,X_p) \longrightarrow B$ which is isometric if $\psi_e$ is injective, see for instance  \cite[Lemma 2.2, Remark 2.3]{KPW}.   It is straightforward to see that the relations
 $$
 \Psi_{p,q}(S)^*= \Psi_{q,p}(S^*)\qquad \text{and}\qquad \Psi_{p,q}(S)\psi_{q,r}(T)=\psi_{p,r} (ST)
 $$
 hold for 'rank one' operators $S=\Theta_{x,y}$, $T=\Theta_{u,w}$. Hence these formulas hold for arbitrary $S\in \K(X_q,X_p)$,  $T \in \K(X_r,X_q)$.
Thus $\Psi=\{\Psi_{p,q}\}_{p,q\in P}$ is a representation of $\KK_X$.

If now $\Psi:=\{\Psi_{p,q}\}_{p,q\in P}$ is an arbitrary  representation  of $\KK_X$ in a $C^*$-algebra $B$ then by  \eqref{C-correspondence isomorphism} we can define a family of maps $\psi_p:X_p\to B$ by
$
\psi_p(x):=\Psi_{p,e}(t_x)$,  $x\in X_p, \,\, p\in P.
$ A routine verification shows that  $(\psi_e,\psi_p):X_p\to B$ is a right-Hilbert module representation.

The equality $\ker\Psi= \KK_X(\ker\psi_e)$ follows from Remark \ref{rem:ideals in C*-precategories of Hilbert modules}, as $\KK_X(\ker\psi_e)(e,e)=\ker\psi_e=(\ker\Psi)(e,e)$. By  \cite[Proposition 2.13]{kwa-larI} we have $(\ker\overline{\Psi})(p,q)=\{a\in \LL_X(p,q): a\KK_X(q,q)\subseteq \ker\Psi_{p,q}\}$. Thus the inclusion  $\LL_X(\ker\psi_e)(p,q) \subseteq (\ker \overline{\Psi}) (p,q)$ is immediate. For the reverse let  $a\in (\ker \overline{\Psi}) (p,q)$ and $x\in  X_q$. Note that $x$ may be written as $x=bx'$ where $b\in \KK(X_q)$ and $x'\in X_q$. Hence $ax= (ab)x'\in  \KK_X(p,q) x'\subseteq X_p(\ker\psi_e)$, and $a\in \LL_X(\ker\psi_e)(p,q)$.
\end{proof}
\begin{rem}
The semigroup operation in $P$ is irrelevant for the assertions in  Remark \ref{rem:ideals in C*-precategories of Hilbert modules} and Lemma \ref{representations of C*-precategories of Hilbert modules} --  they remain true when $P$ is any set with a distinguished element $e\in P$ and $\{X_p\}_{p \in S}$ is a family of right Hilbert modules over a $C^*$-algebra $A$ such that $X_e=A_A$.
\end{rem}
\begin{prop} \label{going forward cor}
Let $X$ be a  product system over  an arbitrary semigroup $P$. The bijective correspondence in Lemma \ref{representations of C*-precategories of Hilbert modules}
restricts to a one-to-one correspondence  between  representations $\psi=\{\psi_{p}\}_{p\in P}$ of $X$  and   right-tensor  representations $\Psi:=\{\Psi_{p,q}\}_{p,q\in P}$ of $\KK_X$. In particular, $\TT(X)$ is isomorphic to $\TT_{\LL_X}(\KK_X)$, the Toeplitz algebra of $\KK_X$.
\end{prop}
\begin{proof}
 Let $\psi$ and $\Psi$ be the corresponding objects in Lemma \ref{representations of C*-precategories of Hilbert modules}.
Suppose first that $\Psi$ is a right-tensor representation of $\KK_X$. 
For any $x\in X_p$ and $y\in X_q$, $p,q\in P$,  we have  $(t_x\otimes 1_y) y=t_{xy}$. Thus
$$
\psi_p(x)\psi_q(y)=\Psi_{p,e}(t_x) \Psi_{q,e}(t_y)=\Psi_{pq,e}((t_x\otimes 1_y)t_y)=\psi_{pq}(x),
$$
so $\psi=\{\psi_{p}\}_{p\in P}$ is a representation of the product system $X$. Suppose now  that $\psi$ is a representation of $X$.  Let $p,q,s,t\in P$ with $s\geq q$. Consider $S=\Theta_{x,y}\in \KK(X_q,X_p)$ and  $T=\Theta_{u'u,w}\in \KK(X_t,X_s)$ where $u'\in X_{q}$ and $u\in  X_{q^{-1}s}$. Then
\begin{align*}
\Psi_{p,q}(S) \Psi_{s,t}(T)&=\psi(x) \psi(y)^*\psi(u'u)\psi(w)^*= \psi(x\langle y,u'\rangle_A u)\psi(w)^*
= \Psi_{pq^{-1},t}(\Theta_{x\langle y,u'\rangle_A u, w})
\\
&=\Psi_{pq^{-1},t}(\iota^{pq^{-1}s,s}_{p,q}(\Theta_{x, y})\Theta_{u'u,w})= \Psi_{pq^{-1},t}((S\otimes 1_{q^{-1}s}) T).
\end{align*}
Hence, by linearity and continuity, $\Psi$ is a right-tensor representation.

 \end{proof}

\subsection{$C^*$-algebras associated with product systems over LCM semigroups}\label{product systems over LCMs}
For the remaining of this section we assume that $P$ is a right LCM semigroup.
\begin{lem}
A product system $X$ over $P$ is compactly-aligned if and only if the ideal $\KK_X$  is well-aligned  in the associated right-tensor $C^*$-precategory $(\LL_X,\{\otimes 1_r\}_{r\in P})$. In particular,  $\KK_X$ satisfies \eqref{non-degeneracy condition} and \eqref{condition for reducing Fock-reps}.
\end{lem}
\begin{proof}
$\KK_X$ satisfies \eqref{non-degeneracy condition} by Lemma \ref{non-degeneracy of K_X}. The remaining claims are immediate.
\end{proof}

The next proposition generalizes \cite[Proposition 3.14]{kwa-doplicher} from $\N$ to right LCM semigroups.

\begin{prop}\label{going forward prop}
 If $X$ is a compactly-aligned product system  over a right LCM semigroup $P$, then the bijective correspondence in Proposition~\ref{going forward cor} preserves Nica covariance of representations and, hence, it gives rise to a canonical isomorphism
$$
\NT_{\LL_X}(\KK_X)\cong \NT(X).
$$
\end{prop}
\begin{proof}

If $\Psi:=\{\Psi_{p,q}\}_{p,q\in P}$ is a Nica covariant representation  of $\KK_X$, then it is also a right-tensor representation and therefore $\psi=\{\psi_{p}\}_{p\in P}$ is a representation of the product system $X$.  Since  $\psi^{(p)}=\Psi_{p,p}$ and $\iota^{pq}_p(S)=S\otimes 1_q$,  for $S\in \KK(X_p)$ and $p, q\in P$, Nica covariance of $\Psi$  implies Nica covariance of $\psi$.

 Let $\psi=\{\psi_{p}\}_{p\in P}$ be a Nica covariant representation of $X$ on a Hilbert space $H$. Let  $\Psi=\{\Psi_{p,q}\}_{p,q\in P}$ be the representation of $\KK_X$ given by \eqref{precategory representation}.  To see that $\Psi$ is Nica covariant, let $S=\Theta_{x,y}\in \KK(X_q,X_p)$ and  $T=\Theta_{u,w}\in \KK(X_t,X_s)$ where $p,q,s,t\in P$.
Express $y=Yy'$ with $Y\in \KK(X_q)$ and $y'\in X_q$, and similarly $u=Uu'$ with $U\in \KK(X_s)$ and $u'\in X_s$.
Then  $\psi(y)^*\psi(u)=\psi(y')^*\psi^{(q)}(Y^*)\psi^{(s)}(U)\psi(u')$. Therefore, by Nica covariance of $\psi$, if $qP\cap sP=\emptyset$, then $\psi(y)^*\psi(u)=0$ and hence $\Psi_{p,q}(S)\Psi_{s,t}(T)=0$. Assume that  $qP\cap sP=rP$, for some $r\in P$.
Again by Nica covariance of $\psi$ we get
$$\psi(y)^*\psi(u)= \psi(y')^*\psi^{(r)}\Big( (Y^*\otimes {1_{q^{-1}r}})  (U\otimes {1_{s^{-1}r}})\Big)\psi(u')
$$
We claim that $\psi(y)^*\psi(u)\in \Psi_{q^{-1}r,s^{-1}r } (\KK(X_{s^{-1}r},X_{q^{-1}r})$.
Indeed, the operator $\psi^{(r)}\Big( (Y^*\otimes{1_{q^{-1}r}})  (U\otimes{1_{q^{-1}s}})\Big)$ can be approximated by finite sums of elements of the form $\psi_r(v'v)\psi_r(z'z)^*$ where $v' \in X_{q}$, $v\in X_{q^{-1}r}$, and $z' \in X_{s}$, $z\in X_{s^{-1}r}$. Since
$$
\psi_q(y')^*\psi_r(v'v)\psi_r(z'z)^*\psi_s(u')=\psi_{q^{-1}r}(\langle y',v'\rangle_q v)\psi_{s^{-1}r}(\langle u',z'\rangle_s z)^*
$$
is an element of $\Psi_{q^{-1}r,s^{-1}r }\big(\KK(X_{s^{-1}r},X_{q^{-1}r})\big)$, so is  $\psi(y)^*\psi(u)$.
Accordingly,
$$
\Psi_{p,q}(S)\Psi_{s,t}(T)=\psi_p(x)\psi_q(y)^*\psi_s(u)\psi_t(w)^* \in \Psi_{pq^{-1}r,ts^{-1}r }\big(\KK(X_{ts^{-1}r},X_{pq^{-1}r})\big).
$$
Hence the product $\Psi_{p,q}(S)\Psi_{s,t}(T)$ acts as zero on the orthogonal complement of the space
$
H_{ts^{-1}r }:=\overline{\psi_{ts^{-1}r }(X_{ts^{-1}r})H}.
$
 Clearly, the same is true for the operator $\Psi_{pq^{-1}r,ts^{-1}r }\big((S\otimes 1_{q^{-1}r})(T\otimes 1_{s^{-1}r})\big)$.
Consider an element $\psi_{ts^{-1}r }(w_0u_0)h$ where  $w_0\in X_t$, $u_0\in X_{s^{-1}r}$, $h\in H$. The linear span of such elements  is in $H_{ts^{-1}r }$ and we have
\begin{align*}
\Psi_{s,t}(T)\psi_{ts^{-1}r }(w_0u_0)&= \psi_s(u)\psi_t(w)^* \psi_{t}(w_0)\psi_{s^{-1}r}(u_0)=\psi_{r}(u\langle w, w_0\rangle_{t} u_0)
\\
&
=\psi_{r}\big((\Theta_{u,w}\otimes 1_{s^{-1}r})  w_0 u_0\big)=\Psi_{r,ts^{-1}r}(T\otimes 1_{s^{-1}r})\psi_{ts^{-1}r }(w_0u_0).
\end{align*}
Hence $\Psi_{s,t}(T)$ and $\Psi_{r,ts^{-1}r }\big(T\otimes 1_{s^{-1}r}\big)$ coincide on the space $H_{ts^{-1}r }$, and they map this space into the space $
H_{r }:=\overline{\psi_{r}(X_{r})H}$.  Consider an element $\psi_{r }(y_0x_0)h$ where  $x_0\in X_{q^{-1}r}$, $y_0\in X_{q}$, $h\in H$. The linear span of such elements  is in $H_{r}$ and we have
\begin{align*}
\Psi_{p,q}(S)\psi_{r }(y_0x_0)&= \psi_p(x)\psi_q(y)^* \psi_{q}(y_0)\psi_{q^{-1}r}(x_0)=\psi_{pq^{-1}r}(x \langle y, y_0\rangle_{q} u_0)
\\
&
=\psi_{r}\big((\Theta_{x,y}\otimes 1_{q^{-1}r})  y_0 x_0\big)=\Psi_{pq^{-1}r,r}(S\otimes 1_{q^{-1}r})\psi_{r }(y_0x_0).
\end{align*}
Hence $\Psi_{p,q}(S)$ and $\Psi_{pq^{-1}r,r}(S\otimes 1_{q^{-1}r})$ coincide when restricted to $H_r$. Combining these two observations we get
\begin{align*}
\Psi_{p,q}(S)\Psi_{s,t}(T) &= \Psi_{pq^{-1}r,r}(S\otimes 1_{q^{-1}r})\Psi_{r,ts^{-1}r}(T\otimes 1_{s^{-1}r})
\\
&=\Psi_{pq^{-1}r,ts^{-1}r }\big((S\otimes 1_{q^{-1}r})(T\otimes 1_{s^{-1}r})\big).
\end{align*}
Thus $\Psi:=\{\Psi_{p,q}\}_{p,q\in P}$ is Nica covariant.

 \end{proof}

The above result motivates the following definition.
\begin{defn}  Let $X$ be a compactly-aligned product system  over a right LCM semigroup $P$.
 We  let $\NT^r(X):=\NT_{\LL_X}^r(\KK_X)$ and call it \emph{the reduced Nica-Toeplitz algebra} of $X$.
We also put
$$
\mathcal{DR}(\NT(X)):=\NT(\LL_X)\qquad\text{and}\qquad \mathcal{DR}^r(\NT(X)):=\NT^r(\LL_X)
$$
($\mathcal{DR}$ stands for Doplicher-Roberts).
 We denote by $\overline{\Lambda}:\mathcal{DR}(\NT(X))\to  \mathcal{DR}^r(\NT(X))$ and $\Lambda:\NT(X)\to  \NT^r(X)$ the canonical epimorphisms.
\end{defn}
Let $X$ be a compactly-aligned product system.
In \cite{F99} Fowler constructed the \emph{Fock representation} $l:X\to \LL(\mathcal{F}(X))$ of $X$.
The Fock spaces for $X$, $\KK_X$ and $\LL_X$ coincide and are equal to the Hilbert $A$-module direct sum
$\mathcal{F}(X)=\bigoplus_{p\in P} X_p$. The Fock representation $\overline{\LLL}:\LL_X\to \LL(\mathcal{F}(X))$
of $\LL_X$ is an extension of the  Fock representation  $\LLL:\KK_X\to \LL(\mathcal{F}(X))$ of $\KK_X$ which is in turn an
extension of  $l$. This in particular leads to the inclusion
$$
\NT^r(X)=\clsp\{ l(x)l(y)^*: x, y\in X\} \subseteq \mathcal{DR}^r(\NT(X))
$$
By  \cite[Lemma 11.1]{kwa-larI}, there is a commutative diagram
$$
\begin{xy}
\xymatrix{
\NT(X) \ar[d]_{\Lambda} \ar[rr]^{\hookrightarrow }&  & \mathcal{DR}(\NT(X)) \ar[d]^{\overline{\Lambda}}
  \\
	\NT^r(X) \ar[rr]^{\hookrightarrow}	&         & 	\mathcal{DR}^r(\NT(X))}
  \end{xy}
$$
in which the horizontal maps are embeddings.  The map $\overline{\Lambda}$ is injective on the core subalgebra $B_e^{i_{\LL_X}}=\clsp\{i_{\LL_X}(a):a\in \LL_X(p,p), p \in P\}$  of $\DR(\NT(X))$ and $\Lambda$ is injective on
$$
B_e^X:=B_e^{i_{\KK_X}}=\clsp\{i_X(x) i_X(y)^*: x,y \in X, d(x)=d(y)\},
$$
see \cite[Corollary 6.4]{kwa-larI}. Clearly, $B_e^X\subseteq B_e^{i_{\LL_X}}$.	We will characterize representations of $X$ that give rise to injective representations of  $B_e^X$ and
$B_e^{i_{\LL_X}}$ respectively. To this end, we introduce  canonical projections associated to a  representation of $X$.
\begin{defn}
If $\psi:X\to \B(H)$ is a  representation of a compactly-aligned product system $X$, for every $p\in P$ we denote by $Q^\psi_{p}\in \B(H)$ the projection such that
$$
Q^\psi_{p}H=\begin{cases}
\psi^{(p)}(\KK(X_p))H  & \textrm{ if } p\in P\setminus \{e\},
\\
\overline{\psi(X_e)H} & \textrm{ if } p=e.
\end{cases}
$$
\end{defn}
\begin{rem}\label{properties of family projections2}
If  $\Psi:=\{\Psi_{p,q}\}_{p,q\in P}$ is the  representation of $\KK_X$ given by Proposition~\ref{going forward cor}, then  $Q^\psi_{p}$ equals the projection $Q^\Psi_{p}$ associated to $\Psi$ in \cite[Definition 9.1]{kwa-larI} for $p\in P$.
In particular, if  $\overline{\Psi}:=\{\overline{\Psi}_{p,q}\}_{p,q\in P}$ is the extension of $\Psi$ to $\LL_X$, then $Q^\Phi_p=\overline{\Psi}_{p,q}(i_\LL(1_{p}))$, for all $p\in P\setminus \{e\}$.
\end{rem}
\begin{lem}\label{Nica relation Lemma}
A representation $\psi:X\to \B(H)$ is Nica covariant if and only if the projections $\{Q^\psi_{p}\}_{p\in P}\in \B(H)$ satisfy the Nica covariance relation
\begin{equation}\label{Nica equation for projections}
Q^\psi_{p} Q_{q}^\psi=
\begin{cases}
Q_r^\psi, & \text{if } qP\cap sP=rP \text{ for some }r\in P,
\\
0, & \text{if } qP\cap sP=\emptyset.
\end{cases}
\end{equation}
Moreover, if $\psi:X\to \B(H)$ is Nica covariant,  the representation
$\psi\rtimes P:\NT(X)\to B(H)$ extends uniquely to a representation $\overline{\psi\rtimes P}$ of $\mathcal{DR}(\NT(X))$ such that
\begin{equation}\label{extensions relations}
(\overline{\psi\rtimes P})(i_{\LL_X}(a))=  Q^\psi_{p} (\overline{\psi\rtimes P})(i_{\LL_X}(a))Q^\psi_{q}, \quad a\in \LL(X_p,X_q), p,q\in P\setminus\{e\}.
\end{equation}
In fact, $\overline{\psi\rtimes P}=\overline{\Psi}\rtimes P$ where  $\Psi:=\{\Psi_{p,q}\}_{p,q\in P}$ is the associated representation of  $\KK_X$.
\end{lem}
\begin{proof}
If $\psi:X\to \B(H)$ is Nica covariant then \eqref{Nica equation for projections} holds by \cite[Propositions  9.4 and 9.7]{kwa-larI}. Conversely, if \eqref{Nica equation for projections} holds then for $a \in \KK(X_p)$ and $b\in \KK(X_q)$
the product   $\psi^{(p)}(a)\psi^{(q)}(b)$ equals $\psi^{(p)}(a)Q^\psi_{p}Q^\psi_{q}\psi^{(q)}(b)$,
which is zero if  $pP\cap qP=\emptyset$ or is $\psi^{(p)}(a)Q^\psi_{r}Q^\psi_{q}\psi^{(q)}(b)$ if $pP\cap qP=rP$. In the latter case, we have $Q^\psi_{r}\leq Q^\psi_{q}$ and  so by \cite[Proposition 9.4]{kwa-larI} applied to $\overline{\Psi}:=\{\overline{\Psi}_{p,q}\}_{p,q\in P}$, we obtain
$\psi^{(p)}(a)Q^\psi_{r}\psi^{(q)}(b)=\psi^{(r)}\big(\iota^{r}_p(a)\iota^{r}_q(b)\big)$, using also that $\iota^{r}_p(a)\iota^{r}_q(b)\in \KK_X(r,r)$.

If $\psi:X\to \B(H)$ is Nica covariant, then $\overline{\Psi}:=\{\overline{\Psi}_{p,q}\}_{p,q\in P}$ is Nica covariant by \cite[Proposition 9.5]{kwa-larI}. Putting $\overline{\psi\rtimes P}:=\overline{\Psi}\rtimes P$ for any  $a\in \LL(X_p,X_q)$, $p,q\in P\setminus\{e\}$ we get
$(\overline{\psi\rtimes P})(i_{\LL_X}(a))= \overline{\Psi}_{p,q}(i_{\LL_X}(a))$, and therefore relations \eqref{extensions relations} are satisfied. Conversely, if $\overline{\psi\rtimes P}$ is any representation of $\mathcal{DR}(\NT(X))$ that extends $\psi\rtimes P$, we get a Nica covariant representation $\Phi$ of $\LL_X$ that extends $\Psi$. The relations \eqref{extensions relations} imply that $\Phi=\overline{\Psi}$ and hence $\overline{\psi\rtimes P}:=\overline{\Psi}\rtimes P$.
\end{proof}

Via the bijective correspondence in \eqref{precategory representation}, we transport the notions of Toeplitz representation and condition (C) for representations of $\KK_X$, see \cite[Definition 6.2 and Definition 10.1]{kwa-larI} to representations of the product system.
\begin{defn} A Nica covariant representation $\psi:X\to \B(H)$  is  \emph{Toeplitz covariant} or \emph{Nica-Toeplitz covariant} if for each  finite family $q_1,\ldots ,q_n\in P\setminus P^*$, $n\in \N$, we have
\begin{equation}\label{Toeplitz condition666}
\psi_e(A)\cap \clsp\{ \psi^{(q_i)}(\KK(X_{q_i})):i=1,\ldots,n\} =\{0\}.
\end{equation}
The representation $\psi$ \emph{satisfies condition (C)} if for each finite family $q_1,\ldots ,q_n\in P\setminus P^*$, $n\in \N$,
\begin{equation}\label{Coburn condition666}
\text{the map $A\ni a \longmapsto  \psi_e(a) \prod_{i=1}^{n}(1-Q^\psi_{q_i})$ is injective.}
\end{equation}
	\end{defn}
\begin{lem}\label{conditions T and C for product systems}
Let $\psi:X\to \B(H)$ be a Nica covariant representation of a compactly-aligned product system $X$  and $\Psi:=\{\Psi_{p,q}\}_{p,q\in P}$ the associated representation of $\KK_X$.
\begin{itemize}
\item[(i)]  $\psi$ is Toeplitz covariant  if and only if $\Psi$ is Toeplitz covariant.
\item[(ii)]  $\psi$ satisfies condition (C) if and only if $\Psi$ satisfies condition (C).
\end{itemize}
\end{lem}
\begin{proof}

(i). Since $\psi^{(q)}(\KK(X_q))=\psi_q(X_q)\psi_q(X_q)^*=\Psi_{q,q}(\KK(X_q))$, Toeplitz covariance of $\Psi$ immediately implies that of $\psi$.  The converse is left to the reader.

(ii). In one direction, it is immediate that $\psi$ satisfies condition (C) when $\Psi$ does. For the converse,  let $p\in P$  and $q_1,\ldots,q_n\in P$ be such that  $p\not\geq q_i$, for $i=1,\ldots,n$ for some $n\in \N$. We must show that $\KK(X_p)\ni a  \longrightarrow \psi^{(p)}(a)\prod_{i=1}^{n}(1-Q^\psi_{q_i})=\Psi_{p,p}(a)\prod_{i=1}^{n}(1-Q^\Psi_{q_i})$ is injective.

 It follows from \eqref{Nica equation for projections} that $Q^\psi_{p}\prod_{i=1}^{n}(1-Q^\psi_{q_i})=\prod_{i=1}^{k}(Q^\psi_{p}-Q^\psi_{s_i})$ for some $k\geq n$ and $p \leq s_i \not\leq p$. Then  $p^{-1}s_i\notin P^*$, for all $i=1,\ldots,k$. Hence the space $K=\prod_{i=1}^{k}(1-Q^\psi_{p^{-1}s_i})H$ is invariant for $\psi_e$ and $\psi_e|_{K}$ is faithful. By  \cite[Proposition 1.6(2)]{Fow-Rae}, this implies that $\overline{\psi_p(X_p)K}$ is invariant for $\psi^{(p)}$ and  $\psi^{(p)}|_{\psi_p(X_p)K}$ is faithful. For any $q\in P$, applying \cite[Proposition 9.4]{kwa-larI} twice we get
$$
\psi_p(X_p)Q^\psi_q=\Psi_{p,e}(X_p)Q^\Psi_q=\overline{\Psi}_{pq,p}(X_p\otimes 1_q)=Q^\Psi_{pq}\Psi_{p,e}(X_p)=Q^\psi_{pq}\psi_p(X_p).
$$
Therefore
$$
\psi_p(X_p)\prod_{i=1}^{k}(1-Q^\psi_{p^{-1}s_i})= \prod_{i=1}^{k}(Q^\psi_{p}-Q^\psi_{s_i}) \psi_p(X_p)= Q^\psi_{p}\prod_{i=1}^{n}(1-Q^\psi_{q_i}) \psi_p(X_p),
$$
which gives the desired injectivity.
\end{proof}
\begin{thm}[Faithfulness on core subalgebras]\label{thm:faithulness on the core subalgebras}
Let $\psi:X\to \B(H)$ be a Nica covariant representation of a compactly-aligned product system $X$. Let $\overline{\psi\rtimes P}:\mathcal{DR}(\NT(X))\to B(H)$  be the extension of $\psi\rtimes P:\NT(X)\to B(H)$ described in Lemma \ref{Nica relation Lemma}.
\begin{itemize}
\item[(i)] $\psi\rtimes P$ is faithful on the core $B_e^X$  of $\NT(X)$ if and only if $\psi$ is injective and Toeplitz covariant.

\item[(ii)]  $\overline{\psi\rtimes P}$ is faithful on  the core  $B_e^{i_{\LL_X}}$  of $\mathcal{DR}(\NT(X))$  if and only if  $\psi$ satisfies condition (C).
\end{itemize}
Moreover, if   $\phi_p(A)\subseteq \KK(X_p)$ for every $p\in P$, then the equivalent conditions in (i) are satisfied if and only if those in (ii) hold.
\end{thm}
\begin{proof}
Item (i) follows from Lemma  \ref{conditions T and C for product systems}(i) and \cite[Corollary 6.3]{kwa-larI} applied to $\Psi$. By Lemma  \ref{conditions T and C for product systems}(ii) and \cite[Corollary 10.5]{kwa-larI},
$\psi$ satisfies condition (C) if and only if $\overline{\Psi}$ is Nica-Toeplitz covariant and injective. Hence (ii) follows from \cite[Corollary 6.3]{kwa-larI}  applied to $\overline{\Psi}$. The last claim of the theorem follows from Lemma \ref{non-degeneracy of K_X} and \cite[Proposition 10.4]{kwa-larI}.
\end{proof}

\subsection{Uniqueness theorems}\label{subsect:uniqueness theorems}
We aim to prove a uniqueness result for $\NT(X)$. Our result, see Theorem~\ref{Uniqueness Theorem for  product systems I}, may be considered a far-reaching generalization of \cite[Theorem 2.1]{Fow-Rae} and \cite[Theorem 3.7]{LR}, and was motivated in part by the need to better understand both hypotheses and claims of \cite[Theorem 7.2]{F99}. The proof will employ our abstract uniqueness theorem for $C^*$-algebras associated to well-aligned ideals in $C^*$-precategories, cf. \cite[Corollary 10.14]{kwa-larI}.
As spin-offs of our strategy of proof we will obtain a uniqueness result in a new context, see Theorem~\ref{Uniqueness for product systems II}, and a generalization of \cite[Theorem 5.1]{FR}.

We start with some preparation. We recall that aperiodicity for the group of right tensoring $\{\otimes 1_h\}_{h\in P ^*}$ in a $C^*$-precategory was introduced in \cite[Definition 10.8]{kwa-larI}. The notion of aperiodic Fell bundle is from \cite{KS}. Further, for any product system $X$ the spaces $\{X_h\}_{h\in P^*}$ form a saturated Fell bundle over the discrete group of units $P^*$, see Remark~\ref{rem:on essentiality and Fell bundles}. By  \cite[Theorem 9.8]{KM}, $\{X_h\}_{h\in P^*}$ is \emph{aperiodic} if and only if its dual action on the spectrum
$\widehat{A}$ is \emph{topologically free}, at least when $A$ contains an essential ideal which is separable or of Type I.
\begin{prop}\label{lem: aperiodicity for product systems}
If $X$ is a product system over a semigroup $P$, then the group $\{\otimes 1_{h}\}_{h\in P^*}$ of automorphisms of  $\KK_X$ is aperiodic if and only if the Fell bundle $\{X_h\}_{h\in P^*}$ is aperiodic.
\end{prop}
\begin{proof}
The only if part is trivial as $X_h=\KK_{X}(h,e)$, $h\in P^*$. For the converse, let $p\in P$ and $h\in P^*\neq \{e\}$. We may view $X_p$ as an equivalence $\KK_X(p,p)$-$A$-bimodule, in an obvious way. Also we  may view $\KK_X(ph,p)$ as a $C^*$-correspondence over $\KK_X(p,p)$  with left action implemented by $\otimes 1_h$. With $\widetilde{X}_{p}$ denoting the dual correspondence, we clearly have isomorphisms of $C^*$-correspondences
$$
\KK_X(ph,p)\cong X_{ph}\otimes_A \widetilde{X}_{p}\cong X_{p}\otimes_A X_{h}\otimes_A \widetilde{X}_{p}
$$
Hence $\KK_X(ph,p)$ is an equivalence bimodule Morita equivalent to the equivalence $A$-bimodule $X_h$, cf. \cite[Lemma 6.4]{KM}.
Thus the assertion follows from \cite[Corollary 6.3]{KM}.
\end{proof}

We recall that  various  criteria for amenability of ideals in right-tensor $C^*$-precategories are  given in \cite[Section 8]{kwa-larI}.  For example, any such ideal is amenable when the underlying semigroup admits a controlled map into an amenable group, see \cite[Theorem 8.4]{kwa-larI} in conjunction with the fact that any Fell bundle over an amenable group has amenable full sectional $C^*$-algebra.

\begin{thm}[Uniqueness Theorem for $\NT(X)$]\label{Uniqueness Theorem for  product systems I}
Let $X$ be a compactly-aligned product system  over a right LCM semigroup $P$ such that  $\KK_X$ is amenable. Suppose that either
$P^*=\{e\}$ or that the Fell bundle $\{X_h\}_{h\in P^*}$  is aperiodic.

 Consider the following conditions
on a Nica covariant representation $\psi:X\to B(H)$:
\begin{itemize}
\item[(i)] $\psi$ satisfies condition (C);
\item[(ii)] $\psi\rtimes P$ is an isomorphism from $\NT(X)$  onto $\clsp\{\psi(x)\psi(y)^*: x,y \in X \}
$;
\item[(iii)] $\psi$ is injective and Toeplitz covariant.
\end{itemize}
Then  (i)$\Rightarrow$(ii)$\Rightarrow$(iii)  and if  $\phi_p(A)\subseteq \KK(X_p)$ for every $p\in P$, then all these three conditions are equivalent.
\end{thm}
\begin{proof}Taking into account Lemmas \ref{non-degeneracy of K_X}, \ref{conditions T and C for product systems} and Proposition \ref{going forward prop}, the assertion follows from \cite[Corollary 10.14]{kwa-larI}.
\end{proof}

In general, condition (i) in Theorem \ref{Uniqueness Theorem for  product systems I} is stronger then (iii), cf. Example \ref{ex:DR-Oinfty}. It is an open problem whether, under the assumptions of Theorem \ref{Uniqueness Theorem for  product systems I},  conditions (ii) and (iii) are always equivalent. We believe the answer to be affirmative and in the next result confirm this under an  assumption of aperiodicity.

\begin{prop}\label{Abstract uniqueness}
Suppose that the LCM semigroup $P$ is a subsemigroup of a group $G$. Let $X$ be a compactly-aligned product system  over
$P$, and let $\B^\theta=\{B_g^\theta\}_{g\in G}$ be the Fell
bundle associated to $\KK_X$ and $\theta=\id$ in \cite[Theorem 8.4]{kwa-larI}. If $\B^\theta$ is amenable and aperiodic, then for any Nica covariant representation $\psi:X\to B(H)$, the representation $\psi\rtimes P$ of $\NT(X)$ is faithful if and only if $\psi$ is injective and Toeplitz covariant.
\end{prop}
\begin{proof}
By \cite[Proposition 12.10]{kwa-larI}, see also  \cite[Corollary 4.3]{KS},  $\psi\rtimes P$ is faithful on $\NT(X)$ if and only if it is faithful on the core subalgebra  $B_e^\theta= B_e^X$. By Theorem \ref{thm:faithulness on the core subalgebras}, this holds if and only if   $\psi$ is injective and Toeplitz covariant. See also \cite[Remark 10.13]{kwa-larI}.
\end{proof}
We can use  condition (C) in its full force by  exploiting the Doplicher-Roberts version of the Nica-Toeplitz algebra.

\begin{thm}[Uniqueness Theorem for $\mathcal{DR}(\NT(X))$]\label{Uniqueness for product systems II}
Let $X$ be a compactly-aligned product system  over a right LCM semigroup $P$. Suppose   that either
$P^*=\{e\}$ or that the Fell bundle $\{X_h\}_{h\in P^*}$  is aperiodic. Assume also that $\LL_X$ is  amenable.
 Then  for a Nica covariant representation $\psi:X\to B(H)$ the following are equivalent:
\begin{itemize}
\item[(i)] $\psi$ satisfies  condition (C);
\item[(ii)] $\overline{\psi\rtimes P}$ is an isomorphism from $\mathcal{DR}(\NT(X))$ onto the
closed linear span of operators  $T$ satisfying $T\in \psi(X_e)\cup \psi(X_e)^*$ or
$$
T\in Q_{p}^{\psi}B(H)Q_{q}^{\psi}\, \text{ where }\,  T\psi(X_q)\subseteq \psi(X_p)\, \text{ and } \,T^{*}\psi(X_p)\subseteq \psi(X_q), \text{ for } p,q\in P\setminus\{e\}.
$$
\end{itemize}
The isomorphism in item (ii) restricts, under the embedding $\NT(X)\hookrightarrow \mathcal{DR}(\NT(X)))$, to a natural isomorphism
$\NT(X) \cong \clsp\{\psi(x)\psi(y)^*: x,y \in X \}.$
\end{thm}
\begin{proof} In view of Lemmas \ref{non-degeneracy of K_X}, \ref{conditions T and C for product systems}, we may apply \cite[Theorem 10.15]{kwa-larI}. To finish the proof, we need to show that for any $p,q\in P\setminus\{e\}$, we have
$$
\overline{\Psi}(\LL(X_q,X_p))=\{T\in Q_{p}^{\psi}B(H)Q_{q}^{\psi}:T\psi(X_q)\subseteq \psi(X_p)\,\text{ and } \, T^{*}\psi(X_p)\subseteq \psi(X_q)\}
$$
where $\overline{\Psi}:\LL_X\to B(H)$ is the extension of the Nica-Toeplitz representation  $\Psi:=\{\Psi_{p,q}\}_{p,q\in P}$ of $\KK_X$. This equality readily follows from the fact that $\psi$ is injective (and hence isometric on each fiber) and for $a\in \LL(X_q,X_p)$,  $\overline{\Psi}(a)$ is  determined by the formulas $\overline{\Psi}(a)\psi_q(x)=\psi_p(ax)$ and $\overline{\Psi}(a) (Q_{q}^{\psi})^\bot=0$, where $x\in X_q$.
  \end{proof}

	\begin{ex}[Doplicher-Roberts version of $\OO_\infty$]\label{ex:DR-Oinfty} We will now illustrate the above uniqueness results for $C^*$-algebras of product systems in the case of the Cuntz algebra $\OO_\infty$. This in particular will explain the results and phenomena encountered in \cite{Fow-Rae}.

Let $\{u_i: i\in \N\}\subset \B(H)$ be a family of isometries with orthogonal ranges, so 	$u_i^*u_j=\delta_{i,j} 1$ for all $i,j \in \N$. For $n>0$ and any finite collection $i_1,\dots, i_n$ of indices, let  $u_{i_1i_2\dots i_n}:=u_{i_1}u_{i_2}\dots u_{i_n}$ and define
$$
	X_n:=\clsp\{u_{i_1i_2\dots i_n}: i_1, \dots, i_n\in \N\} \qquad Q_n:=\sum_{i_1, \dots, i_n\in \N} u_{i_1i_2\dots i_n} u_{i_1i_2\dots i_n}^*.
					$$
					We also put $X_0=\C I$ and $Q_0=1$.
	 	The family $X=\{X_n\}_{n\in \N}$ with operations inherited from $\B(H)$ becomes a product system over the semigroup $\N$ with coefficient algebra $A=\C$. For each $n>0$, $Q_n$ is   the orthogonal projection onto the space  $X_nH$.
	Note that  \eqref{Coburn condition666}, which is our geometric condition (C), is equivalent to asking that $Q_1$ is not equal to $1$, i.e.
\begin{equation}\label{eq:ranges dont span}
  \sum_{i\in \N}u_iu_i^* < 1,
\end{equation}
where the infinite sum is defined using the strong operator topology.
 Since $\NT(X)$ is generated by an infinite family of isometries with orthogonal ranges given by $\{i_X(u_j):j\in \N\}$, \cite[Theorem 1.12]{Cu77} gives an isomorphism
\begin{equation}\label{normal O_infty}
	\NT(X)\cong \clsp\{ X_n X_m^*: n,m\in \N\}\cong \OO_\infty.
\end{equation}
In particular, every countably infinite family of isometries with orthogonal ranges gives rise to an injective Nica-Toeplitz representation of $X$ - the algebraic condition \eqref{Toeplitz condition666} is satisfied automatically.
We denote by $\mathcal{DR}(\OO_\infty)$ the Doplicher-Roberts algebra associated to $(X_n)_{n\in \N}$.  Theorem \ref{Uniqueness for product systems II} implies that \eqref{eq:ranges dont span} is equivalent to having an isomorphism
\begin{equation}\label{Doplicher-Roberts O_infty}
	\mathcal{DR}(\OO_\infty)\cong \clsp\left\{\bigcup_{n,m\in \N } \{T \in Q_mB(H)Q_n: \,\,TX_n \subseteq X_m \,\, \text{ and } \,\,T^{*}X_m\subseteq X_n\}\right\}.
	\end{equation}
Without  condition \eqref{eq:ranges dont span}, all we can say is that there is a  surjective homomorphism from $\mathcal{DR}(\OO_\infty)$ to the right-hand side of \eqref{Doplicher-Roberts O_infty} obtained from the universal property of $\mathcal{DR}(\NT(X))$.
\end{ex}

This example illustrates the fact that condition (C) captures uniqueness of  the Doplicher-Roberts  algebra $\mathcal{DR}(\OO_\infty)$, which is a $C^*$-algebra containing $\OO_\infty$, and that uniqueness of $\OO_\infty$ as a $C^*$-algebra generated by isometries with orthogonal ranges is independent of condition (C). This phenomenon  is consistent with our Theorem~\ref{Uniqueness Theorem for  product systems I}, as  the left action of $A=\C 1$ on $X_1\cong \ell^2(\N)$ is not by generalized compacts.

In order to get an efficient uniqueness theorem for $\OO_\infty$ one needs to view it as a Nica-Toeplitz algebra over the free semigroup $\F_{\N}^+$. This idea, in disguise, was exploited in \cite{Fow-Rae}. With our results in hand we can make it formal and explicit. Note that any product system over a free semigroup $\F_\Lambda^+$ is automatically compactly-aligned.
\begin{lem}\label{bla bla lemma for free product systems}
Let $Y:=\bigoplus_{\lambda \in \Lambda} Y_\lambda$ be a direct sum of $C^*$-correspondences $Y_\lambda$, $\lambda \in \Lambda$, over a  $C^*$-algebra $A$. There is a  product system  $X=\{X_p\}_{p\in  \F_\Lambda^+}$ over $A$ such that for any  word  $p=\lambda_1\dots \lambda_n\in \F_\Lambda^+$ we have
$$
X_{p}:=Y_{\lambda_1}\otimes Y_{\lambda_2}\otimes\dots\otimes Y_{\lambda_n}
$$
and the product in $X$ is given by the iterated internal tensor product.
We have a  one-to-one correspondence between Nica covariant representations $\Psi$ of $X=\{X_p\}_{p\in  \F_\Lambda^+}$ and  representations $(\pi,\psi)$ of the $C^*$-correspondence $Y$ where
$$
\Psi(y_{\lambda_1}\otimes y_{\lambda_2}\otimes\dots\otimes y_{\lambda_n})=\psi(y_{\lambda_1})\psi(y_{\lambda_2})\dots\psi(y_{\lambda_n}) ,\qquad y_{\lambda_i}\in Y_{\lambda_i}, i=1,\dots,n.
$$
 Thus we have a natural isomorphism $\TT_Y\cong \NT(X)$.
\end{lem}
\begin{proof}
The proof is straightforward. We leave the details to the reader. 
\end{proof}
\begin{cor}\label{Fowler-Raeburn result}
Let $Y=\bigoplus_{\lambda \in \Lambda} Y_\lambda$ be a direct sum of $C^*$-correspondences $Y_\lambda$, $\lambda \in \Lambda$, over a  $C^*$-algebra $A$. Consider the following conditions that a representation  $(\pi,\psi)$ of the  $C^*$-correspondence $Y$ in a Hilbert space $H$ may satisfy:
\begin{itemize}
\item[(i)] $A$ acts, via $\pi$, faithfully on $(\psi(\bigoplus_{\lambda \in F} Y_\lambda)H)^\bot$ for every finite subset $F$ of $\Lambda$;
\item[(ii)] The $C^*$-algebra generated by $\pi(A)\cup \psi(Y)$ is naturally isomorphic to  $\TT_Y$;
\item[(iii)] $\pi(A)\cap \clsp\{\psi(x)\psi(y)^*: x,y\in Y_\lambda,  \lambda \in F\}=\{0\}
$ for every finite  $F\subseteq \Lambda$.
\end{itemize}
Then   (i)$\Rightarrow$(ii)$\Rightarrow$(iii). Moreover, if $A$ acts by generalized compacts on the left of each  $Y_\lambda$, $\lambda\in \Lambda$,   then all the above conditions are equivalent.
\end{cor}	
\begin{proof}
Let $X=\{X_p\}_{p\in  \F_\Lambda^+}$ be the product system described in Lemma \ref{bla bla lemma for free product systems}. Since $\NT(X)$ and $\NT^r(X)$ are isomorphic by \cite[Corollary 8.6]{kwa-larI} and $(\F_\Lambda^+)^*=\{e\}$, we may apply Theorem \ref{Uniqueness Theorem for  product systems I} to $X$. Translating the result, using Lemma \ref{bla bla lemma for free product systems}, to $C^*$-correspondences $Y_\lambda$, $\lambda\in \Lambda$, we get the assertions.
\end{proof}
\begin{rem}
The relationship between conditions (i) and (ii) in Corollary \ref{Fowler-Raeburn result} 		was established in \cite[Theorem 3.1]{Fow-Rae}. The  algebraic condition (iii), which is what we call Toeplitz covariance, in general does not imply (i), see Example \ref{ex:DR-Oinfty}. We have already seen pieces of evidence  that in general Toeplitz covariance could be the right condition for characterizing uniqueness of
$\NT(X)$. Another evidence for this is again the case of $\OO_\infty$, as we shall now explain.

If we specialize Corollary~\ref{Fowler-Raeburn result} to the $C^*$-correspondence $X_1\cong \ell^2(\N)$ over $\C$ from  Example~\ref{ex:DR-Oinfty}, then we may view $X_1$ as a direct sum over $\N$ of finite dimensional (even one dimensional) spaces $Y_n$.
It is readily seen that for  every representation of $X$ coming from  a countably infinite family of isometries  with orthogonal ranges, both of conditions (i) and (iii) in  Corollary~\ref{Fowler-Raeburn result} are satisfied. Since the left action  is by  compacts in each $Y_n$, $n\geq 1$,   we may use each of these conditions to recover the uniqueness of $\OO_\infty$.
\end{rem}
\subsection{Semigroup $C^*$-algebras twisted by product systems}\label{Fowler-Raeburn section}
Let $X$ be a compactly-aligned product system over a right LCM semigroup $P$.
For each $p\in P$, let  $\mathds{1}_p \in \ell^\infty(P)$ be the characteristic function of $pP$. Since the
product $\mathds{1}_p \mathds{1}_q$ is either $\mathds{1}_r$ (if $pP\cap qP=rP$) or $0$, we have that $B_P:= \clsp\{\mathds{1}_p:p\in P\}$
 is a $C^*$-subalgebra of $\ell^\infty(P)$. Moreover, the projections $\mathds{1}_p$ form a semilattice isomorphic to $J(P)$. Recall that $1_r$ denotes the identity in $\LL(X_r)$ for every $r\in P$.
If $1$ is the identity in the unitization $\DR(\A)^{\sim}$  of $\DR(\A)$,
then the projections $\{i_{\LL_X}(1_p)\}_{p\in P\setminus\{e\}}\cup \{1\}$ form a semilattice isomorphic to $J(P)$, cf. \cite[Lemma 5.8]{kwa-larI}.  Since the family $J(P)$ is independent, see \cite[Corollary 3.6]{bls},
it follows from \cite[Proposition 2.4]{Li2} that  the assignment
$$
B_P \ni \mathds{1}_p \longmapsto i_{\LL_X}(1_{p})\in \DR(\A),   \qquad  p\in P\setminus \{e\},
$$
extends  uniquely to an injective unital homomorphism $B_P \hookrightarrow \DR(\A)^{\sim}$. We will use it
  to identify $B_P$ with a $C^*$-subalgebra of $\DR(\A)^{\sim}$.
\begin{defn}
Let $X$ be a compactly-aligned product system.  We call the $C^*$-algebra
$$
\FR(X):=C^*\left(B_P\cdot \NT(X)\right)\subseteq \DR(\A),
$$
generated by elements $i_{\LL_X}(1_p)i_{\KK_X}(a) $ where
$a\in \KK(q,r), p,q,r\in P$, $p\neq e$, the   \emph{Fowler-Raeburn algebra} of $X$ or the \emph{semigroup $C^*$-algebra of $P$ twisted by $X$}.
\end{defn}
\begin{prop}\label{Fowler-Raeburn algebra form}
We have  $\FR(X)=\clsp\{i_{\KK_X}(x) i_{\LL_X}(1_p)i_{\KK_X}(y)^* : x,y\in X, p \in P\}$.
In particular, $B_P \subseteq M(\FR(X))$. Moreover,
$\FR(X)=\A$ if and only if the left action of $A$ on each fiber $X_p$ is by generalized compacts.
\end{prop}
\begin{proof}
For any $x\in X$ and $p\in P$, using Nica covariance of $i_{\LL_X}$ twice, we get
$
i_{\KK_X}(x) i_{\LL_X}(1_p)=i_{\LL_X}(x\otimes 1_{p})= i_{\LL_X}(1_{d(x)p})i_{\KK_X}(x),
$
and similarly
\begin{equation}\label{Fowlers proof relation}
i_{\LL_X}(1_p) i_{\KK_X}(x)=
\begin{cases}
i_{\KK_X}(x)i_{\LL_X}(1_{d(x)^{-1}r}) & \text{if }pP\cap d(x)P=rP,
\\
0, & \text{otherwise}.
\end{cases}
\end{equation}
This implies that $B_P\cdot \NT(X)\subseteq \clsp\{i_{\KK_X}(x) i_{\LL_X}(1_p)i_{\KK_X}(y)^* : x,y\in X, p \in P\}\subseteq \FR(X)$. Hence to prove the first part of the assertion, it suffices to show that the product of two elements of the form $i_{\KK_X}(x) i_{\LL_X}(1_p)i_{\KK_X}(y)^*$ and  $ i_{\KK_X}(z) i_{\LL_X}(1_s)i_{\KK_X}(w)^*$, $x,y,z,w \in X$, $p,s\in P$, can be approximated by a finite sum of elements of that form. The product $i_{\KK_X}(y)^* i_{\KK_X}(z)$ can be approximated by a finite sum of elements of the form $i_{\KK_X}(f) i_{\KK_X}(g)^*$, $f,g\in X$. Applying \eqref{Fowlers proof relation} twice,  we see that  the product
$$
i_{\KK_X}(x) i_{\LL_X}(1_p) i_{\KK_X}(f) i_{\KK_X}(g)^* i_{\LL_X}(1_s)i_{\KK_X}(w)^*
$$
is either zero or of the form  $i_{\KK_X}(xf) i_{\LL_X}(1_t)i_{\KK_X}(wg)^*$. Thus $\FR(X)$ is the closed linear span of $\{i_{\KK_X}(x) i_{\LL_X}(1_p)i_{\KK_X}(y)^* : x,y\in X, p \in P\}$. Now, \eqref{Fowlers proof relation} implies $B_P \subseteq M(\FR(X))$.

If the left action of $A$ on each fiber $X_p$ is by compact operators, then
$\KK_X\otimes 1 \subseteq \KK_X$, by Lemma \ref{non-degeneracy of K_X}. Hence
$i_{\KK_X}(x) i_{\LL_X}(1_p)i_{\KK_X}(y)^*= i_{\KK_X}(x\otimes 1_{p})i_{\KK_X}(y)^*\in \A$ for every $x,y\in X$ and $p\in P$.
Therefore $\FR(X)\subseteq \A$ and the reverse inclusion is obvious.

Conversely, if $\FR(X)\subseteq \A$, then for any $a\in A$ and $p\in P\setminus\{e\}$ we have that
 $i_{\LL_X}(\phi_p(a))=i_{\LL_X}(1_p)i_{\KK_X}(a)\in \NT(X)$. Hence for any $\varepsilon >0$ there is a finite
sum of the form $S=\sum_{s,t} i_{\KK_X}(a_{s,t})$ where $a_{s,t}\in \KK(X_t,X_s)$ such that  $\|i_{\LL_X}(\phi_p(a)) -S\|< \varepsilon$. Hence $\|E^\LLL(\Lambda (i_{\LL_X}(\phi_p(a)))- E^\LLL(\Lambda (S))\|< \varepsilon$ where $E^\LLL$ is the transcendental conditional expectation on $\NT^r(X)$ constructed in \cite{kwa-larI}. By   \cite[Proposition 5.4]{kwa-larI}, cf. also \cite[Remark 5.7]{kwa-larI}, we have $E^{\LLL}\Big(\LLL(a_{p,q})\Big)= \bigoplus_{w\in pP\cap qP,\atop p^{-1}w=q^{-1}w}\LLL_{p,q}^{(w)}(a_{p,q})$, for $a_{p,q}\in \KK(X_q,X_p)$, $p,q \in P$ where  $\LLL^{(p)}_{p,p}=\id$ for every $p\in P$. This implies that $\|\phi_p(a)-a_{p,p} \|< \varepsilon$. Thus  $\phi_p(a)\in \KK(X_p)$.
\end{proof}

 Left translation on $\ell^\infty(P)$ restricts to a unital semigroup homomorphism $\tau:P \to  \End(\psi_e(A)')$,
determined by $\tau_q(\mathds{1}_p) = \mathds{1}_{qp}$ for $p,q\in P$. The isometric crossed product $B_P \rtimes _\tau P$ is naturally isomorphic to the semigroup $C^*$-algebra $C^*(P)$, see \cite[Lemma 2.14]{Li}, so $\FR(X)$ may be viewed as a version of $C^*(P)$ twisted by $X$, see  \cite{FR}, \cite{F99}. We  make this  explicit in our setting.

\begin{lem}\label{induced_endomorphic_action}
Let  $\psi$ be a nondegenerate representation of $X$ on a Hilbert space $H$. For each $p\in P\setminus \{e\}$ there is a unique endomorphism
$\alpha^\psi_p$ of $\psi_e(A)'$ such that
$$
\alpha^\psi_p(S) \psi_p(x) =\psi_p(x)S, \qquad \text{for all } S \in \psi_e(A)', \,\, x\in X_p,
$$
and  $\alpha^\psi_p(1)$ vanishes on $(\psi_p(X_p)H)^\bot$. We put $\alpha_e^\psi=\id$.
Then $\alpha^\psi: P \to  \End(\psi_e(A)')$  is a
unital semigroup homomorphism.
\end{lem}
\begin{proof} The existence of $\alpha^\psi_p$ for each $p\in P$ is proved in \cite[Proposition 4.1 (1)]{F99}. The semigroup law $\alpha^\psi_p\circ \alpha^\psi_q=\alpha^\psi_{pq}$ for $p,q\in P\setminus\{e\}$ is proved in \cite[Proposition 4.1 (2)]{F99}. To allow $p=e$, Fowler assumes all $X_p$ are essential. With our definition of $\alpha_e^\psi$, the semigroup law follows if one or both of $p,q$ equal $e$ by a direct verification.
\end{proof}

\begin{defn}
A \emph{covariant representation} of the quadruple $(B_P, P, \tau, X)$  on a Hilbert
space $H$ is a pair $(\pi,\psi)$ consisting of a nondegenerate representation $\pi:B_P \to B(H)$  and a nondegenerate representation
 $\psi:X\to B(H)$ such that
$
\pi(B_P)\subseteq \psi_e(A)'$  and $\pi \circ \tau_p = \alpha_p^\psi \circ \pi$,  $p \in P,$
where $\alpha^\psi: P \to  \End(\psi_e(A)')$ is defined in Lemma \ref{induced_endomorphic_action}.
\end{defn}
\begin{prop}
There is a bijective correspondence between  covariant representations $(\pi,\psi)$ of $(B_P, P, \tau, X)$ and Nica covariant representations $\psi$ of $X$ implemented by $\pi(\mathds{1}_p)= \alpha_p^\psi(1)$ for $p\in P$.

In particular,  there is a covariant representation $(i_{B_P}, i_X)$ of $(B_P, P, \tau, X)$ such that
\begin{itemize}
\item[1)] $\FR(X)=C^*(i_{B_P}(B_P) i_X(X))$
\item[2)] for every covariant representation $(\pi,\psi)$ of $(B_P, P, \tau, X)$ there is a representation $\pi\rtimes \psi$ of $\FR(X)$ such that $\overline{\pi\rtimes \psi} \circ i_{B_P} = \pi$ and
$\overline{\pi\rtimes \psi} \circ i_{X} =  \psi$.
\end{itemize}
  \end{prop}
\begin{proof}
If  $(\pi,\psi)$ is a covariant representation of $(B_P, P, \tau, X)$, then  $\pi(\mathds{1}_p)= \alpha_p^\psi(1)$, $p\in P$ and this relation determines $\pi$. Moreover,  we have $\pi(\mathds{1}_p)= \alpha_p^\psi(1)=Q_p^\psi$, and therefore $\psi$ is Nica covariant by Lemma \ref{Nica relation Lemma}. Conversely, if $\psi$ is a Nica covariant representations  of $X$, then  $\alpha_p^\psi(1)=Q_p^\psi$ satisfy \eqref{Nica equation for projections} and belong to $\psi_e(A)'$, cf. \cite[Proposition 9.5]{kwa-larI}. Hence there is a representation $\pi:B_P\to \psi_e(A)'$ determined by $\pi(\mathds{1}_p)= \alpha_p^\psi(1)$, $p\in P$, by \cite[Proposition 2.4]{Li2}. Since
$
\pi (\tau_p (\mathds{1}_{q})) = \pi(\mathds{1}_{pq})=  \alpha_{pq}^\psi(1)= \alpha_{p}^\psi(\alpha_{q}^\psi(1))=  \alpha_{p}^\psi(\pi(\mathds{1}_q))
$
for every $p,q\in P$, we conclude that  $(\pi,\psi)$ is a covariant representation of $(B_P, P, \tau, X)$.

The second part of the assertion is now immediate (by representing $\FR(X)$ faithfully and nondegenerately  on a Hilbert space).
\end{proof}

\begin{thm}[Uniqueness Theorem for $\FR(X)$]\label{Uniqueness for product systems III}
Retaining the assumptions in Theorem \ref{Uniqueness for product systems II} each of the conditions (i) and (ii)  therein
are equivalent to the following one:
\begin{itemize}
\item[(iii)]  $\FR(X)\cong \clsp\{\psi(x) Q_p^\psi \psi(y)^*: x,y\in X, p \in P\} $. In particular, $\pi\rtimes \psi$
is faithful.
\end{itemize}

\end{thm}
\begin{proof}
 The implication (ii)$\Rightarrow$(iii) is obvious as  $\FR(X)\subseteq \DR(\A)$. To see that (iii)$\Rightarrow$(i) follows because  condition \eqref{Coburn condition666} involves only elements lying in the image of the corresponding representation of $\FR(X)$. Hence if they are satisfied in
 $\DR(\A)$ they need to be satisfied in $\FR(X)$.
\end{proof}
\begin{rem}
When $X$ is compactly aligned, $P$ is a positive cone in a quasi-lattice ordered group  and all the fibers  $X_p$, $p\in P$, are essential,
then  $\FR(X)$ coincides with the algebra denoted by $B_P\rtimes_{\tau,X}P$ in \cite{F99}, see also \cite{FR}.  In this case the equivalence of (i) and (iii) in Theorem \ref{Uniqueness for product systems III}
 is \cite[Theorem 7.2]{F99}, which in turn is a generalization of \cite[Theorem 5.1]{FR}.
\end{rem}
\begin{ex}[Fowler-Raeburn version of $\OO_\infty$]\label{ex:FR-Oinfty}
Retain the notation of Example \ref{ex:DR-Oinfty}. We noticed there that conditions \eqref{eq:ranges dont span}
and \eqref{Doplicher-Roberts O_infty} are equivalent. Denoting by  $\mathcal{FR}(\OO_\infty)$ the Fowler-Raeburn algebra $\FR(X)$ associated to $(X_n)_{n\in \N}$  we see now, using Theorem \ref{Uniqueness for product systems III}, that  these equivalent conditions are further equivalent  to having an isomorphism
\begin{equation}\label{Fowler-Raburn O_infty}
	\mathcal{FR}(\OO_\infty)\cong \clsp\{ X_n Q_k X_m^*: n,m,k\in \N\}.
	\end{equation}
In particular, $\OO_\infty$ embeds as a proper subalgebra of $\mathcal{FR}(\OO_\infty)$ by \eqref{normal O_infty} and the second part of Proposition \ref{Fowler-Raeburn algebra form}, cf. \cite[Example 5.6(2)]{FR}. The algebra $\mathcal{FR}(\OO_\infty)$
 is separable while $\mathcal{DR}(\OO_\infty)$ is not.
\end{ex}
\section{Nica-Toeplitz crossed products by completely positive maps}\label{section:NT-cp-ccp-maps}

In this section we will introduce a general definition of Nica-Toeplitz $C^*$-algebra for an action of a LCM semigroup by completely positive maps.
We will do it in two steps. First we introduce a Toeplitz $C^*$-algebra, and then obtain a Nica-Toeplitz $C^*$-algebra as a quotient by ‘eliminating redundancies’.
In the next sections we will analyze these $C^*$-algebras in more detail in two special cases. Namely, when the action is   by endomorphisms or  by transfer operators.

\subsection{General construction}
Let  $P$  be a right LCM semigroup. Let $\CP(A)$ denote a semigroup of completely positive maps on a $C^*$-algebra $A$ (with semigroup operation given by composition).
\begin{defn}\label{C*-dynamical system}
Let $\varrho:P \ni p \mapsto \varrho_p\in \CP(A)$ be a unital semigroup antihomomorphism, i.e. $\alpha_e=id$ and $\varrho_q\circ \varrho_p=\varrho_{pq}$ for all $p,q\in P$. We call $(A,P,\varrho)$  a \emph{$C^*$-dynamical system}.
\end{defn}
A \emph{representation of the semigroup} $P$ in a Hilbert space $H$ is a unital semigroup homomorphism $S:P\to \B(H)$ into the multiplicative semigroup of  $\B(H)$.
The following is an obvious semigroup generalization of \cite[Definition 3.1]{kwa-exel}.
\begin{defn}
 A \emph{representation} of a $C^*$-dynamical system  $(A,P,\varrho)$ on a Hilbert space $H$ is a pair $(\pi, S)$ consisting of a nondegenerate representation $\pi:A\to \B(H)$ and a homomorphism  $S:P\to \B(H)$ such that
\begin{equation}\label{cp map representation relation}
S_p^*\pi(a)S_p=\pi(\varrho_p(a))
\end{equation}
for all $p,q\in P$ and $a\in A$. We put $C^*(\pi,S):=C^*(\bigcup_{p\in P} \pi(A)S_p)$. Exactly as  in the proof of \cite[Lemma 3.2]{kwa-exel},  one can prove that there is a universal representation $(i_A, \hat{t})$ of $(A,P,\varrho)$; universal in the  sense that for any  other representation $ (\pi,S)$ of $(A,P,\varrho)$
the maps
\begin{equation}\label{toeplitz epimorphism}
 i_A(a)\longmapsto \pi(a), \qquad i_A(a)\hat{t}_p \longmapsto \pi(a)S_p, \qquad a\in A,\,\, p\in P,
 \end{equation}
give rise to an epimorphism from $C^*(i_A,\hat{t})$ onto $C^*(\pi,S)$. Up to a natural isomorphism the $C^*$-algebra $\TT(A,P,\varrho):=C^*(i_A(A),\hat{t})$ is uniquely determined by $(A,P,\varrho)$, and we call it the \emph{Toeplitz algebra of $(A,P,\varrho)$}.
\end{defn}
For any $C^*$-algebra $C$ we denote by $\RM(C)$, $\LM(C)$ and $\M(C)$ the algebras of right, left and two-sided multipliers of $C$, respectively, cf. \cite[3.12]{pedersen}.
We say that a map $\varrho$ on  $C$ is \emph{strict} if for any approximate unit $\{\mu_\lambda\}$ in $C$, the net $\{\varrho(\mu_\lambda)\}$ converges strictly to a multiplier of $C$.
\begin{lem}\label{lem: Toeplitz algebra for completely positives}
We have $\TT(A,P,\varrho)=C^*(\bigcup_{p\in P} i_{A}(A)\hat{t}_p i_{A}(A))$. Hence $i_{A}(A)$ is a nondegenerate  subalgebra of $\TT(A,P,\varrho)$ and $\{\hat{t}_p\}_{p\in P}\subseteq \RM(\TT(A,P,\varrho))$. If every  $\varrho_p$, $p\in P$, is strict then  $\TT(A,P,\varrho)=C^*(\bigcup_{p\in P} \hat{t}_p i_{A}(A))$  and  $\{\hat{t}_p\}_{p\in P}\subseteq \M(\TT(A,P,\varrho))$.
\end{lem}
\begin{proof}
Suppose that $\TT(A,P,\varrho)$ acts in a nondegenerate  way on a Hilbert space $H$. By  \cite[Proposition 3.10 and Lemma 3.8]{kwa-exel}, for any $p\in P$ and an approximate unit $\{\mu_\lambda\}_{\lambda\in \Lambda}$ in $A$ we have
$\textrm{s-}\lim_{\lambda\in \Lambda} i_{A}(\mu_\lambda)\hat{t}_p=\hat{t}_p$  and  $i_{A}(A)\hat{t}_p \subseteq  i_{A}(A)\hat{t}_p i_{A}(A)$.
In particular, $\hat{t}_p\in \TT(A,P,\varrho)''$ and  $ \TT(A,P,\varrho)\subseteq C^*(\bigcup_{p\in P} i_{A}(A)\hat{t}_p i_{A}(A))$. The reverse inclusion $ C^*(\bigcup_{p\in P} i_{A}(A)\hat{t}_p i_{A}(A)) \subseteq \TT(A,P,\varrho)$ is clear since  $i_{A}(A)=i_{A}(A)\hat{t}_e\subseteq \TT(A,P,\varrho)$.  Thus $i_{A}(A)$ is a nondegenerate  subalgebra of $\TT(A,P,\varrho)$. Every $b\in\TT(A,P,\varrho)$ is of the form $b'i_{A}(a)$, where $b'\in\TT(A,P,\varrho)$, $a\in A$, and
$
b \hat{t}_p= b'i_{A}(a)\hat{t}_p\in b' i_{A}(A)\hat{t}_p i_{A}(A) \subseteq \TT(A,P,\varrho).
$
Hence $\hat{t}_p \in \RM(\TT(A,P,\varrho))$,  for every $p\in P$.

 Suppose now that  every map $\varrho_p$, $p\in P$, is strict. By \cite[Proposition 3.10 and Remark 3.9]{kwa-exel}, we get $\hat{t}_p i_{A}(A) \subseteq  i_{A}(A)\hat{t}_p i_{A}(A)$, for every $p\in P$. Using this, similarly as above, one gets that $\TT(A,P,\varrho)=C^*(\bigcup_{p\in P} \hat{t}_p i_{A}(A)) $ and $\hat{t}_p \in \LM(\TT(A,P,\varrho))$,  for every $p\in P$.
\end{proof}
Let $(\pi, S)$ be a representation of  $(A,P,\varrho)$. In view of \eqref{cp map representation relation} the Banach spaces
$$
\KK_{(\pi,S)}(p,q):=\overline{\pi(A)S_p\pi(A)S_q^*\pi(A)}, \qquad p,q\in P,
$$
form a  $C^*$-precategory. In general, it  is not obvious that there exists a  right-tensor $C^*$-precategory containing $\KK_{(\pi,S)}$ as an ideal. Nevertheless, we can  mimic the definition of Nica covariance  to define a Nica-Toeplitz algebra  as follows, where we also draw on inspiration from \cite{exel3}.
\begin{defn}\label{redundancy definition}
Let $(\pi, S)$ be a representation of  $(A,P,\varrho)$. We say that a pair  $(a \cdot b,k)$ is a\emph{ redundancy for $(\pi, S)$} if  $a\in \KK_{(\pi,S)}(p,q)$, $b\in   \KK_{(\pi,S)}(s,t) $  and
$k\in \KK_{(\pi,S)}(pq^{-1}r,ts^{-1}r)$, for some $p,q,s,t,r\in P$ with  $qP\cap sP=rP$, are such that
\begin{equation}\label{eq:redundancy condition}
ab \pi(c)S_{ts^{-1}r}=k \pi(c)S_{ts^{-1}r} \,\, \textrm{ for all }c \in A.
 \end{equation}
We say that $(\pi, S)$ is \emph{Nica covariant} if
\begin{itemize}
\item[(1)]  for every redundancy $(a \cdot b,k)$ we have $a \cdot b=k$;
\item[(2)]  $\KK_{(\pi,S)}(p,q) \KK_{(\pi,S)}(s,t)=\{0\}$ whenever $qP\cap sP=\emptyset$.
\end{itemize}
 \end{defn}
\begin{rem}\label{rem:uniqueness of redundancy}
Note that if $(a \cdot b,k)$ is a redundancy for $(\pi, S)$, then  $a \cdot b$ determines $k$ uniquely, via \eqref{eq:redundancy condition}. Indeed, just note that the essential subspace for $k$ is
 $$
\overline{\KK_{(\pi,S)}(ts^{-1}r,ts^{-1}r)H}=\overline{\pi(A)S_{ts^{-1}r}\pi(A)H}=\overline{\pi(A)S_{ts^{-1}r}H}.
$$
\end{rem}
We define the \emph{Nica-Toeplitz algebra of the $C^*$-dynamical system $(A,P,\varrho)$} to be the $C^*$-algebra
$\NT(A,P,\varrho):=C^*(j_A,\hat{s})$ generated by a universal Nica covariant representation $(j_A,\hat{s})$ of $(A,P,\varrho)$. As the next result shows, there is an alternate way to justify its existence.
\begin{lem}
We have $\NT(A,P,\varrho)\cong \TT(A,P,\varrho)/\NN$ where $\NN$ is the ideal of $\TT(A,P,\varrho)$ generated by the differences
$$
a\cdot b -k  \quad \text{ where } \quad (a \cdot b,k) \text{ is a redundancy for }(i_A,\hat{t})
$$
and products
$$
a\cdot b  \quad \text{ where }  a\in \KK_{(i_A,\hat{t})}(p,q),\,\, b\in   \KK_{(i_A,\hat{t})}(s,t) \text{ and } qP\cap sP=\emptyset.
$$
\end{lem}
\begin{proof}
It is straightforward and therefore left to the reader.
\end{proof}
We aim to investigate uniqueness of representations of $\NT(A, P, \rho)$. We will do this by specializing to two classes of actions where $\NT(A, P, \rho)$ admits realizations as a Nica-Toeplitz algebra of a right-tensor $C^*$-precategory.

\begin{rem}\label{well-aligned dynamical system}
Even though in the greatest generality of an action by completely positive maps there is no obvious structure of a right-tensor category, it is still possible to define a notion similar to well-aligned for $C^*$-precategory and show that it provides a structural description of $\NT(A, P, \rho)$ similar to the Nica-Toeplitz algebra of a $C^*$-precategory, cf.  \cite[Remark 3.9]{kwa-larI}. We include the details here for two reasons: first of all because the  classes of examples we consider exhibit this additional feature and second because we believe the observation may be of use in future investigations.

We say that a $C^*$-dynamical system $(A,P,\varrho)$ is \emph{well-aligned} if  for every representation  $(\pi,S)$ of $(A,P,\varrho)$ and all pairs $a \in \KK_{(\pi,S)}(p,q)$ and $b\in \KK_{(\pi,S)}(s,t)$  with $qP\cap sP=
rP$ there is  $k\in \KK_{(\pi,S)}(pq^{-1}r,ts^{-1}r)$ such that $(a \cdot b,k)$ is a redundancy for ${(\pi,S)}$ (obviously it suffices to check this requirement only for the universal  representation $(i_A,\hat{t})$).

We now claim that if a $C^*$-dynamical system $(A,P,\varrho)$ is well-aligned, then
$$
\NT(A,P,\varrho)=\clsp\{\bigcup_{p,q\in P} \KK_{(j_A,\hat{s})}(p,q) \}.
$$
Indeed, the Banach space $\clsp\{\bigcup_{p,q\in P} \KK_{(j_A,\hat{s})}(p,q) \}$ is closed under taking adjoints. Thus we only need to check that it is closed under multiplication. Let  $a \in \KK_{(j_A,\hat{s})}(p,q)$ and $b\in \KK_{(j_A,\hat{s})}(s,t)$. If $qP\cap sP=\emptyset$, then $a\cdot b= 0$ by Nica covariance of $(j_A,\hat{s})$.
 Assume then that   $qP\cap sP=rP$.
 By well-alignment there is $k\in \KK_{(j_A,\hat{s})}(pq^{-1}r,ts^{-1}r)$ such that $(a \cdot b,k)$ is a redundancy for $(j_A,\hat{s})$. Hence $a\cdot b= k\in \KK_{(j_A,\hat{s})}(pq^{-1}r,ts^{-1}r)$, again by Nica covariance.
\end{rem}

\subsection{Nica-Toeplitz crossed products by endomorphisms}\label{subsection:NT-cp-endo}
Throughout this subsection we
let  $P$  be a right LCM semigroup and denote by $\alpha:P \ni p \mapsto \alpha_p\in \End(A)$  a unital semigroup antihomomorphism, i.e. $\alpha_e=id$ and $\alpha_q\circ \alpha_p=\alpha_{pq}$ for all $p,q\in P$, where we assume that $\alpha_p$ is an endomorphism of $A$ for each $p\in P$. Since $*$-homomorphisms are completely positive maps,  $(A,P,\alpha)$  is a $C^*$-dynamical system in the sense of Definition \ref{C*-dynamical system}.

Earlier approaches to associating a Toeplitz-type crossed product to $(A,P,\alpha)$ involve a product system over $P$,  see e.g. \cite[Section 3]{F99}. Along the same lines, for each $p\in P$, let $E_p:=\alpha_p(A)A$  be the $C^*$-correspondence over $A$ where
$$
  \quad \langle x, y\rangle_p :=x^*y, \qquad
a\cdot x\cdot b:=\alpha_p(a)xb, \quad x,y \in E_p,\,\, a,b \in A.
$$
We define multiplication on $E_\alpha=\bigsqcup_{p\in P} E_p$  by
\begin{equation}\label{multiplication-Ealpha}
E_p\times E_q\ni(x,y)\longmapsto \alpha_q(x)y\in E_{pq}.
\end{equation}
It is readily seen that the above map induces an isomorphism $E_p\otimes E_q\cong E_{pq}$ and hence $E_\alpha$ is a product system, cf. \cite[Lemma 3.2]{F99}. The left action on each fiber is by generalized compacts, cf. \cite[Lemma 3.25]{kwa-exel}. Hence by Lemma \ref{non-degeneracy of K_X} we have a  right-tensor $C^*$-precategory $\KK_{E_\alpha}=\{\KK(E_p,E_q)\}_{p,q\in P}$ which is a well-aligned ideal in the $C^*$-precategory associated to $E$.

We describe next another right-tensor $C^*$-precategory from $(A,P,\alpha)$ which will be useful in proving   that our dynamical system is well-aligned. As in  \cite[Example 3.4]{kwa-doplicher},
  $\KK_\alpha:=\{\alpha_p(A)A\alpha_q(A)\}_{p,q\in P}$ is a $C^*$-precategory with multiplication, involution and norm inherited from $A$. There is a right tensoring on $\KK_\alpha$  given by
$$
\al_p(A)A\al_q(A)\ni a \longrightarrow a \otimes 1_r :=  \alpha_r(a) \in \al_{pr}(A)A\al_{qr}(A).
$$
Our first observation is that  $\KK_\alpha$ and $\KK_{E_\alpha}$ are the same, up to isomorphism.

\begin{lem}\label{isomorphism of categories for endomorphisms}
The right-tensor $C^*$-precategories  $\KK_{E_\alpha}$ and  $\KK_{\alpha}$ are isomorphic with
the isomorphism given by  $\KK(E_p,E_q) \ni\Theta_{x,y}\longmapsto xy^*\in \KK_{\alpha}(p,q)$.
\end{lem}
\begin{proof}
Let $x_i\in E_q$ and $y_i\in E_p$, for $i=1,\dots ,n$. Since $\sum_{i=1}^n x_iy_i^*\in \alpha_q(A)A\alpha_p(A)$ we have
$$
\|\sum_{i=1}^n x_iy_i^*\|=\sup_{y\in\alpha_p(A)A \atop\|y\|=1}\|\sum_{i=1}^n x_iy_i^*y\|=\sup_{y\in E_p\atop\|y\|=1}\|\sum_{i=1}^n x_i\langle y_i, y\rangle\|
=\|\sum_{i=1}^n \Theta_{x_i,y_i}\|.
$$
Thus $\KK(E_p,E_q) \ni\Theta_{x,y}\mapsto xy^*\in \KK_{\alpha}(p,q)$ extends to an isometric isomorphism, and straightforward calculations show that these maps form an isomorphism of $C^*$-precategories $\KK_{E_\alpha}$ and  $\KK_{\alpha}$. Further, the maps intertwine right tensoring because
for $x\in E_q$, $y, z\in E_p$, $w\in E_r$ we have $\alpha_r(x)\in E_{qr}$ and  $\alpha_r(y)\in E_{pr}$, thus
$$
(\Theta_{x,y}\otimes 1_r) (z \cdot w)=x \cdot \langle y,z\rangle_p\cdot w=\alpha_r(x)\alpha_r(y^*z)w=\Theta_{\alpha_r(x),\alpha_r(y)} z \cdot w.
$$
\end{proof}

\begin{lem}\label{properties of representations of endomorphisms}
Let $(\pi,S)$ be a representation of  $(A,P,\alpha)$ on a Hilbert space $H$.
\begin{itemize}
\item[(i)] For every $p\in P$,
$S_p$ is a partial isometry
and
$$
\pi(a) S_p=S_p \pi(\alpha(a)),\qquad \textrm{ for all }a\in A.
$$
In particular,  the projection $S_pS_p^*$ belongs to the commutant of $\pi(A)$ and
 $$
\KK_{(\pi,S)}(p,q)=S_p\pi\big(\alpha_p(A)A\alpha_q(A)\big)S_q^*,\qquad\text{ for all }p,q\in P;
$$
\item[(ii)] For any approximate unit $\{\mu_\lambda\}$ in $A$ and  all $p\in P$ we have $
S_p^*S_p=\text{s-}\lim_{\lambda\in \Lambda} \pi(\alpha_p(\mu_\lambda))$. In particular,
$\pi(a)S_p^*S_p=\pi(a)$ for all $a\in A\alpha_p(A)$;
\item[(iii)] The family of projections $\{S_p^*S_p\}_{p\in P^{op}}$ forms a decreasing net, that is
$$
q=tp \,\,\, \Longrightarrow \,\,\,  S_q^*S_q \leq S_p^*S_p;
$$
\item[(iv)] Let $a\in \KK_{(\pi,S)}(p,q)$, $b\in   \KK_{(\pi,S)}(s,t) $  and
$k\in \KK_{(\pi,S)}(pq^{-1}r,ts^{-1}r)$,  where $p,q,s,t,r\in P$ with  $qP\cap sP=rP$. The pair $(a\cdot b,k)$ is a redundancy if and only if
$$
k=S_{pq^{-1}r} \pi(\alpha_{q^{-1}r}(a_0)\alpha_{s^{-1}r}(b_0) )S_{ts^{-1}r}^*
$$
where $a_0\in\alpha_p(A)A\alpha_q(A)$ and $b_0\in\alpha_s(A)A\alpha_t(A)$ are such that $a=S_p\pi(a_0)S_q^*$, $b=S_s\pi(b_0)S_t^*$.
\end{itemize}
\end{lem}
\begin{proof}
Part (i)  follows from  \cite[Proposition 3.12]{kwa-exel}. Part (ii) follows from \eqref{cp map representation relation}, because $\pi$ is nondegenerate  and therefore   $\pi(\mu_\lambda)$ converges strongly to identity. To see part (iii) assume that  $q=tp$, and notice that  by part (ii) we have
\begin{align*}
(S_q^*S_q)(S_p^*S_p)&=\text{s-}\lim_{\lambda\in \Lambda}\text{s-} \lim_{\lambda'\in \Lambda} \pi\big(\alpha_q(\mu_\lambda) \alpha_p(\mu_\lambda')\big)= \text{s-}\lim_{\lambda\in \Lambda}\text{s-} \lim_{\lambda'\in \Lambda}\pi\big(\alpha_p(\alpha_t(\mu_\lambda)\mu_\lambda')\big)
\\
&=\text{s-}\lim_{\lambda\in \Lambda}\pi\big(\alpha_p(\alpha_t(\mu_\lambda))\big)=\text{s-}\lim_{\lambda\in \Lambda}\pi\big(\alpha_q(\mu_\lambda)\big)=S_q^*S_q.
\end{align*}
Let now $a,b$ and $k$ be as in part (iv). By part (i) there are $a_0\in\alpha_p(A)A\alpha_q(A)$ and $b_0\in\alpha_s(A)A\alpha_t(A)$  such that $a=S_p\pi(a_0)S_q^*$, $b=S_s\pi(b_0)S_t^*$. By Remark \ref{rem:uniqueness of redundancy}, condition \eqref{eq:redundancy condition} determines $k$ uniquely.
Thus  it suffices to show that for $k:=S_{pq^{-1}r} \pi(\alpha_{q^{-1}r}(a_0)\alpha_{s^{-1}r}(b_0) )S_{ts^{-1}r}^*$, the pair $(a\cdot b,k)$ is a redundancy. This follows from the following computation:
\begin{align*}
a\cdot b \pi(c)S_{ts^{-1}r}&
=\big(S_p\pi(a_0)S_q^*\big)  \big(S_s\pi(b_0)S_t^*\big)  \pi(c)S_{t}S_{s^{-1}r}
\stackrel{\eqref{cp map representation relation}}{=}S_p\pi(a_0)S_q^*  S_s\pi(b_0 \alpha_t(c))S_{s^{-1}r}
\\
&\stackrel{(i)}{=}S_p\pi(a_0)S_q^*  S_r\pi\big(\alpha_{s^{-1}r}(b_0 \alpha_t(c))\big)
=S_p\pi(a_0)S_q^* S_q S_{q^{-1}r}\pi\big(\alpha_{s^{-1}r}(b_0 \alpha_t(c))\big)
\\
&\stackrel{(ii)}{=}S_p\pi(a_0) S_{q^{-1}r}\pi\big(\alpha_{s^{-1}r}(b_0 \alpha_t(c))\big)
\\
&\stackrel{(i)}{=}S_{pq^{-1}r}\pi\big(\alpha_{q^{-1}r}(a_0)\alpha_{s^{-1}r}(b_0 \alpha_t(c))\big)
\\
&=
S_{pq^{-1}r}\pi\big(\alpha_{q^{-1}r}(a_0)\alpha_{s^{-1}r}(b_0)\big) \pi(\alpha_{ts^{-1}r}(c))\stackrel{\eqref{cp map representation relation}}{=}k \pi(c)S_{ts^{-1}r}.
\end{align*}
\end{proof}
\begin{prop}\label{representations of product systems for endomorphisms}
There are bijective correspondences between:
\begin{itemize}
\item[(i)]   representations $(\pi,S)$ of the $C^*$-dynamical system $(A,P,\alpha)$;
\item[(ii)] nondegenerate  right-tensor  representations $\Psi$ of  $\KK_\alpha$ on a Hilbert space;
\item[(iii)]  nondegenerate  representations $\psi$ of the product system  $E_\alpha$ on a Hilbert space.
\end{itemize}
Explicitly, these correspondences are determined   by
\begin{equation}\label{representation of category from endomorphisms}
 \Psi_{p,q}(a)=S_p\pi(a)S_q^*, \qquad\text{ for } a\in \KK_\alpha(p,q), \,\, p,q\in P
\end{equation}
  and
\begin{equation}\label{representation of system from endomorphisms}
 \psi_{p}(x)=S_p\pi(x), \qquad\text{ for } x\in E_{p}=\alpha_p(A)A \text{ and }  S_p=\text{s-}\lim_{\lambda\in \Lambda} \psi_{p}(\alpha_p(\mu_\lambda)),
\end{equation}
where $\{\mu_\lambda\}$ is an approximate unit in $A$ and  $p,q\in P$.  In particular, there are canonical isomorphisms
$
\TT(E_\alpha)\cong \TT(\KK_\alpha)\cong \TT(A,P,\alpha).
$
\end{prop}
\begin{proof}
By  \cite[Proposition 3.10]{kwa-exel}, modulo \cite[Lemma 3.25]{kwa-exel}, for each $p\in P$ the formula for $\psi_p$ in \eqref{representation of system from endomorphisms} yields a bijective correspondence between representations $(\psi_p,\psi_e)$ of the $C^*$-correspondence $E_p$ and representations $(\pi,S_p)$ of the single endomorphism $\alpha_p$.
Thus to establish the bijective correspondence between representations in (i) and (iii)  it suffices to check the equivalence of semigroup laws. Suppose that $\psi$ is given by  \eqref{representation of system from endomorphisms} for a representation  $(\pi,S)$ of  $(A,P,\alpha)$. By Lemma~\ref{properties of representations of endomorphisms},
$$
\psi_{p}(x)\psi_q(y)=S_p\pi(x)S_q\pi(y)=S_pS_q\pi(\alpha_q(x))\pi(y)=S_{pq}\pi( \alpha_q(x)y)=\psi_{pq}(x\cdot y).
$$
Hence, $\psi$ is a representation of $E_\alpha$. Conversely, if $\psi$ is a representation of $E_\alpha$ and $S$ is given by the strong limits in \eqref{representation of system from endomorphisms}, then
\begin{align*}
S_pS_q&=\text{s-}\lim_{\lambda\in \Lambda}\text{s-} \lim_{\lambda'\in \Lambda} \psi_{p}(\alpha_p(\mu_\lambda))\psi_{q}(\alpha_q(\mu_\lambda'))= \text{s-}\lim_{\lambda\in \Lambda}\text{s-} \lim_{\lambda'\in \Lambda}\psi_{pq}(\alpha_q(\alpha_p(\mu_\lambda)\mu_\lambda'))
\\
&=\text{s-}\lim_{\lambda\in \Lambda}\psi_{pq}(\alpha_q(\alpha_p(\mu_\lambda)))=S_{pq}.
\end{align*}
This proves the bijective correspondence between representations in (i) and (iii). By virtue of Lemma \ref{isomorphism of categories for endomorphisms}, the correspondence between representations in (ii) and (iii) is given by Corollary~\ref{going forward cor}. In particular, in view of
\eqref{representation of system from endomorphisms}, \eqref{precategory representation} translates to \eqref{representation of category from endomorphisms}.

\end{proof}
\begin{prop}\label{prop:Nica covariance of various representations} The system
$(A,P,\alpha)$ is well-aligned and the bijective correspondences in  Proposition~\ref{representations of product systems for endomorphisms} respect Nica covariance of representations. In particular,
$$
\NT(E_\alpha)\cong \NT(\KK_\alpha)\cong \NT(A,P,\alpha).
$$
Moreover, a representation $(\pi,S)$ of $(A,P,\alpha)$ is Nica covariant if and only if
$S$ is Nica covariant as a representation of $P$, i.e.:
\begin{equation}\label{Nica covariance for semigroups}
(S_pS_p^*)  (S_qS_q^*)=
\begin{cases}
S_r S_r^*, & \text{if } qP\cap sP=rP \text{ for some }r\in P,
\\
0, & \text{if } qP\cap sP=\emptyset.
\end{cases}
\end{equation}
\end{prop}
\begin{proof}
The first claim in the proposition follows from applying
 Proposition \ref{representations of product systems for endomorphisms}, Lemma \ref{properties of representations of endomorphisms} (iv),
  and Proposition \ref{going forward prop}, see \eqref{precategory representation}. Chasing universal properties will give the claimed isomorphisms of Nica-Toeplitz algebras.

To prove the last claim of the proposition, let  $(\pi,S)$, $\psi$  and $\Psi$ be in the correspondence described in  Proposition \ref{representations of product systems for endomorphisms}. Assume that $(\pi,S)$, and therefore also $\Psi$, is Nica covariant. Then for every $p,q\in P$ and
 $a\in \alpha_p(A)A\alpha_p(A)$, $b \in \alpha_q(A)A\alpha_q(A)$ we have
$$
S_p \pi(a)S_p^* S_q \pi(b)S_q^*=
\begin{cases}
S_r \pi(\alpha_{p^{-1}r}(a)\alpha_{q^{-1}r}(b))S_r^*, & \text{if } qP\cap sP=rP \text{ for some }r\in P,
\\
0, & \text{if } qP\cap sP=\emptyset.
\end{cases}
$$
Inserting in the above formula $a=\alpha_p(\mu_\lambda)$ and $b=\alpha_q(\mu_\lambda)$, where $\{\mu_\lambda\}$ is an approximate unit  in $A$, and passing to strong limit gives
$$
(S_pS_p^*)  (S_qS_q^*)=
\begin{cases}
S_r (S_{p^{-1}r}^*S_{p^{-1}r}) (S_{q^{-1}r}^*S_{q^{-1}r}) S_r^*, & \text{if } qP\cap sP=rP \text{ for some }r\in P,
\\
0, & \text{if } qP\cap sP=\emptyset,
\end{cases}
$$
 by Lemma \ref{properties of representations of endomorphisms} (ii). By Lemma \ref{properties of representations of endomorphisms}  we have
 $S_r (S_{p^{-1}r}^*S_{p^{-1}r})=S_r$ and  $(S_{q^{-1}r}^*S_{q^{-1}r})S_r^*=S_r^*$. Hence we get \eqref{Nica covariance for semigroups}.

Conversely, suppose that \eqref{Nica covariance for semigroups} holds. Let $a\in \alpha_p(A)A\alpha_p(A)$ and $b \in \alpha_q(A)A\alpha_q(A)$
for some $p,q\in P$. If $qP\cap sP=\emptyset$, then $S_p \pi(a)S_p^* S_q \pi(b)S_q^*=0$ because $S_p$ and $S_q$ have orthogonal ranges.
Assume that  $qP\cap sP=rP$. By appealing to Lemma~\ref{properties of representations of endomorphisms} (ii) and (i), we get
\begin{align*}
S_p \pi(a)S_p^* S_q \pi(b)S_q^*&
\stackrel{\eqref{Nica covariance for semigroups}}{=}S_p \pi(a)S_p^* S_r S_r^*S_q \pi(b)S_q^*
\\
&=S_p \pi(a)(S_p^* S_p)S_{p^{-1}r} S_{q^{-1}r}^*(S_q^* S_q) \pi(b)S_q^*
\\
&{=}S_p \pi(a)S_{p^{-1}r} S_{q^{-1}r}^*\pi(b)S_q^*
\\
&{=}S_r \pi(\alpha_{p^{-1}r}(a)\alpha_{q^{-1}r}(b))S_r^*.
\end{align*}
 This, in conjunction with Lemma \ref{isomorphism of categories for endomorphisms}, proves Nica covariance of $\psi$. Hence $(\pi,S)$ is Nica covariant, by the first part of the proposition.
\end{proof}
Let us notice that for $h\in P^*$ the endomorphism $\alpha_h$ is in fact an automorphism. Thus we have a group action $(P^*)^{op}\ni h \mapsto \alpha_h\in \Aut(A)$ of the opposite group to $P^*$. Recall, cf.  \cite[Definition 2.15]{KM}, that the group $\{\alpha_{h}\}_{h\in P^{*op}}$ of automorphisms of  $A$  is \emph{aperiodic} if
for every $h\in  P^*\setminus \{e\}$ and every non-zero hereditary subalgebra $D$ of $A$ we have
$
\inf \{\|\alpha_h(a) a\| : a\in D^+,\,\, \|a\|=1\}=0.
$

\begin{lem}\label{aperiodicity for endomorphisms} The group $\{\alpha_{h}\}_{h\in P^{*op}}$   is aperiodic if and only if the Fell bundle   $\{E_{\alpha_{h}}\}_{h\in P^*}$ is aperiodic if and only if the group of automorphisms $\{\otimes 1_{h}\}_{h\in P^*}$ of $\KK_\alpha$ is aperiodic.
\end{lem}
\begin{proof}
It is known, see \cite[Theorem 2.9]{KM}, that aperiodicity of $\{\alpha_{h}\}_{h\in P^*}$ is equivalent to the following condition: for every $h\in  P^*\setminus \{e\}$, every $b\in A$  and every non-zero hereditary subalgebra $D$ of $A$ we have
$
\inf \{\|\alpha_h(a)b a\| : a\in D^+,\,\, \|a\|=1\}=0.
$ The latter is exactly aperiodicity of $\{E_{\alpha_{h}}\}_{h\in P^*}$. Aperiodicity of $\{E_{\alpha_{h}}\}_{h\in P^*}$ is equivalent to aperiodicity of $\{\otimes 1_{h}\}_{h\in P^*}$ on $\KK_\alpha$, by Proposition \ref{lem: aperiodicity for product systems} and Lemma \ref{isomorphism of categories for endomorphisms}.
\end{proof}
We are now ready to  state the uniqueness theorem for Nica-Toeplitz crossed products associated to $(A,P,\alpha)$.

\begin{thm}[Uniqueness Theorem for   $\NT(A,P,\alpha)$]\label{Uniqueness Theorem for  crossed products by endomorphisms}
Let $(A,P,\alpha)$ be a $C^*$-dynamical system where each $\alpha_p$, $p\in P$, is an endomorphism,  and $P$ is  a right LCM semigroup. Suppose that  either
 $P^*=\{e\}$ or that the group $\{\alpha_{h}\}_{h\in P^{*op}}$ of automorphisms of  $A$ is aperiodic.
Assume moreover that $\KK_\alpha$ is amenable. Then for a Nica covariant representation
$(\pi,S)$ of $(A,P,\alpha)$, i.e. a representation satisfying \eqref{Nica covariance for semigroups}, the canonical epimorphism:
$$\NT(A,P,\alpha) \longrightarrow
\clsp\{S_p\pi(a)S_q^*: a\in \alpha_p(A)A\alpha_q(A), p,q\in P\}$$
is an isomorphism if and only if for any finite family $q_1,\ldots,q_n\in P\setminus P^*$ the representation
$A\ni a \mapsto  \pi(a) \prod_{i=1}^{n}(1-S_{q_i}S_{q_i}^*)$
 is faithful.
\end{thm}
\begin{proof} By Proposition \ref{prop:Nica covariance of various representations}, we may view $\NT(A,P,\alpha)$ as the Nica-Toeplitz algebra
$\NT(E_\alpha)$ of the compactly-aligned  product system $E_\alpha$, where the left action on each fiber is by compacts. Thus the assertion follows from Theorem \ref{Uniqueness Theorem for  product systems I} modulo Lemma \ref{aperiodicity for endomorphisms} and the observation that for a  representation $\Psi$ of $\KK_\alpha$ associated to $(\pi,S)$ we have, due to Lemma~ \ref{properties of representations of endomorphisms} (ii),
$
Q_p^\Psi H=\Psi_{p,p}(\KK_\alpha(p,p))H=S_p\pi(\KK_\alpha(p,p))S_p^*H=S_pS_p^*H.
$
\end{proof}

\subsection{Nica-Toeplitz crossed products by transfer operators}\label{subsection:NT-cp-transfer}
Throughout this section we assume that  $L:P \ni p \mapsto L_p\in \Pos(A)$ is a unital semigroup antihomomorphism taking values in the semigroup of positive maps on a $C^*$-algebra $A$. We additionally assume that for each $p\in P$  the map $L_p:A\to A$ admits a `multiplicative section', i.e. a $*$-homomorphism
$\alpha_p:A\to M(A)$ such that
\begin{equation}\label{transfer operator equality}
L_p(a\alpha_p(b))=L_p(a)b, \qquad a,b \in A.
\end{equation}
Thus $(A,\alpha_p,L_p)$ is a so called Exel-system and $L_p$ is a (generalized) \emph{transfer operator} for the endomorphism $\alpha_p$ \cite{exel3}, \cite{exel-royer}.  We emphasize that the choice of endomorphisms $\{\alpha_p\}_{p\in P}$ in general is far from being unique, cf. \cite{kwa-exel}. In particular, we do not assume that the family $\{\alpha_p\}_{p\in P}$ forms a semigroup.
Nevertheless, we show that we may associate to $(A,P,L)$  a product system  mimicking \cite{Larsen}. We also note that  \eqref{transfer operator equality} implies that each $L_p$ is not only positive but in fact a completely positive map, cf.  \cite[Lemma 4.1]{kwa-exel}.
Thus $(A,P,L)$ is a $C^*$-dynamical system in the sense of Definition \ref{C*-dynamical system}.

Let $p\in P$. The \emph{$C^*$-correspondence $M_{p}$ associated to the transfer operator} $L_p$ is  the completion  of the  space $A_{p}:=A$ endowed with a  right semi-inner-product $A$-bimodule structure  given by
$$
a\cdot  x \cdot b:=ax\alpha_p(b)\,\, \textrm{ for }\,\,a,b\in A \,\,\textrm{ and } \,\,\langle x, y \rangle_p:=L_p(x^*y)\,\, \textrm{ for all }x,y, a\in A_{p}.
$$
 The image of $x \in A_{p}=A$ in $M_{p}$ will be denoted by $(p, x)$.
\begin{rem}\label{remark_on_KSGNS_vs_transfers}
 By \cite[Lemma 4.4]{kwa-exel}, for each $p\in P$ the map $ (p,a\alpha_p(b)) \mapsto a\otimes_{L_p} b$, $a, b\in A$, determines an isomorphism of  $C^*$-correspondences from $M_p$ onto the KSGNS-correspondence $X_{L_p}$ of the completely positive map $L_p$ (\cite[Lemma 4.4]{kwa-exel} is stated under the assumption that $\alpha_p(A)\subseteq A$, but a quick inspection of the proof shows that this assumption is not needed, cf. also Lemma \ref{lemma before Product system for transfer operators} below).
\end{rem}
\begin{lem}\label{lemma before Product system for transfer operators}
Let $p,q\in P$. For any $a\in A$, $x\in A_{pq}$, and any approximate unit  $\{\mu_\lambda\}$  in $A$ the elements $(pq,x\alpha_p(\mu_\lambda))$ converge to $(pq,x)$ in $M_{pq}$.
\end{lem}
\begin{proof}
This follows from taking limits in the equalities
\begin{align*}
\|(pq,x\alpha_p(\mu_\lambda)-x)\|^2&=\|L_{pq}(\alpha_p(\mu_\lambda)x^*x\alpha_p(\mu_\lambda)- \alpha_p(\mu_\lambda)x^*x - x^*x\alpha_p(\mu_\lambda) +x^*x)\|
\\
&=\|L_{q}\Big( \mu_\lambda L_p(x^*x) \mu_\lambda- \mu_\lambda L_p(x^*x) - L_p(x^*x)\mu_\lambda +L_p(x^*x)\Big)\|.
\end{align*}
\end{proof}
\begin{prop}\label{Product system for transfer operators}
The disjoint union of $C^*$-correspondences $M_L=\bigsqcup_{p\in P} M_p$ is a product system over the semigroup $P$, with multiplication determined by
\begin{equation}\label{product system multiplication}
(p,x)(q,y):=(pq,x\alpha_p(y)), \qquad x,y\in A,\,\, p,q\in P.
\end{equation}
\end{prop}
\begin{proof}
For all $x,y,x',y'\in A,\,\, p,q\in P$ we have
\begin{align*}
\langle (p,x)\otimes_A (q,y),(p,x')\otimes_A (q,y') \rangle&=\langle  (q,y),\langle(p,x),(p,x')\rangle (q,y')\rangle
\\
&=L_{q}( y^*L_p(x^*x') y')= L_{pq}(\alpha_p(y^*)x^*x\alpha_p(y)\big)
\\
&=\langle (pq,x\alpha_p(y)),(pq,x'\alpha_p(y')) \rangle
\\
&=\langle (p,x)(q,y) ,(p,x')(q,y') \rangle.
\end{align*}
Thus we see that \eqref{product system multiplication} extends uniquely to a  multiplication $M_p\times M_q\to M_{pq}$ that factors trough to an isometric $C^*$-correspondence map $M_p\otimes_A M_q\to M_{pq}$, which is also surjective by Lemma \ref{lemma before Product system for transfer operators}.
 What is left to be shown is that multiplication \eqref{product system multiplication} is associative. To this end, we note that for any $x,y,z\in A$ and $p,q,r\in P$ we have
$$
\big( (p,x) (q,y)\big)(r,z)=(pqr,x\alpha_p(y)\alpha_{qp}(z)),\qquad  (p,x) \big( (q,y)(r,z)\big)=(pqr,x\alpha_p\big(y\alpha_{q}(z)\big)).
$$
Thus it suffices to show that $\|x\alpha_p(y)\alpha_{qp}(z)-x\alpha_p\big(y\alpha_{q}(z)\big)\|_{M_{pqr}}^2=0$.
This however follows from the transfer property of $L$ since for any $a\in A$ we have
\begin{align*}
L_{pqr}\left(ax\alpha_p\big(y\alpha_{q}(z\big)\right)&=L_{qr}\big(L_{p}(ax)y\alpha_{q}(z)\big)=L_{r}\Big(L_q\big(L_{p}(ax\alpha_p(y))\big) z\big)
\\
&=L_{r}\Big(L_{pq}\big(ax\alpha_p(y)\alpha_{qp}(z)\big)\Big)=L_{pqr}\big(ax\alpha_p(y) \alpha_{qp}(z)\big).
\end{align*}
\end{proof}

\begin{prop}\label{prop:representations of Exel systems and product systems}
Let $M_L$  be the product system constructed above.  We have a one-to-one correspondence between   representations $(\pi,S)$ of  $(A,P, L)$ and  nondegenerate  representations $\psi$ of $M_L$ on  Hilbert spaces, given by
\begin{equation}\label{associated correspondence representation}
\psi_p(p,x)= \pi(x)S_p,\qquad    x\in A,
\end{equation}
\begin{equation}\label{conjugate by limit}
S_p=\textrm{s-}\lim_{\lambda\in \Lambda} \psi_p(p,\mu_\lambda)
\end{equation}
where $\{\mu_\lambda\}_{\lambda\in \Lambda}$ is an approximate unit in $A$. For  the corresponding representations, we have $C^*(\pi,S)=C^*(\psi(M_L))$, and in particular
$
\TT(A,P, L)\cong \TT(M_L).
$
\end{prop}
\begin{proof} In view of Remark \ref{remark_on_KSGNS_vs_transfers},  it follows from \cite[Proposition 3.10]{kwa-exel}  that
 relations \eqref{associated correspondence representation} and \eqref{conjugate by limit} establish bijective correspondence between representations   $(\pi,S_p)$ of  $(A,L_p)$ and  nondegenerate  representations  $(\psi_e, \psi_p)$ of $M_p$. Thus  we only need to check the semigroup laws.  Assume  first that  $(\pi,S)$ is a representation of  $(A,L)$ and let $\psi$ be given by  \eqref{associated correspondence representation}. The isomorphism $M_p\cong X_{L_p}$ in Remark \ref{remark_on_KSGNS_vs_transfers} implies that $\pi(a\alpha_p(b))S_p=\psi_p(p,a\alpha_p(b))=\pi(a)S_p \pi(b)$ for any $a,b\in A$. Using this, for $x,y\in A, p,q \in P$, we get
\begin{align*}
\psi(p,x)\psi(q,y)&= \pi(x)S_p \pi(y) S_q=\pi(x \alpha_p(y)) S_{pq}
=\psi(pq,x\alpha_p(y)).
\end{align*}
Hence $\psi:M_L\to B(H)$ is a semigroup homomorphism.
\\
Now assume that $\psi:M_L\to B(H)$ is a representation. By \eqref{associated correspondence representation},  \eqref{conjugate by limit} and Lemma \ref{lemma before Product system for transfer operators} we have
\begin{align*}
S_p S_q &= \textrm{s-}\lim_{\lambda\in \Lambda} \Big( \textrm{s-}\lim_{\lambda'\in \Lambda'} \psi(p,\mu_\lambda) \psi(q,\mu'_{\lambda'})\Big)=\textrm{s-}\lim_{\lambda\in \Lambda} \Big( \textrm{s-}\lim_{\lambda'\in \Lambda'} \psi(pq,\mu_\lambda \alpha_p(\mu'_{\lambda'}))\Big)
\\
&=\textrm{s-}\lim_{\lambda\in \Lambda}  \psi(pq,\mu_\lambda)\Big)=\textrm{s-}\lim_{\lambda\in \Lambda}  \pi(\mu_\lambda) S_{pq}=S_{pq}.
\end{align*}
\end{proof}
\begin{rem} If for  $p\in P$ we may choose $\alpha_p$ taking values in $A$, then for each $p\in P$, $L_p$ extends to a strictly continuous map $\overline{L}_p:\M(A)\to \M(A)$, see \cite[Proposition 4.2]{kwa-exel}.
This implies that  the limit \eqref{conjugate by limit} may be taken in the strict topology of $\M(C^*(\pi,S))$ and so the  multiplier $S_p$ is determined by  $S_p\pi(a)=\psi_p(p, \alpha_p(a))$,  cf. \cite[Proposition 3.10]{kwa-exel} and Lemma \ref{lem: Toeplitz algebra for completely positives}.
\end{rem}
\begin{prop}\label{Nica-Toeplitz algebras for transfer operators vs product systems}
Suppose that the product system $M_L$ is  compactly aligned. The bijective correspondence in Proposition \ref{prop:representations of Exel systems and product systems} restricts to a bijective correspondence between Nica-Toeplitz representations  $(\pi,S)$ of  $(A,P, L)$ and nondegenerate  Nica-Toeplitz representations $\psi$ of $M_L$.  In particular,
$$
\NT(A,P, L)=\clsp\{j_A(a)\hat{s}_p\hat{s}_q^*j_A(b): a\in \alpha_p(A)A, b\in \alpha_q(A)A, p,q\in P\}\cong \NT(M_L).
$$
\end{prop}
\begin{proof}
Let $(\pi, S)$ be a representation of  $(A,P,L)$ and $\psi$ a representation of $M_L$ such that \eqref{associated correspondence representation} and \eqref{conjugate by limit} hold.  By Proposition \ref{going forward prop}, $\psi$ is Nica covariant if and only if the representation $\Psi:=\{\Psi_{p,q}\}_{p,q\in P}$ of $\KK_{M_L}$, given by \eqref{precategory representation},  is  Nica covariant. Note that for every $p,q\in P$, the map    $\Psi_{p,q}:\K(M_{q}, M_{p})\to \overline{\pi(A)S_p\pi(A)S_q^*\pi(A)}$,
 given by \eqref{precategory representation}, is surjective.
 Thus if
 $a\in \KK_{(\pi,S)}(p,q)$, $b\in   \KK_{(\pi,S)}(s,t) $ for some $p,q,s,t\in P$,   then there are $a'\in \KK_{M_L}(p,q)$, $b'\in   \KK_{M_L}(s,t) $ such that
 $\Psi_{p,q}(a')=a$ and $\Psi_{s,t}(b')=b$.     It suffices to show that if  $qP\cap sP=rP$, for some $r\in P$, then
$(a\cdot b, k)$, where $k=\Psi_{pq^{-1}r,ts^{-1}r}\Big((a' \otimes 1_{q^{-1}r}) (b'\otimes 1_{s^{-1}r})\Big)$, is a redundancy for $(\pi,S)$. 
For any $c\in A$ we have
\begin{align*}
 b\pi(c)S_{ts^{-1}r}&=b\pi(c)S_{t}S_{s^{-1}r} = b\psi_{t}((t,c))S_{s^{-1}r}=\psi_{s}\big(b'(t,c)\big)S_{s^{-1}r}
\\
& \stackrel{\eqref{conjugate by limit}}{=}
\textrm{s-}\lim_{\lambda\in \Lambda} \psi_{s}\big(b'(t,c)\big) \psi_{s^{-1}r}\big( (s^{-1}r,\mu_\lambda)\big)
 \\
 &=\textrm{s-}\lim_{\lambda\in \Lambda} \psi_{r}\big(b'(t,c)  (s^{-1}r,\mu_\lambda)\big)
 \\
 &=\textrm{s-}\lim_{\lambda\in \Lambda}\overline{\Psi}_{r,ts^{-1}r}\Big(b'\otimes 1_{s^{-1}r}\Big)\psi_{ts^{-1}r}\big((t,c)  (s^{-1}r,\mu_\lambda)\big)
 \\
 &=\textrm{s-}\lim_{\lambda\in \Lambda}\overline{\Psi}_{r,ts^{-1}r}\Big(b'\otimes 1_{s^{-1}r}\Big)
\pi(c)S_{t}\psi_{s^{-1}r}\big( (s^{-1}r,\mu_\lambda)\big)
  \\
 &=\overline{\Psi}_{r,ts^{-1}r}\Big(b'\otimes 1_{s^{-1}r}\Big)\pi(c)S_{ts^{-1}r}.
\end{align*}
Hence for any $c\in A$ we have $b\pi(c)S_{ts^{-1}r}=\overline{\Psi}_{r,ts^{-1}r}\Big(b'\otimes 1_{s^{-1}r}\Big)\pi(c)S_{ts^{-1}r}\in \psi_{r}\big(M_r)$.
Since, for any $x\in M_r$ we have $a \psi_{r}(x)=\overline{\Psi}_{pq^{-1}r,r}\Big(a'\otimes 1_{s^{-1}r}\Big)\psi_{r}(x)$, we conclude that
 for any $c\in A$ we have $a\cdot b\,\pi(c)S_{ts^{-1}r}=\Psi_{pq^{-1}r,ts^{-1}r}\Big((a' \otimes 1_{q^{-1}r}) (b'\otimes 1_{s^{-1}r})\Big)\pi(c)S_{ts^{-1}r}$.
 Hence $(a\cdot b, \Psi_{pq^{-1}r,ts^{-1}r}\Big((a' \otimes 1_{q^{-1}r}) (b'\otimes 1_{s^{-1}r})\Big))$ is a redundancy for $(\pi,S)$.

\end{proof}

Using the above result we can apply Theorem \ref{Uniqueness Theorem for  product systems I} to the product system $ M_L$ to get a uniqueness theorem for the Nica-Toeplitz crossed product
$\NT(A,P, L)$. Nevertheless, in this generality we can not simplify the assertion of Theorem \ref{Uniqueness Theorem for  product systems I} in a meaningful way. Therefore we will specialize to the case of `transfer operators of finite type'.

\subsection{Nica-Toeplitz crossed products by transfer operators of finite type}\label{subsection:NT-cp-transfer-finitetype}
As in the previous subsection, we let $L:P \ni p \mapsto L_p\in \Pos(A)$ be a unital semigroup antihomomorphism.
We recall that if $\varrho:A\to A$ is a positive map, then the \emph{multiplicative domain} of $\varrho$ is  the $C^*$-subalgebra of $A$
given by
$$
MD(\varrho):=\{a\in A: \varrho(b)\varrho(a)=\varrho(ba) \text{ and }\varrho(a)\varrho(b)=\varrho(ab)\text{ for every } b\in A\}.
$$

Throughout this subsection for every $p\in P$ we  make the following standing assumptions:
\begin{enumerate}
\item[(A1)] $L_p$  faithful, i.e. 
$L_p(a^*a)=0$ implies $a=0$;
\item[(A2)]  $L_p$ maps its multiplicative domain  onto $A$.
\end{enumerate}

We note that in the presence of axiom (A1), axiom (A2) is equivalent to the following two conditions:
\begin{enumerate}
\item[(A2a)]  there is an  endomorphism $\alpha_p:A\to A$ such that $L_p$ is a transfer operator for $\alpha_p$ as in \cite{exel3} and $\CE_p:=\alpha_p\circ L_p$ is a conditional expectation onto the range of $\alpha_p$;
\item[(A2b)]  $L_p(\mu_\lambda)$ converges strictly to $1\in \M(A)$, for any approximate unit $\{\mu_\lambda\}$ in $A$.
\end{enumerate}
Specifically, (A1) and (A2) imply that  $L_{p}|_{MD(L_p)}$ is a $*$-isomorphism onto $A$ and  its inverse:
\begin{equation}\label{endomorphism definition}
\alpha_p:=(L_{p}|_{MD(L_p)})^{-1}
\end{equation}
defines a monomorphism $\alpha_p$ with properties as in (A2a), cf. \cite[Propositions 4.16]{kwa-exel}. Then property (A2b) follows from
\cite[Propositions 4.13]{kwa-exel}, and  we also have $\|L_p\|=1$, cf. \cite[Lemma 2.1]{kwa-exel}.
Conversely, properties (A2a), (A2b) imply
(A2) by \cite[Propositions 4.16]{kwa-exel}, and then (A1) implies that $\alpha_p$ in (A2a)
has to be of the form \eqref{endomorphism definition}, see  \cite[Propositions 4.18]{kwa-exel}.
\begin{lem}
The maps in \eqref{endomorphism definition} form an action of $P$ by endomorphisms of $A$.
\end{lem}
\begin{proof}
We need to prove that $
\alpha_{pq}=\alpha_{p}\circ \alpha_q$  for all $p,q\in P$. We claim first that $L_p(MD(L_{pq}))\subseteq MD(L_q)$. Let $a\in A$ and  $b\in MD(L_{qp})$. Then
$$
L_q(L_p(b)a)=L_q(L_p(b\alpha_p(a)))=L_{pq}(b\alpha_p(a))=L_{pq}(b)L_{pq}(\alpha_p(a))=L_q(L_p(b))L_q(a),
$$
and similarly one gets $L_q(aL_p(b))=L_q(a)L_q(L_p(b))$.
\\
Secondly, we show that $MD(L_{pq})\subseteq MD(L_p)$. Indeed, since $L_{p}$ is a contractive completely positive map, we have
$$
MD(L_p)=\{a\in A: L_p(a^*)L_p(a)=L_p(a^*a) \text{ and } L_p(a)L_p(a^*)=L_p(aa^*)\},
$$
cf. \cite[Proposition 2.6]{kwa-exel}.
Now,  since $L_p(MD(L_{pq}))\subseteq MD(L_q)$ and $L_{pq}=L_q\circ L_p$, for any $a\in MD(L_{pq})$ we get
$$
L_q(L_p(a^*)L_p(a))= L_{pq}(a^*)L_{pq}(a)=L_{pq}(a^*a)=L_q(L_p(a^*a)).
$$
Faithfulness of $L_q$ implies that $L_p(a^*)L_p(a)=L_p(a^*a)$. Replacing $a$ with $a^*$, we get  $L_p(a)L_p(a^*)=L_p(aa^*)$.
Hence $MD(L_{pq})\subseteq MD(L_p)$.

Using the above inclusions, we conclude that $L_p$ restricts to a  monomorphism  $L_p:MD (L_{pq}) \to MD(L_q)$. In fact, since $L_{pq}=L_q\circ L_p$ restricts to an isomorphism from $MD(L_{pq})$  onto $A$ and $L_q$ restricts to an  isomorphism from $MD(L_q)$  onto $A$, we  see that
  $L_p:MD (L_{pq}) \to MD(L_q)$ is  an isomorphism. It is restriction of the isomorphism $L_p:MD (L_{p}) \to A$. Hence $(L_{p}|_{MD(L_{p})})^{-1}|_{MD(L_q)}=(L_{p}|_{MD(L_{pq})})^{-1}$. Thus we obtain
\begin{align*}
\alpha_{pq}&=(L_{pq}|_{MD(L_{pq})})^{-1}=(L_{q}|_{MD(L_{q})}\circ L_{p}|_{MD(L_{pq})})^{-1}=(L_{p}|_{MD(L_{pq})})^{-1}\circ (L_{q}|_{MD(L_{q})})^{-1}
\\
 &= (L_{p}|_{MD(L_{p})})^{-1}|_{MD(L_q)}\circ (L_{q}|_{MD(L_{q})})^{-1}=\alpha_p\circ \alpha_q.
\end{align*}
\end{proof}

For each $p\in P$,   $\CE_p=\alpha_p\circ L_p$ is a faithful conditional expectation onto the multiplicative domain $MD(L_p)=\alpha_p(A)$ of $L_p$. We will assume that each $\CE_p$ for $p\in P$ is of \emph{index-finite type} as in \cite{Wat}. Namely, for every $p\in  P$  we assume that
\begin{enumerate}
\item[(A3)] there is a finite  quasi-basis $\{u_{1}^p,...,u_{m_p}^{p}\}\subseteq A$  for $\CE_p$, for each $p\in P$, i.e. we have
\begin{equation}\label{finite-index equality}
a =\sum_{i=1}^{m_p}  u_i^p \CE_p((u_i^p)^*a),\quad \text{ for all }a\in A.
\end{equation}
\end{enumerate}
Associated to  $(A,P,L)$  we have the product system $M_L$ of Proposition \ref{Product system for transfer operators}.
Axiom (A1) implies that the map $A\ni a \mapsto (p,a)\in M_p$ is injective, for each $p\in P$.
Axiom (A3) implies that the left action of $A$ on each $M_p$, $p \in P$, is by compacts because for $a,x\in A$, $p\in P$, a simple calculation using \eqref{finite-index equality} gives
$$
\sum_{i=1}^{m_p}\Theta_{(p,u_i^p), (p,a^* u_i^p)}(p,x)=   (p,ax).
$$
Therefore the left action of $a$ on $M_p$ is given by the operator $\sum_{i=1}^{m_p}\Theta_{(p,u_i^p), (p,a^* u_i^p)}\in \K(M_p)$.
By Lemma \ref{non-degeneracy of K_X}, the ideal  $\KK_{M_L}$ in $\LL_{M_L}$ is invariant under right tensoring. Hence $\KK_{M_L}$ is a right-tensor  $C^*$-precategory itself.

Under the assumptions (A1)--(A3) it is possible to describe a new right-tensor $C^*$-precategory, isomorphic to $\KK_{M_L}$, but admitting an explicit formula for the right tensoring. With this at hand, after invoking Propositions ~\ref{Nica-Toeplitz algebras for transfer operators vs product systems} and \ref{going forward prop}, we give a more explicit characterization of  Nica covariance of a representation $(\pi, S)$ of $(A, P, L)$.

For each $p\in P$  denote by $\K_p$ the  \emph{reduced $C^*$-basic construction} associated to the conditional expectation $\CE_p$ cf. \cite[Subsection 2.1]{Wat}. Thus   $\K_p:=\K(\mathcal{E}_p)$ where $\mathcal{E}_p$ is the right Hilbert $\alpha_p(A)$-module obtained by completion of $A$ with  respect to the norm induced by the sesquilinear form $\langle x,y \rangle_{\alpha_p(A)}= \CE_p(x^*y)$, $x,y \in A$. We recall that there is an injective left action of $A$ on $\mathcal{E}_p$, induced by multiplication in $A$. Thus we  identify $A$ as a subalgebra of $\LL(\mathcal{E}_p)$. The operator $\CE_p:A\to A$ extends to an idempotent $e_p\in \LL(\mathcal{E}_p)$, and then
$$
\K_p=\clsp\{ a e_p b: a,b \in A\}.
$$
For each $p,q\in P$ we  equip
the algebraic tensor product $A\odot A$ with  the $\K_q$-valued sesquilinear form determined by
$$
\langle a\odot b, c\odot d\rangle_{p,q}:=b^* \alpha_q (L_p(a^*c))e_q d, \qquad a,b,c,d\in A.
$$
We let  $\K_L(p,q)$  be the Hilbert $\K_q$-module arising as the completion of $A\odot A$ with the semi-norm associated to the above sesquilinear form. We denote by $a\otimes_{p,q} b$ the image of a simple tensor $a\odot b$ in the space $\K_L(p,q)$.

\begin{prop}\label{C^*-category associated to transfer operators}
 The family of Banach spaces $\K_L:=\{\K_L(p,q)\}_{p,q\in P}$ defined above form a right-tensor $C^*$-precategory where
\begin{equation}\label{C*-category relations to be checked}
(a\otimes_{p,q} b)^*:=b\otimes_{q,p} a, \qquad (a\otimes_{p,q} b)\cdot  (c\otimes_{q,r} d)  :=a\alpha_p(L_q(bc))\otimes_{p,r} d,
 \end{equation}
 \begin{equation}\label{right tensoring to be checked}
   (a\otimes_{p,q} b)\otimes 1_r:= \sum_{i=1}^{m_r} a\alpha_p(u_{i}^{r}) \otimes_{pr,qr} \alpha_q(u_{i}^{r})^* b,
 \end{equation}
   for all $a,b,c,d \in A$, $p,q,r \in P$. Moreover, if $M_L$ is the product system associated to $L$, then the map
   \begin{equation}\label{iso of compacts}
   a\otimes_{p,q} b\longmapsto \Theta_{(p,a), \, (q,b^*)}, \qquad a,b \in A,
   \end{equation}
   establishes an isomorphism of right-tensor $C^*$-precategories from  $\K_L$ onto the right-tensor $C^*$-precategory $\K_{M_L}=\{\K(M_q,M_p)\}_{p,q\in P}$.
   \end{prop}
   \begin{proof} The strategy of the proof is to show that \eqref{iso of compacts} yields an isometric isomorphism $\K_L(p,q)\cong \K(M_q,M_p)$ under which the right-tensor $C^*$-precategory operations from $\K_{M_L}$ translate to the prescribed formulas for $\K_L$. To this end, note that for any $p, q\in P$ the maps
   $
   \HH:=\alpha_q\circ L_p  \textrm{ and } \VV:=\alpha_p\circ L_q
   $
     form an interaction in the sense of \cite[Definition 3.1]{exel-inter}. Indeed,  we have
     $
     \HH \circ \VV=\alpha_q \circ L_p\circ \alpha_p \circ L_q =\alpha_q \circ L_q =\CE_q,
     $
 and thus  $\HH \circ \VV \circ \HH= \CE_q\circ \HH=\HH$. For any $a,b\in A$ we get
  $$
     \HH(\VV(a)b)=\alpha_q\Big(L_p\big(\alpha_p(L_q(a))b\big)\Big)=\alpha_q\big(L_q(a)\big) \cdot \alpha_q\big(L_p(b)\big)=\HH(\VV(a))\HH(b).
 $$
 The other relations follow by symmetric arguments.
 Now, by \cite[Proposition 5.4]{exel-inter},  for $x=\sum_{i=1}^{n}a_i\otimes_{p,q} b_i$, $a_i,b_i\in A$, 
 we have
\begin{align*}
\|x\|_{\K(p,q)}&=\|[\HH(a_i^*a_j)]_{i,j}^{\frac{1}{2}} [\HH(\VV(b_ib_j^*))]_{i,j}^{\frac{1}{2}}\|_{M_n(A)}
=\|[\alpha_q(L_p(a_i^*a_j)]_{i,j}^{\frac{1}{2}} [\alpha_q(L_q(b_ib_j^*))]_{i,j}^{\frac{1}{2}}\|_{M_n(A)}.
\end{align*}
Using the fact that $\alpha_q$ amplifies to an isometric $*$-homomorphism on $M_{n}(A)$ we get
\begin{align*}
\|x\|_{\K(p,q)}&=\|[L_p(a_i^*a_j)]_{i,j}^{\frac{1}{2}} [L_q(b_ib_j^*)]_{i,j}^{\frac{1}{2}}\|_{M_n(A)}
\\
&=\|[\langle (p, a_i), (p,a_j)\rangle_p]_{i,j}^{\frac{1}{2}} [\langle (q,b_i^*), (q,b_j^*)\rangle_q]_{i,j}^{\frac{1}{2}}\|_{M_n(A)}.
\end{align*}
Comparing this with the norm of the operator $\sum_{i=1}^{n}\Theta_{(p,a_i), (q,b_i^*)}$ described in \cite[Lemma 2.1]{KPW}, we finally arrive at
$
\|\sum_{i=1}^{n}a_i\otimes_{p,q} b_i\|_{\K_L(p,q)}=\|\sum_{i=1}^{n}\Theta_{(p,a_i), (q,b_i^*)}\|_{\K(M_q,M_p)}.
$
Thus \eqref{iso of compacts} defines a linear isometry.

The standard formulas: $\Theta_{x,y}^*=\Theta_{y,x}$,  $\Theta_{x,y}\circ \Theta_{z,v}=\Theta_{x\langle y, z\rangle_q, v}$ for $x\in M_p$, $y,z\in M_q$, and $v\in M_r$, translate via \eqref{iso of compacts} to \eqref{C*-category relations to be checked}. Hence relations \eqref{C*-category relations to be checked} indeed define a $C^*$-precategory structure on $\K_L$, and $\K_L$ is isomorphic
to $\K_{M_L}$ as a $C^*$-precategory. Thus it remains to show that the right tensoring in $\K_{M_L}$ translates to \eqref{right tensoring to be checked} on the level of $\K_L$.

Note that the product system  $M_L$ is (left) essential.  Let $a,b,x,y\in A$, $p,q,r\in P$, and $T=\Theta_{(p,a), \,(q, b^*)} $. Taking into account that
$(p,x)\otimes_A (q,y)=(pq,x\alpha_p(y))$ we get
\begin{align*}
(T\otimes 1_r) (q,x)\otimes_A (r,y)&
=\big(p, a\alpha_p(L_q(bx))\big)\otimes_A (r,y)=\Big(pr,a\alpha_p\big(L_q(bx)y\big)\Big)\\
&= \Big(pr,a\alpha_p\Big(\sum_{i=1}^{m_r} u_i^r (\alpha_r\circ L_r)\big((u_i^r)^*L_q(bx\alpha_q(y))\big) \Big)\Big)\\
&=\Big(pr,\sum_{i=1}^{m_r} a\alpha_p(u_i^r) \alpha_{pr}\Big(L_{qr}\big(\alpha_q(u_i^r)^*bx\alpha_q(y)\big)\Big)\Big)\\
&=\Big(pr,\Big(\sum_{i=1}^{m_r} \Theta_{(pr,a\alpha_p(u_i^r)),(qr, b^*\alpha_q(u_i^r))}\Big)  x\alpha_q(y) \Big)
\\
&=\Big(\sum_{i=1}^{m_r} \Theta_{(pr,a\alpha_p(u_i^r)), (qr,b^*\alpha_q(u_i^r))}\Big)  (q,x)\otimes_A (r,y).
\end{align*}
Thus $\Theta_{(p,a), (q, b^*)}=\Big(\sum_{i=1}^{m_r} \Theta_{(pr,a\alpha_p(u_i^r)), (qr,b^*\alpha_q(u_i^r))}\Big) $
 and therefore \eqref{right tensoring to be checked} defines the desired right tensoring on $\K_L$.
\end{proof}
\begin{rem}\label{Not important remark}
   In view of Proposition \ref{C^*-category associated to transfer operators},  since  the image of  $A\alpha_p(A)$ in $M_p$ is a dense subspace of $M_p$, cf. Lemma \ref{lemma before Product system for transfer operators}, we have that
  $
  \KK_L(p,q)=\clsp\{a\otimes_{p,q} b: a\in A\alpha_p(A), b\in \alpha_q(A)A\}  $.
\end{rem}
\begin{lem}\label{relation for Nica covariance of transfers} Let
$pP\cap qP=rP$. For any $a,b,c,d\in A$ we have
\begin{equation}\label{eq:product in K_L}
\bigl((a\otimes_{p,p} b)\otimes 1_{p^{-1}r}\bigr) \cdot \bigl((c\otimes_{q,q} d)\otimes 1_{q^{-1}r}\bigr)= \sum_{i=1}^{m_{q^{-1}r}} a\CE_p\bigl(bc\alpha_q(u_i^{q^{-1}r})\bigr) \otimes_{r,r} \alpha_q(u_i^{q^{-1}r})^* d.
\end{equation}
\end{lem}
\begin{proof}  By \eqref{C*-category relations to be checked} and  \eqref{right tensoring to be checked}, the left hand side of \eqref{eq:product in K_L} is equal to
$$
\sum_{i=1}^{m_{p^{-1}r}} \bigl(a\alpha_p(u_{i}^{p^{-1}r}) \otimes_{r,r} \alpha_p(u_{i}^{p^{-1}r})^* b\bigr) \cdot \sum_{i=1}^{m_{q^{-1}r}} \bigl(c\alpha_q(u_{i}^{q^{-1}r}) \otimes_{r,r} \alpha_q(u_{i}^{q^{-1}r})^* d\bigr)
$$
$$
\,\,\,\,\,\,\,\,\,\,\, =\sum_{i=1,j=1}^{m_{p^{-1}r}, m_{q^{-1}r}} a\alpha_p(u_{i}^{p^{-1}r}) \CE_{r} \Big(\alpha_p(u_{i}^{p^{-1}r})^* b  c\alpha_q(u_{j}^{q^{-1}r})\Big) \otimes_{r,r} \alpha_q(u_{i}^{q^{-1}r})^* d.
$$
However,  using that $\CE_r=\alpha_p\circ \CE_{p^{-1}r} \circ L_p$, for any $f\in A$, it follows that
\begin{align*}
\sum_{i=1}^{m_{p^{-1}r}} \alpha_p (u_{i}^{p^{-1}r}) \CE_{r} \Big(\alpha_p(u_{i}^{p^{-1}r})^* f\Big)&=\sum_{i=1}^{m_{p^{-1}r}} \alpha_p\Big(u_{i}^{p^{-1}r} \CE_{p^{-1}r} \big(u_{i}^{p^{-1}r*} L_p( f)\big)\Big)
\\
&= \alpha_p\big( L_p( f)\big)=\CE_p(f).
\end{align*}
Now inserting $f=b  c\alpha_q(u_{j}^{q^{-1}r})$ in the  computations above gives the assertion.
\end{proof}
 \begin{prop}\label{Nica covariance for transfers characterised}
A representation  $(\pi,S)$ of  $(A,P, L)$  is Nica covariant if and only if for every $p,q\in P$, and $a\in \alpha_p(A)A\alpha_q(A)$ the following are satisfied:
\begin{equation}\label{orthogonality for Nica for transfers}
S_p^*\pi(a)S_q = 0 \quad\text{ if }\quad pP\cap qP=\emptyset
\end{equation}
and
\begin{equation}\label{eq:Nica covariance as Wick ordering}
S_pS_p^*\pi(a)S_qS_q^*
=\sum_{i=1}^{m_{q^{-1}r}} \pi\big(\CE_p(a\alpha_q(u_i^{q^{-1}r}))\big)S_r S_r^*\pi(u_i^{q^{-1}r})^* \quad\text{ if } \quad pP\cap qP=rP.
\end{equation}
\end{prop}
\begin{proof}
By  Propositions \ref{going forward cor}, \ref{Nica-Toeplitz algebras for transfer operators vs product systems} and \ref{C^*-category associated to transfer operators}, there is a one-to-one correspondence between  representations  $(\pi,S)$ of  $(A,P, L)$ and  right-tensor representations $\Psi$ of $\K_L$ determined by
\begin{equation}\label{correspondence of representations for transfer operators}
 \Psi(a\otimes_{p,q} b)=\pi(a)S_p S_q^*\pi(b), \qquad a,b\in A.
 \end{equation}
Moreover, $(\pi,S)$ is Nica covariant if and only if $\Psi$ is, cf. Proposition~\ref{going forward prop}.

Let $(\pi,S)$ be a Nica covariant representation. Thus the associated $\Psi$ in \eqref{correspondence of representations for transfer operators} is Nica covariant, too.
Let $p,q\in P$, and $a\in \alpha_p(A)A\alpha_q(A)$. If $pP\cap qP=\emptyset$, then for any $b,c,d\in A$ we have
$$
0=\Psi(b\otimes_{p,p} a) \Psi(c\otimes_{q,q} d)=\pi(b) S_pS_p^*\pi(a c)S_qS_q^*\pi(d).
$$
Letting $b,c,d\in A$ run through an approximate unit in $A$ and taking (strong) limit, we get $S_pS_p^*\pi(a)S_qS_q^* = 0$, which is equivalent to $S_p^*\pi(a)S_q = 0$.  Let now $pP\cap qP=rP$.  Invoking Lemma \ref{relation for Nica covariance of transfers}, with the roles of $a$ and $b$ exchanged, we get
$$
\pi(b)S_pS_p^*\pi(ac)S_qS_q^* \pi(d)
=\Psi\big((b\otimes_{p,p} a) \,\cdot\, (c\otimes_{q,q} d)\big)
= \Psi\Big(((b\otimes_{p,p} a)\otimes 1_{p^{-1}r}) \cdot ((c\otimes_{q,q} d)\otimes 1_{q^{-1}r}) \Big)
$$
$$
=\pi(b)\sum_{i=1}^{m_{q^{-1}r}} \pi\big(\CE_p(ac\alpha_q(u_i^{q^{-1}r}))\big)S_r S_r^*\pi(u_i^{q^{-1}r})^*\pi(d).
$$
Letting $b,c,d\in A$ run through an approximate unit in $A$ and taking (strong) limit gives \eqref{eq:Nica covariance as Wick ordering}.

Conversely, relations \eqref{orthogonality for Nica for transfers} and \eqref{eq:Nica covariance as Wick ordering} imply Nica covariance as a reversal of the above arguments shows.
\end{proof}
For $h\in P^*$, the mapping  $L_h$ is invertible and  hence by  (A1) and (A2) we have  $MD(L_h)=A$,
 which means that $L_h$ is an automorphism of $A$: we have $L_h=\alpha_{h^{-1}}= \alpha_{h}^{-1}$.
\begin{thm}[Uniqueness Theorem for   $\NT(A,P,L)$]\label{Uniqueness Theorem for  crossed products by transfers}
Let $(A,P,L)$ be a $C^*$-dynamical system satisfying (A1), (A2), (A3) above.  Suppose that  either
$P^*=\{e\}$ or that the group $\{\alpha_{h}\}_{h\in P^*}$ of automorphisms of  $A$ is aperiodic.
Assume also that $\KK_L$ is amenable. Let $(\pi,S)$  be a Nica covariant representation
of $(A,P,L)$. For each $p\in P$, let $Q_p$ be the projection onto the space $\overline{\pi(A)S_pH}$. Then the canonical surjective $*$-homomorphism
$$\NT(A,P,\alpha) \longrightarrow
\clsp\{\pi(a)S_pS_q^*\pi(b): a,b\in A, p,q\in P\}$$
is an isomorphism if and only if for all finite families $q_1,\dots ,q_n\in P\setminus P^*$ the representation
$A\ni a \mapsto  \pi(a) \prod_{i=1}^{n}(1-Q_{q_i})$
 is faithful.
\end{thm}
\begin{proof} Note that $\{\pi(a)S_pS_q^*\pi(b): a,b\in A, p,q\in P\}$ is closed under multiplication due to Proposition ~\ref{Nica covariance for transfers characterised}. By Proposition \ref{Nica-Toeplitz algebras for transfer operators vs product systems},  $\NT(A,P,L)$ may be viewed as $\NT(M_L)$. For $h\in P^*$, the automorphism  $L_h$ of $A$ induces an isomorphism of $C^*$-correspondences $E_{\alpha_{h}}=E_{L_{h^{-1}}}\cong M_{L_{h}}$.  Thus by Lemma \ref{aperiodicity for endomorphisms}, the group $\{\alpha_{h}\}_{h\in P^*}=\{L_h\}_{h\in P^{*op}}$ of automorphisms of  $A$ is aperiodic if and only if the Fell bundle   $\{M_{L_{h}}\}_{h\in P^*}$ is aperiodic. Recalling that under assumption (A3), the left action is by generalized compacts in each fiber, the assertion follows from Theorem \ref{Uniqueness Theorem for  product systems I} modulo Proposition \ref{C^*-category associated to transfer operators} and  the fact that for the   representation $\Psi$ of $\KK_L$ associated to $(\pi,S)$ and every $p\in P$ we have
$
\Psi_{p,p}(\KK_L(p,p))H=\pi(A)S_pS_p^*\pi(A)H=\overline{\pi(A)S_pH}.
$

\end{proof}
\section{$C^*$-algebras associated to right LCM semigroups}\label{section:semigroupCstar alg}

Throughout this section we use the notation $S$ for a generic right LCM semigroup and reserve $P$ for semigroups in semidirect products as in \cite{bls2}. Associated to any right LCM semigroup $S$ there is a universal $C^*$-algebra $C^*(S)=\overline{\operatorname{span}}\{v_s^{\phantom{*}}v_t^*: s,t\in S\}$  generated by an isometric representation $v$ of $S$ such that
$$
(v_pv_p^*) (v_q v_q^*)
=\begin{cases}
v_rv_r^*
& \text{if $pP\cap qP=rP$} \\
0 &\text{if $pP\cap qP=\emptyset$}.
\end{cases}
$$ See \cite{Li} for the abstract construction of $C^*(S)$ valid for arbitrary left cancellative semigroups, and \cite{bls} or \cite{No0}  for the case of right LCM semigroups. When $S$ is a right LCM semigroup,
 \cite[Corollary 7.11]{bls2} implies that $C^*(S)$ is isomorphic to the Nica-Toeplitz algebra $\NT(X)$ for the compactly aligned  product system $X$ over $S$ with fibers $X_s\cong\mathbb{C}$  for all $s\in S$. It is therefore natural to ask if Theorem~\ref{Uniqueness Theorem for  product systems I} can be applied. On one hand, since the left action is by generalized compact operators in every fiber $X_s$, for $s\in S$, we will have equivalence of the three assertions (i)-(iii) if  $\KK_X$ is amenable and $S^*=\{e\}$. This in particular recovers the case (1) in \cite[Theorem 4.3]{bls}. On the other hand, in case that $S^*\neq \{e\}$, the Fell bundle $\{X_h\}_{h\in S^*}$ can never be aperiodic because $X_h=\mathbb{C}$ for all $h\in S^*$. Therefore viewing $C^*(S)$ as a Nica-Toeplitz $C^*$-algebra associated to the   product system $X$ with trivial (thus small) fibers we can not  apply Theorem~\ref{Uniqueness Theorem for  product systems I} when $S^*\neq \{e\}$. A possible solution  to this obstacle is to consider $C^*(S)$ as a Nica-Toeplitz $C^*$-algebra  associated to another product system with larger fibers, so that we can detect aperiodicity. For instance, given a controlled function into a group one could obtain a uniqueness result in terms of Fell bundles based on  Proposition \ref{Abstract uniqueness}. However, in general the fibers of the arising Fell bundle will be very large. We propose an intermediate approach in the case that $S$
is a semidirect product of an LCM semigroup and a group. We will show useful alternative realizations of $C^*(S)$ as Nica-Toeplitz algebras associated to  product systems with larger fibers over a smaller semigroup, which lead to efficient uniqueness results.

\subsection{Semidirect products of LCM semigroups}
Even though left and right semidirect products are equivalent as abstract constructions, it turns out that right semidirect products have rather different properties than the left semidirect products of a group by a semigroup considered for instance in \cite{bls, bls2}. To exemplify, we will use right semidirect products for actions of semigroups on groups to construct right LCM semigroups $S$ with non-trivial group of units $S^*$ and which are not necessarily right cancellative (see Proposition \ref{construction of LCM's} below). Moreover, for these examples the  constructible right ideals depend only the acting semigroup, unlike the case of left semidirect products.

We begin by fixing our conventions for the two constructions. For a semigroup $T$ we let $\End T$ denote the semigroup of all semigroup homomorphisms $T\to T$ that preserve the identity $e_T$ in $T$. The identity endomorphism in $\End T$ is $\id_T$. A \emph{left action} $P\stackrel{\theta}{\curvearrowright} T$ of a semigroup $P$ on $T$ is a unital semigroup homomorphism $\theta:P \to \End T$. A \emph{right action} $T\stackrel{\vartheta}{\curvearrowleft} P$ is a unital semigroup antihomomorphism $\vartheta:P \to \End T$,  i.e. $\vartheta_p\vartheta_q=\vartheta_{qp}$
for all $p,q\in P$.

\begin{defn}\label{semidirect products definition} Let $T, P$ be semigroups.
The  \emph{(left) semidirect product} of $T$ by $P$ with respect to a left action $P
\stackrel{\theta}{\curvearrowright} T$, denoted
$T{\rtimes_\theta}P$, is the semigroup $T\times P$ with composition given by
$$
(g,p)(h,q) = (g\theta_{p}(h),pq),\qquad \text{ for }g,h\in T \text{ and }p, q\in P.
$$
The  \emph{(right) semidirect product} of $T$ by $P$ with respect to a right action $T\stackrel{\vartheta}{\curvearrowleft} P$, denoted
$P{_\vartheta\ltimes}T$, is the semigroup $P\times T$ with composition given by
$$
(p,g)(q,h) = (pq, \vartheta_{q}(g)h), \qquad \text{ for }g,h\in T \text{ and }p, q\in P.
$$
\end{defn}
\begin{rem}
The opposite semigroup $P^{op}$ to a semigroup $P$  coincides with $P$ as a set but has  multiplication  defined by reversing the factors. Treating the corresponding endomorphisms as maps on the same set, we have $\End P=\End P^{op}$. Thus every right action $T\stackrel{\vartheta}{\curvearrowleft} P$ can be treated as the left action $P^{op} \stackrel{\vartheta}{\curvearrowright} T^{op}$, and there is an isomorphism of semigroups
$$
P{_\vartheta\ltimes}T\ni (p,g) \to (g,p) \in \left(T^{op}{\rtimes_\vartheta}P^{op}\right)^{op}.
$$

\end{rem}
The following proposition should be compared with \cite[Lemma 2.4]{bls} proved for left semidirect products. It shows that right semidirect products in the realm of  right LCM semigroups are always left cancellative and have easier structure of principal right ideals.
\begin{prop}\label{construction of LCM's}
Suppose that  $G\stackrel{\vartheta}{\curvearrowleft} P$ is a right action of a right LCM semigroup $P$ with identity on a group $G$. Then $P{_\vartheta\ltimes}G$ is a right LCM semigroup such that
$$
J(P)\cong J(P{_\vartheta\ltimes}G) \quad \textrm{ and } \quad  (P{_\vartheta\ltimes}G)^*=P^*{_\vartheta\ltimes}G.
$$
Moreover, $P{_\vartheta\ltimes}G$ is cancellative if and only if $P$ is cancellative and every  $\vartheta_p$ is injective, $p\in P$.
\end{prop}
\begin{proof} The element $(e_P,e_G)$ is the identity of $P{_\vartheta\ltimes}G$.
If $(p,g)(q,h)=(p,g)(q',h')$ then $pq=pq'$ and $\vartheta_{q}(g)h=\vartheta_{q'}(g)h'$. By left cancellation in $P$ we get $q=q'$ and therefore also $h=h'$. Thus $P{_\vartheta\ltimes}G$ is left cancellative. Since the action of the group $G$ on itself is transitive, we have $(p,g)(P{_\vartheta\ltimes}G)=(pP)\times G$ for every $(p,g)\in P{_\vartheta\ltimes}G$. Hence $P{\ltimes_\vartheta}G$ is a right LCM semigroup with the semilattice of principal right ideals isomorphic to that of $P$, with isomorphism given by $
pP\mapsto (pP)\ltimes G
$. Plainly, relations $(p,g)(q,h)= (q,h)(p,q)=(e_P,e_G)$ hold if and only if $p\in P^*$, $q=p^{-1}$ and $h=\vartheta_{p^{-1}}(g^{-1})$. This immediately gives $(P{_\vartheta\ltimes}G)^*=P^*{_\vartheta\ltimes}G$.

The claim about $P{_\vartheta\ltimes}G$ being right cancellative follows by noting that
 $
(p,g)(q,h)=(p',g')(q,h)$  if and only if $pq=p'q$  and $\vartheta_{q}(g)=\vartheta_{q}(g')$.
\end{proof}
\begin{rem}
Using the last part of Proposition \ref{construction of LCM's} it is easy to construct examples of not cancellative LCM semigroups from cancellative ones.
\end{rem}
In general the (left) semidirect product of a group $G$  by a right LCM semigroup $P$ with respect to a left action $P
\stackrel{\theta}{\curvearrowright} G$ is not an LCM semigroup. As introduced in \cite[Definition 2.1]{bls2},  an \emph{algebraic dynamical system} is a triple $(G, P,\theta)$ where $G$ is a group, $P$ is a right LCM semigroup, and $P
\stackrel{\theta}{\curvearrowright} G$ is  a left action  by injective endomorphisms of  $G$ which respects the order, i.e. $\theta_p(G)\cap \theta_q(G)=\theta_r(G)$ whenever for $p,q\in P$ there is $r\in P$ such that $pP\cap qP=rP$. By \cite[Proposition 8.2 and Lemma 2.4]{bls}, whenever   $(G, P,\theta)$ is an algebraic dynamical system,  the left semidirect product $P{\rtimes_\theta}G$ is a right LCM semigroup
and $(P{\rtimes_\theta}G)^*=P^*{\rtimes_\theta}G$.

\subsection{Semigroup $C^*$-algebras associated to right semidirect products $P{_\vartheta\ltimes}G$}\label{subsection:semigroupCstar-right}
 Here we assume that ${\vartheta}$ is a right action  of a right LCM semigroup $P$ on a group $G$. We let $\delta_g$ for $g\in G$ be the generating unitaries in $C^*(G)$.
\begin{prop}\label{prop:right semidirect products}
Let ${\vartheta}$ be a right action  of a right LCM semigroup $P$ on a group $G$. There is an antihomomorphism $\alpha$ of $P$ into $\operatorname{End}C^*(G)$ given by $\alpha_p(\delta_g)=\delta_{\vartheta_p(g)}$  for $g\in G$ and $p\in P$. Further, $(C^*(G), P,\alpha)$ is a $C^*$-dynamical system as in subsection~\ref{subsection:NT-cp-endo}, and $C^*(P{_\vartheta\ltimes}G)\cong \NT(C^*(G), P,\alpha)$.
\end{prop}
\begin{proof} We only prove the last assertion as the rest is routine.
Proposition~\ref{prop:Nica covariance of various representations} provides natural isomorphisms between three different $C^*$-algebras associated to $(C^*(G), P,\alpha)$. We aim to show that $C^*(P{_\vartheta\ltimes}G)$ is isomorphic to the Nica-Toeplitz algebra $\NT(E_\alpha)$ associated to the product system $E_\alpha$ with multiplication defined in \eqref{multiplication-Ealpha}.

Let $i_{E_\alpha}$ be the universal Nica covariant representation of $E_\alpha$. For each $p\in P$, denote $i_p$ the restriction of $i_{E_\alpha}$ to $E_p$. We claim that $w_{(p,g)}:=i_p(\delta_g)$ for $(p,g)\in P{_\vartheta\ltimes}G$ is a Li-family in $\NT(E_\alpha)$. Let $(p,g), (q,h)\in P{_\vartheta\ltimes}G$. Then
$$
w_{(p,g)}w_{(q,h)}=i_p(\delta_g)i_q(\delta_h)=i_{pq}(\alpha_q(\delta_g)\delta_h)=w_{(pq,\vartheta_q(g)h)},
$$
and since $w_{(e,e)}=1$, we have a representation of $P{_\vartheta\ltimes}G$. Each $w_{(p,g)}$ is an isometry because $w_{(p,g)}^*w_{(p,g)}=i_p(\delta_g)^*i_p(\delta_g)=i_e(\langle \delta_g, \delta_g\rangle_p)=1$. Next we compute $w_{(p,g)}^*w_{(q,h)}$. Since $(p,g)(P{_\vartheta\ltimes}G)\cap (q,h)(P{_\vartheta\ltimes}G)=\emptyset$ if and only if $pP\cap qP=\emptyset$, in which case $i_p(\delta_g)^*i_q(\delta_h)=0$ by Nica covariance of $i_{E_\alpha}$. Thus $w_{(p,g)}^*w_{(q,h)}=0$ when $(p,g)(P{_\vartheta\ltimes}G)\cap (q,h)(P{_\vartheta\ltimes}G)=\emptyset$. Now assume the intersection is non-empty, and write $pp'=qq'=r$ for some $p',q',r$ in $P$. Pick a right LCM for $(p,g)$ and $(q,h)$, which we may assume of the form $(r,j)$ for $j\in G$, and write $
(p,g)(p',k)=(q,h)(q',l)=(r,j)$ where $k,l$ in $G$ are determined by $j=\vartheta_{q'}(h)l=\vartheta_{p'}(g)k$. Then $
w_{(p,g)}^*w_{(q,h)}=w_{(p',k)}w_{(q',l)}^*$, and this readily implies the Li-relation $e_Ie_J=e_{I\cap J}$ for $I=(p,g)(P{_\vartheta\ltimes}G)$ and $J=(q,h)(P{_\vartheta\ltimes}G)$. The remaining relations are easy to see, hence there is a $*$-homomorphism $C^*(P{_\vartheta\ltimes}G)\to \NT(E_\alpha)$ which sends a generating isometry $v_{(p,g)}$ to $w_{(p,g)}$.
Conversely, for $p\in P$ we let
\begin{equation}\label{eq:psi from v}
\psi_p(\delta_l)=v_{(p,l)}.
\end{equation}
We claim that $\psi_p:E_p\to C^*(P{_\vartheta\ltimes}G)$   give rise to a Nica covariant representation of $E_\alpha$.  However, this follows from routine calculations. For example, for $\delta_g\in C^*(G)$ and $x=\vartheta_p(k)l\in E_p$ we have
$$
\psi_p(\delta_g\cdot x)=\psi_p(\alpha_p(\delta_g)\delta_{\vartheta_p(k)l}=\psi_p(\delta_{\vartheta_p(gk)l})),
$$
which is $\psi_e(\delta_g)\psi_p(x)$. As a consequence, there is a $*$-homomorphism $\NT(E_\alpha)\to C^*(P{_\vartheta\ltimes}G)$ sending
$i_{E_\alpha}(x)$ for $x=\alpha_p(\delta_k)\delta_l\in E_p$ to $v_{(1,k)}v_{(p,l)}$. This is an inverse to the homomorphism $C^*(P{_\vartheta\ltimes}G)\to \NT(E_\alpha)$ obtained in the first half of the proof, hence the result follows.
\end{proof}

\begin{cor} [Uniqueness Theorem for $C^*(P{_\vartheta\ltimes}G)$]\label{uniqueness for right semidirect products}
Let $S=P{_\vartheta\ltimes}G$ where ${\vartheta}$ is a right action  of a right LCM semigroup $P$ on a group $G$.   Suppose that either $P^*=\{e\}$ or that the action of
$\{\alpha_{h}\}_{h\in P^{*op}}$ on $C^*(G)$ is aperiodic. Assume  that $\KK_\alpha$ is amenable. For a Nica covariant representation $(\pi, W)$ of $(C^*(G), P, \alpha)$, cf. Proposition \ref{prop:Nica covariance of various representations}, we have a  canonical surjective homomorphism
\begin{equation}\label{homomorphism to be isomorphism for right actions}
C^*(S)\mapsto \overline{\operatorname{span}}\{W_p\pi(a)W_q^*: a\in \alpha_p(C^*(G))C^*(G)\alpha_q(C^*(G)), p,q\in P\}
\end{equation}
which is an isomorphism if and only if for every finite family $q_1,\dots, q_n$ in $P\setminus P^*$, the representation $a\mapsto \pi(a)\Pi_{i=1}^n(1-W_{q_i}^{\phantom{*}}W_{q_i}^*)$ of $C^*(G)$ is faithful.
\end{cor}
\begin{proof}
Combine Proposition \ref{prop:right semidirect products} and Theorem~\ref{Uniqueness Theorem for  crossed products by endomorphisms}.
\end{proof}
The above result recovers the Laca-Raeburn uniqueness theorem \cite{LR} for Nica-Toeplitz algebras in the context of quasi-lattice ordered groups: just take $G=\{e\}$ and $P$ to be any weakly quasi-lattice ordered monoid.
It also improves  the case (1) in \cite[Theorem 4.3]{bls}, as we do not require (right) cancellativity.

\begin{ex}[Right wreath product]\label{Right wreath product}
Let $\Gamma$  be a   group  and $P$ a right LCM semigroup.
We form the right wreath product
$$
S:=P\wr \Gamma=P{_\vartheta\ltimes} \bigl(\prod_{p\in P}\Gamma \bigr)
$$
with the action given by left shifts $
\vartheta_p\bigl((\gamma_r)_{r\in P}\bigr):=(\gamma_{rp})_{r\in P}$
for $p\in P$ and $(\gamma_r)_{r\in P}\in G=\prod_{p\in P}\Gamma$. Clearly, $\vartheta_p\circ \vartheta_q=\vartheta_{qp}$ for all $p,q\in P$, as required for a right action.  For any $a\in C^*(\Gamma)$ and $q\in P$ we let $a \delta_q$ be the element of $C^*(G)=\prod_{p\in P} C^*(\Gamma)$ corresponding to the sequence with $a$ on $q$-th coordinate and zeros elsewhere.  The action $\alpha$ by endomorphisms of $C^*(G)$ is determined by $\alpha_p(a\delta_q)=a\sum_{rp=q}\delta_{r}$ (if $P$ is right cancellative this sum has at most one summand, if this sum is infinite we understand it as a series convergent  in  weak topopology). In particular, if $h\in P^*$, then $\alpha_h(a\delta_q)=a\delta_{qh^{-1}}$ and therefore $\|\alpha_h(a\delta_q)a\delta_q\|=0$ unless $qh^{-1}=q$ which is equivalent to $h=e$, by left cancellation. Since every non-zero hereditary subalgebra of $C^*(G)$  contains a non-zero element of the form $a\delta_q$, we conclude that the action $\{\alpha_h\}_{h\in P^{*op}}$  on $C^*(G)$ is always aperiodic. Assuming, for instance, that there is a controlled function from $P$ into an amenable group, we get $\K_\alpha$ amenable. Therefore, if $\{W_p\}_{p\in P}$ is a semigroup of isometries on a Hilbert space $H$ satisfying Nica relations \eqref{Nica covariance for semigroups} and $\pi:C^*(G)\to B(H)$ is a nondegenerate representation such that
$$
W_p^*\pi(a\delta_q))W_p=\sum_{rp=q}\pi(a\delta_{r}), \qquad \text{ for all }a\in C^*(\Gamma), q\in P,
$$
where the sum (if infinite) is  convergent in the strong operator topology, then by Corollary \ref{uniqueness for right semidirect products}  the surjective homomorphism in \eqref{homomorphism to be isomorphism for right actions} is an isomorphism if and only if for every  $q_1,\dots, q_n \in P\setminus P^*$ and $q\in P$, the representation $C^*(\Gamma)\ni a\mapsto \pi(a\delta_q )\Pi_{i=1}^n(1-W_{q_i}^{\phantom{*}}W_{q_i}^*)$ is faithful (then the corresponding representation of $C^*(G)$ is  faithful as well).

\end{ex}
\subsection{Semigroup $C^*$-algebras associated to left semidirect products $G\rtimes_\theta P$}\label{subsection:semigroupCstar-left}
Let  $(G, P, \theta)$ be an algebraic dynamical system.  The authors of \cite{bls2} associated to  $(G, P, \theta)$ a $C^*$-algebra $\mathcal{A}[G,P, \theta]$ universal for  a unitary representation of $G$ and a Nica  covariant isometric representation of $P$ subject to relations that model $\theta$ and the condition of preservation of order. In fact, there is a canonical isomorphism
$\mathcal{A}[G,P, \theta]\cong C^*(G\rtimes_\theta P)$, see \cite[Theorem 4.4]{bls2}.
It was also shown in \cite{bls2} that $\mathcal{A}[G,P, \theta]$ is naturally isomorphic to the Nica-Toeplitz algebra for a compactly-aligned product system $M$ over $P$ with fibers obtained as completions of $C^*(G)$. Specifically, denote by $\delta_g$ for $g\in G$ the generating unitaries in $C^*(G)$. We have two actions $\alpha:P\to \End(C^*(G))$ and  $L:P^{op}\to \Pos(C^*(G))$  given by $\alpha_p(\delta_g)=\delta_{\theta_p(g)}$ and
\begin{equation}\label{def:Lp}
 L_p(\delta_g)=\chi_{\theta_p(G)}(g)\delta_{\theta_p^{-1}(g)},\quad \text{ for }p\in P \text{ and }g\in G.
\end{equation}
 For every $p\in P$, $L_p$ is a transfer operator for $\alpha_p$.
The product system  constructed in \cite[Section 7]{bls2} coincides with the product system $M_L$ we defined in subsection \ref{subsection:NT-cp-transfer} (for general semigroup actions by transfer operators).
 By \cite[Proposition 7.8]{bls2}, $M_L$  is compactly-aligned. Summarizing we get:

\begin{prop}\label{prop:left semidirect products}
 Let $(G, P,\theta)$ be an algebraic dynamical system and consider the associated right LCM semigroup $G\rtimes_\theta P$.   Let $L$ be the action of $P^{op}$  by transfer operators  on $C^*(G)$ described in \eqref{def:Lp}. Then $(C^*(G),P, L)$ is a $C^*$-dynamical system as in subsection~\ref{subsection:NT-cp-transfer} and there are  natural isomorphisms
 \[
 \mathcal{A}[G,P,\theta]\cong C^*(G\rtimes_\theta P)\cong \NT(C^*(G),P, L).
 \]
\end{prop}
\begin{proof}
 Since $\mathcal{A}[G,P,\theta]\cong \NT(M)$ by \cite[Theorem 7.9]{bls2}, the assertions follow by an application of Proposition~\ref{Nica-Toeplitz algebras for transfer operators vs product systems}.
\end{proof}
\begin{cor}[Uniqueness Theorem for $C^*(G\rtimes_\theta P)$]\label{cor:Uniqueness Theorem for ...}
Let $S=G\rtimes_\theta P $ where $(G, P,\theta)$ is an algebraic dynamical system. Suppose that either  $P^*=\{e\}$ or  that  the action of
$\{\alpha_h\}_{h\in P^*}$ on $C^*(G)$ is aperiodic. Assume also that $\KK_{M_L}$  is amenable. Let $(\pi, W)$  be a Nica covariant representation of $(C^*(G), P, L)$, and let $Q_p$ be the projection onto the space $\overline{\pi(A)W_pH}$, $p\in P$. We have  a  surjective homomorphism
\begin{equation}\label{homomorphism to be isomorphism for left actions}
C^*(S)\mapsto \clsp\{\pi(a)W_pW_q^*\pi(b): a\in \alpha_p(A)A, b\in \alpha_q(A)A, p,q\in P\},
\end{equation}
which  is an isomorphism if  for every finite family $q_1,\dots, q_n$ in $P\setminus P^*$, the representation $a\mapsto \pi(a)\Pi_{i=1}^n(1-Q_{q_i})$ of $C^*(G)$ is faithful. If in addition $G/\theta_p(G)$ is finite for every $P$, then the latter condition is also necessary for the representation \eqref{homomorphism to be isomorphism for left actions} to be faithful.
\end{cor}
\begin{proof}
 $C^*(S)$ is isomorphic to the Nica-Toeplitz crossed product $\NT(C^*(G),P, L)$ due to Proposition \ref{prop:left semidirect products}. Thus by  Proposition~\ref{Nica-Toeplitz algebras for transfer operators vs product systems} combined with Lemma~\ref{aperiodicity for endomorphisms} we may apply Theorem~\ref{Uniqueness Theorem for  product systems I} to get the sufficiency claim of the isomorphism in \eqref{homomorphism to be isomorphism for left actions}.
For the last part, note   that
the left action of  $C^*(G)$ on each fiber of $M_L$ is by compacts if and only if $G/\theta_p(G)$ is finite for every $P$, see
\cite[Proposition 7.3]{bls2}.
\end{proof}

\begin{ex}[Left wreath product] Let $P$ and $\Gamma$ be as in Example \ref{Right wreath product}. Form the standard (left) wreath product
$$
S:=\Gamma\wr P=\bigl(\prod_{p\in P} \Gamma \bigr)\rtimes_\theta P,
$$
where $\theta$ acts by right shifts on $G:=\prod_{p\in P} \Gamma$, i.e. $\big(\theta_p\bigl((\gamma_q)_{q\in P}\big))_{r}=\chi_{pP}(r)\gamma_{p^{-1}r}$ for all $r\in P$. Then $(G, P, \theta)$ is an algebraic dynamical system, cf. \cite[Proposition 8.8]{bls2}. As in Example  \ref{Right wreath product}, for any $a\in C^*(\Gamma)$ and $q\in P$ we denote by $a \delta_q$ the corresponding element of $C^*(G)=\prod_{p\in P} C^*(\Gamma)$.
For $h\in P^*$ we have $\alpha_h(a\delta_q)=a\delta_{hq}$ and therefore $\|\alpha_h(a\delta_q)a\delta_q\|=0$ unless $hq=q$.
Using this, cf.  also Example \ref{Right wreath product}, we get that
$$
 \text{ the action $\{\alpha_h\}_{h\in P^*}$ on $C^*(G)$ is aperiodic } \,\,\Longleftrightarrow \,\,\,\, \left(\forall_{h\in P^*}\,\, \forall_{q\in P\setminus P^*}  \,\,\, hq=q\,\,  \Longrightarrow\,\, h=e \right).
$$
In particular, $\{\alpha_h\}_{h\in P^*}$  is aperiodic when $P^*=\{e\}$ or when $P$ is right cancellative.
If it is aperiodic we can  get a uniqueness criterion for $C^*(S)$ using Corollary \ref{cor:Uniqueness Theorem for ...}. If in addition $G/\theta_p(G)$ is finite for every $P$, then the action $L:P^{op}\to \Pos(C^*(G))$ given by \eqref{def:Lp} satisfies assumptions (A1), (A2), (A3) in subsection \ref{subsection:NT-cp-transfer-finitetype}. Therefore in this case  also Theorem \ref{Uniqueness Theorem for  crossed products by transfers} applies.

\end{ex}

\end{document}